\documentclass[a4paper,12pt]{article}
\usepackage[ansinew]{inputenc}
\usepackage{amsfonts}
\usepackage{latexsym}
\usepackage{amsmath}
\usepackage{amssymb}
\usepackage{subfigure}

\usepackage{tikz}
\usetikzlibrary{matrix}
\usetikzlibrary{arrows,positioning,shapes,fit,calc}
\pgfdeclarelayer{background}
\pgfsetlayers{background,main}

\newcommand{\veps}{\varepsilon}
\newcommand{\R}{\mathbb{R}}
\newcommand{\Q}{\mathbb{Q}}
\newcommand{\C}{\mathbb{C}}
\newcommand{\N}{\mathbb{N}}
\newcommand{\Z}{\mathbb{Z}}

\newcommand{\Ce}{\mathcal{C}}

\newcommand{\co}{\operatorname{co}}
\newcommand{\Log}{\operatorname{Log}}
\newcommand{\re}{\mathrm{Re}\,}
\newcommand{\im}{\mathrm{Im}\,}
\newcommand{\meas}{\operatorname{meas}\,}

\newtheorem{defin}{Definition}[section]
\newenvironment{definition}{\begin{defin}\rm}{\end{defin}}
\newtheorem{theorem}[defin]{Theorem}

\newtheorem{exa}[defin]{Example}
\newenvironment{example}{\begin{exa}\rm}{\end{exa}}
\newtheorem{lemma}[defin]{Lemma}
\newtheorem{corollary}[defin]{Corollary}
\newenvironment{proof}
{\noindent{\it Proof.}}{\hfill $\Box$\par\vspace{2.5mm}}
\newenvironment{remark}
{\par\vspace{2.5mm}\noindent{\bf Remark.}}{\par\vspace{2.5mm}}

\newtheorem{que}{Question}

\newtheorem{pro}{Problem}
\newenvironment{problem}{\begin{pro}\rm}{\end{pro}}

\numberwithin{equation}{section}

\title{\bf\Large Value distribution of exponential
polynomials and their role in the theories of complex differential
equations and oscillation theory}
\author{Janne Heittokangas\footnote{Corresponding author.}, Katsuya Ishizaki, Kazuya Tohge and Zhi-Tao Wen}
\date{}
\begin{document}
\maketitle

\begin{abstract}
An exponential polynomial is a finite linear sum of terms
$P(z)e^{Q(z)}$, where $P(z)$ and $Q(z)$ are polynomials. The early results
on the value distribution of exponential polynomials can be traced back to
Georg P\'olya's paper published in 1920, while the latest results have come
out in 2021. Despite of over a century of research work, many intriguing problems on value distribution of exponential polynomials still remain unsolved. The role of exponential polynomials and their quotients
in the theories of linear/non-linear differential equations, oscillation
theory and differential-difference equations will also be discussed.
Thirteen open problems are given
to motivate the readers for further research in these topics.


\end{abstract}

\renewcommand{\thefootnote}{}
\footnotetext[1]{Ishizaki was supported by JSPS KAKENHI Grant Number 20K03658.
Wen was supported by the National Natural Science Foundation of China (No.~11971288) and Shantou University SRFT (NTF18029).}

\tableofcontents

\section{Introduction}

An exponential polynomial $f$ is a finite linear sum of terms
$P(z)e^{Q(z)}$, where $P(z)$ and $Q(z)$ are polynomials.
If $\deg(Q)=1$ for each $Q(z)$ in the sum, then $f$ is called an
exponential sum.
Typical examples of exponential sums are the trigonometric functions
$\sin z= \frac{1}{2i}\left(e^{iz}-e^{-iz}\right)$ and $\cos z= \frac{1}{2}\left(e^{iz}+e^{-iz}\right)$, their quotient being $\tan z$. A polynomial can be considered as a
special case of an exponential polynomial, while transcendental exponential
polynomials are entire functions of finite positive integer order of growth.

In this expository paper, we will review basic value distribution results
for exponential sums and exponential polynomials, and discuss their role
in the theories of complex differential and complex difference equations as well
as in the complex oscillation theory.
The early results go back to P\'olya's paper \cite{Poly} published in
1920, while the most recent results are still fresh out of the press.
In other words, exponential polynomials have been of interest for over
a century.

The following topics will be of interest:
\begin{enumerate}
\item
We detect the portions of the complex plane $\C$, where the majority of
the zeros of exponential sums/polynomials are located. We also illustrate
that, in terms of zeros, there is a rather significant difference between these two classes of functions. We give a new direct proof for the known
fact that the zeros of a given exponential polynomial $f$ are regularly distributed, which in turn implies that $f$ is of completely regular growth. Note that zeros extend to $a$-points for any $a\in\C$ because $f$ is an exponential polynomial if and only if $g=f-a$ is.
\item
Given an exponential polynomial $f$, we find the asymptotic growth for the Nevanlinna characteristic $T(r,f)$
and for the integrated counting function $N(r,1/f)$ of zeros of $f$,
with sharp error terms. Recall that $T(r,f)=r/\pi$ for $f(z)=e^z$.
\item
A theorem of Ritt \cite{Ritt} from 1929 claims that if $f$ is a quotient of two
exponential sums with constant multipliers, and if $f$ is entire, then
$f$ is an exponential sum also or reduces to a constant. Assuming that the
exponential sums are allowed to have polynomial multipliers, or replacing
the exponential sums with higher-order exponential polynomials, we will illustrate that the outcome slightly changes. We will also discuss
the asymptotic growth of the Nevanlinna characteristic $T(r,f)$ and of the proximity functions $m(r,f)$ and $m(r,1/(f-c))$, where $f$ is a ratio of
exponential polynomials, and give updated error terms. These
results are new and need to be proved. The proofs expose the reader to standard manipulations
of exponential polynomials, which is also an essential part
of an expository paper. 
\item
It is known that an exponential polynomial is always a solution of
some linear differential equation with polynomial coefficients.
If one of the coefficients, say $A(z)$, is a transcendental exponential polynomial, then some of the solutions must be of infinite order.
However, exponential polynomial solutions $f$ may still exist. In order
for this to happen, the functions $f$ and $A(z)$ need to obey certain
duality property, which will be discussed.
\item
Suppose that $f\not\equiv 0$ satisfies a linear differential equation with entire coefficients some of
which are transcendental. If $f$ has the property that $\log T(r,f)=o(r)$, then $f$
is called \emph{subnormal}. Examples of such solutions typically involve exponential
polynomials. At times a change of variable method can be used to transform these equations
to differential equations with polynomial coefficients and polynomial solutions.
Hermite and Laguerre equations are famous examples of such equations. We will also
consider finite order periodic solutions, which are representable in terms of
Laurent polynomials.
\item
Suppose that the coefficients of a linear differential equation of
arbitrary order $n$ are entire and of completely regular growth.
Based on Petrenko's earlier work,  Gol'dberg and Ostrovskii asked
the following question in 1994:  \emph{If the differential
equation in question possesses a transcendental entire solution $f$ of
finite order of growth, then is $f$ of completely regular growth also?}
This question along with several known examples in terms of exponential polynomials will be discussed.
\item
Non-linear differential equations such as the Riccati equation, the
Malmquist type equations, and the Tumura-Clunie equations will be
discussed. Many of the known examples of solutions of these equations
are either exponential polynomials (entire case) or ratios of
exponential polynomials (meromorphic case).
\item
The 1982 seminal paper by Bank and Laine \cite{Bank-Laine} initiated the study of oscillation theory of linear differential equations in the complex plane. At the early stages the primary interest was on zeros of solutions of equations $f''+A(z)f=0$, where $A(z)$ is an entire function. The basic research questions are to find conditions on $A(z)$ guaranteeing that $f$ has no zeros, or
that $\lambda(f)\geq \rho(A)$, or that $\lambda(f)=\infty$. Here
$\lambda(f)$ and $\rho(A)$ denote the exponent of convergence of the
zeros of $f$ and the order of growth of $A(z)$, respectively. We will discuss
the properties of exponential polynomials as the coefficient
$A(z)$ inducing the aforementioned conclusions.
\item
As the last research topic of this survey, we will look into classifying
exponential polynomial solutions of complex difference and differential-difference equations.
\end{enumerate}

The research topics above obviously require some preliminary background. For the convenience of the reader, four appendices are included at the back
of the paper covering the fundamentals on the following topics: Nevanlinna theory, convex sets, Riemann zeta-function, and complex differential equations. Further, this survey includes
thirteen open problems to motivate
the reader for further research in these topics.

The basic building block for any transcendental exponential polynomial is of
course the exponential function. We therefore finish this introduction
with a discussion on a quote from Walter Rudin's famous book \emph{Real and Complex Analysis}:
	\begin{quote}
	\emph{This is undoubtedly the most important function in mathematics.
	It is defined, for every complex number $z$, by the formula
	$\exp(z)=\sum_{n=0}^\infty \frac{z^n}{n!}$.} -- W.~Rudin.
	\end{quote}
This summation formula shows that $\exp(z)$ is invariant under differentiation.
On the other hand, the definition does not directly imply that $\exp(z)$ has no zeros at all, or that it takes every
non-zero value with single multiplicity.
We need to know Picard's theorem or Nevanlinna's theory for this, or some knowledge of linear differential equations as $\exp(z)$ is a solution of $f'-f=0$.
Even more difficult it is to find the location of the zeros of an exponential polynomial, since these results answer nothing about it.
One needs powerful techniques for this purpose, and we will present some of them in this article.

\section{Exponential sums}\label{ES-sec}

An \emph{exponential sum} is an entire function of the form
    \begin{equation}\label{ES}
    f(z)=\sum_{j=1}^nP_j(z)e^{w_jz},
    \end{equation}
where the \emph{multipliers} $P_j(z)\not\equiv 0$ are polynomials and the \emph{frequencies} $w_j\in\C$ are pairwise distinct.
It is allowed that one of the constants $w_j$ vanishes.
The trigonometric functions $\sin z=\frac{1}{2i}\left(e^{iz}-e^{-iz}\right)$ and $\cos z=\frac{1}{2}\left(e^{iz}+e^{-iz}\right)$ are typical examples of such functions.
Finite partial sums
	$$
	\sum_{n\leq M}\frac{1}{n^z}=\sum_{n\leq M}e^{(-\log n)z}
	$$
of the \emph{Riemann zeta-function} $\zeta(z)$ are also exponential sums.
See Appendix~\ref{zeta-appendix} for discussions on the zeta-function.
If the multipliers $P_j$ are all complex constants, then the function $f$ in \eqref{ES} is the generic solution of a linear differential equation with constant coefficients determined by the characteristic equation
	$$
	(r-w_1)\cdots (r-w_n)=0.
	$$
The connection to differential equations will be discussed more thoroughly in Sections~\ref{ODE-sec} and \ref{EPO-sec} below.
In this section, however, we will focus on studying the zeros of exponential sums, that is, we focus on the value distribution of exponential sums.

\subsection{Exponential sums from an algebraic aspect}

Following D'Aquino, Macintyre and Terzo~\cite{DMT}, we look at exponential
sums of the form
    \begin{equation}\label{ESC}
    f(z)=\sum_{j=1}^n\alpha_j e^{w_jz}
    \end{equation}
from an algebraic aspect by collecting some basic definitions and
results. Here the frequencies $w_j\in \C$ are distinct and the multipliers
$\alpha_j\in \C$ are non-zero, unless $f(z)\equiv 0$.
Let us denote the collection of such functions by ${\cal E}$. Clearly
${\cal E}$ is a differential field under the usual addition and multiplication. Then the {\it units} in ${\cal E}$ are the products of non-zero constants and $e^{wz}$ for $w\in\C$ because they have multiplicative inverses in ${\cal E}$.

\begin{definition}
An element $f\in{\cal E}$ is called {\it irreducible}, if there are no non-units $g,h\in {\cal E}$ such that $f=gh$. If $f\in{\cal E}$ is not  irreducible, it is called {\it reducible}.
\end{definition}

\begin{definition}\label{simple-definition}
The support of $f\in{\cal E}$,
denoted by $\mathrm{supp} (f)$, is the $\Q$-space ($=$ vector space over
the field of rational numbers) generated by the
frequencies $w_1, \ldots, w_n$ of $f$. If $\dim \mathrm{supp}(f)=1$,
we say that $f$ is {\it simple}.
\end{definition}

A simple exponential sum in ${\cal E}$ is, up to a unit, a polynomial in $e^{wz}$ over $\C$ for some $w\in\C$. Thus, $\sin z=\frac{-i}{2}e^{-iz}(e^{2iz}-1)$ and $\cos z=\frac{1}{2}e^{-iz}(e^{2iz}+1)$ are typical simple exponential sums, while finite partial sums of $\zeta(z)$ are irreducible and not simple in general.

A simple exponential sum can be factorized, up to units, into a finite product of factors of the form $1-\beta e^{\mu z}$, where $\beta, \mu\in \C\setminus\{0\}$. If an exponential sum $f$ is factorized by $f(z)=\prod_{j=1}^N (1-\beta_j e^{\mu_jz})$, then the zeros of $f$ are at the points
	$$
	(1/\mu_j)\Log (1/\beta_j) +\frac{2\pi i}{\mu_j}k,\quad k\in\Z,
	$$
belonging on the straight lines through $(\overline{\mu}_j/|\mu_j|^2)\Log (\overline{\beta}_j/|\beta_j|^2)$ and parallel to vectors $i\overline{\mu}_j$ and thus orthogonal to vectors $\overline{\mu}_j$ for each $1\leq j\leq N$.

Since the zeros of $1-\beta e^{\mu z}$ with $\beta \mu\neq 0$ is the arithmetic progression of the form $z=d_0+dk$, $k\in\Z$, it can include infinitely many such subsequences of the form $z=d_0+\ell d k$, $k\in\Z$, as the zeros of $1-\beta^{1/\ell} e^{\frac{\mu}{\ell}z}$ for $\ell\in\N$.
To avoid this nesting, we mention factorization by a finite product in the above. Therefore, a simple exponential sum is not irreducible, but these concepts are not complementary to each other.

\begin{example}\label{simple-definition2}
There are reducible functions $f\in{\cal E}$ which are not simple. Indeed, if
$f(z)=\bigl(1-e^{z}\bigr)\bigl(1-e^{iz}\bigr)$, then $\dim \mathrm{supp}(f)=2$, so that $f\in{\cal E}$ is reducible but not simple.
\end{example}

If $f\in{\cal E}$ is simple, then it has a representation $f(z)=\prod_{j=1}^N (1-\beta_j e^{\mu_jz})$, and all zeros of $f$ are on $N$ parallel lines
that are orthogonal to the vector $\overline{\mu}_1$, for example.
If $f$ is not simple, but still has the product representation, at least two of those lines must intersect. The number of intersecting lines is given by $\dim \mathrm{supp}\bigl(\prod_{j=1}^N(1-\beta_je^{\mu_jz})\bigr)$.
Further note that if $g\in {\cal E}$ divides $f\in {\cal E}$ and if $f$ is simple, then $g$ is also simple, and when $g$ is simple, its positive integer power $g^m$ is also simple, since $\mathrm{supp}(g^m)=\mathrm{supp}(g)$.

We find that any factor of $f\in{\cal E}$ is either simple or irreducible, unless it is a unit in ${\cal E}$.
In fact, the following factorization theorem on ${\cal E}$ is given by
Ritt in 1927.

\begin{theorem} \textnormal{(Ritt \cite{Ritt2})}\label{Ritt-theorem}
 Every function
$$
 1 + \alpha_1 e^{w_1z} + \cdots + \alpha_n e^{w_n z} \in {\cal E}
$$
distinct from unity, can be expressed in one and in only one way as a product
$$
 (S_1S_2 \cdots S_s) (I_1I_2 \cdots I_i)
$$
in which $S_1, \ldots , S_s$ are simple elements in ${\cal E}$ such that the frequencies in any one of them have irrational ratios to the frequencies in any other,
that is, $\mathrm{supp}(S_t)\neq \mathrm{supp}(S_u)$ for $t\neq u$, and in which $I_1, \ldots, I_i$ are irreducible elements in ${\cal E}$.
\end{theorem}

Any $f\in{\cal E}$ of the form~\eqref{ESC} can be written in the above form uniquely up to order and multiplication by units.
Recall that zeros of $f$ from only one of the factors $S_t$ locates on a single line, while the location of the zeros of a factor $I_j$ could be much more complicated than that of $S_t$.
In Appendix~C, we will make an observation of the case when some special finite partial sums of $\zeta(z)$ can have zeros only on the imaginary axis.
As we note, generally those partial sums are irreducible, so we will rearrange the order of summation so that the resulted partial sums $f_N$ are of the form $\bigl(\prod_{j=1}^N(1-\beta_je^{\mu_jz})\bigr)$ with $\dim \mathrm{supp}(f_N)=1$ and $\beta_j,\mu_j\in\C$
for $j=1,2, \ldots, N$.  We call them partial sums of $\zeta(z)$ only because they tend
to $\zeta(z)$ locally uniformly in $\re (z)>1$ as $N\to\infty$.

In the next subsection we study the zero distribution of exponential
sums $f$ of the form~\eqref{ES}. This includes functions $f\in {\cal E}$ of the form~\eqref{ESC} regardless of $f$ being simple or irreducible.

\subsection{Zeros of exponential sums}\label{zeros-expsum-sec} 

We discuss the number and the location of zeros of exponential sums $f$ in \eqref{ES}.
The early results were obtained by P\'olya and Schwengeler in the 1920's.

Set $W=\{\overline{w}_1,\ldots,\overline{w}_n\}$. The \emph{convex hull} $\co(W)$ of $W$ is
defined as the intersection of all closed convex sets containing $W$. As $W$ consists of
finitely many points, $\co(W)$ is either a line segment or a convex polygon. Let $\Ce=C(\co(W))$ denote the circumference of $\co(W)$. A~deeper
discussion on convex sets can be found in Appendix~\ref{convex-appendix}.

We may suppose that the vertex points of $\co(W)$ are organized in the counterclockwise
direction from $1$ to $s\leq n$, and denote them again by $\overline{w}_1,\ldots ,\overline{w}_s$.
If $s=n$, then all points $\overline{w}_j$ are vertex points of $\co(W)$ to begin with.

Let $R_1=\{z\in\C\setminus\{0\}:\arg(z)=\theta_1^\bot\}$ denote the ray that intersects
orthogonally the line through the points $\overline{w}_{s},\overline{w}_1$. For $j=2,\ldots,s$,
let $R_j=\{z\in\C\setminus\{0\}:\arg(z)=\theta_j^\bot\}$ denote the ray that
intersects orthogonally the line through the points $\overline{w}_{j-1},\overline{w}_j$.
In other words, the rays $R_j$ are parallel to the outer normals of the sides of $\co(W)$. Thus we call $R_j$'s as the \emph{orthogonal rays} for $\co(W)$, see Figure~\ref{Orthogonal-rays-fig}.

We may suppose that the vertex points of $\co(W)$ are chosen so that
	\begin{equation}\label{orthogonal-rays-ordered}
	0\leq \theta_1^\bot<\cdots <\theta_s^\bot<2\pi
	\end{equation}
holds. Finally, let $S_j(\veps)$ denote a sector of opening $\veps$ around the ray $R_j$, where $\veps>0$ is chosen to be
so small that the sectors $S_j(\veps)$ do not overlap.

\begin{figure}
\begin{center}
\begin{tikzpicture}[x=0.6cm, y=0.6cm]
\draw[->,thick] (-1,0) -- (8,0);
\draw[->,thick] (0,-5) -- (0,6);
\draw[-] (0,0) -- (5,5);
\draw[-] (0,0) -- (5,-5);
\draw[-] (3,-1) -- (6,2);
\draw[-] (6,2) -- (3,5);
\draw[-] (3,5) -- (3,-1);
\draw[-,dashed] (1,-3) -- (2.9,-1.1);
\draw[fill=black] (3,-1) circle (1.4pt);
\draw[fill=black] (3,5) circle (1.4pt);
\draw[fill=black] (6,2) circle (1.4pt);
\draw (-2.1,0.5) node[anchor=north west] { $\theta_2^\bot$};
\draw (4.7,-3.9) node[anchor=north west] { $\theta_3^\bot$};
\draw (4.7,5) node[anchor=north west] { $\theta_1^\bot$};
\draw (1.9,5.5) node[anchor=north west] { $\overline{w}_1$};
\draw (3.1,-0.7) node[anchor=north west] { $\overline{w}_2$};
\draw (6.1,2.4) node[anchor=north west] { $\overline{w}_3$};
\end{tikzpicture}
\end{center}
\caption{Orthogonal rays for $\co(W)$.}\label{Orthogonal-rays-fig}
\end{figure}
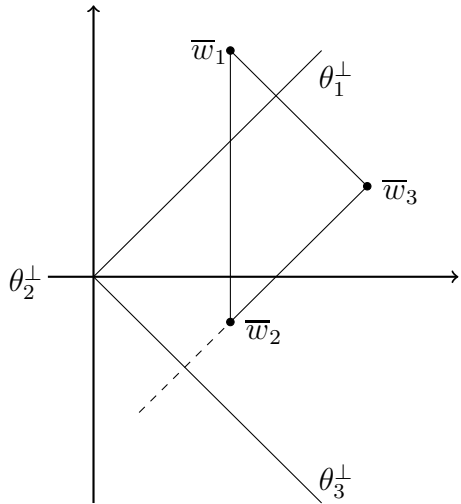

\begin{theorem} \textnormal{(P\'olya \cite{Poly})}\label{Polya-theorem}
Let $f$ be given by \eqref{ES}. Then all zeros
of $f$ lie in the union of the sectors $S_j(\veps)$, with finitely many possible exceptions. Moreover, the number $n(r,S_j)$
of zeros of $f$ in $S_j(\veps)\cap \{|z|\leq r\}$ has the asymptotic growth
	$$
	n(r,S_j)=L_jr/2\pi+O(\log r),
	$$
where $L_j=|\overline{w}_{j}-\overline{w}_{j-1}|$ for $2\leq j\leq s$ and
$L_1=|\overline{w}_{s}-\overline{w}_1|$.
\end{theorem}

As $\Ce=L_1+\cdots+L_s$, P\'olya's theorem implies that
    \begin{equation}\label{zeros-ES}
    n(r,1/f)=\Ce r/2\pi+O(\log r).
    \end{equation}
Hence the integrated counting function has the asymptotic growth rate
    \begin{equation}\label{integrated-Polya}
    N(r,1/f)= \Ce r/2\pi+O(\log^2 r).
    \end{equation}
Schwengeler \cite{Sch} has proved, under some assumptions, that \eqref{zeros-ES}
can be improved to
    \begin{equation}\label{zeros-ES-improved}
    n(r,1/f)=\Ce r/2\pi+O(1).
    \end{equation}
P\'olya claims this to be true in general, but offers no formal proof \cite{Poly2}.
A formal proof of \eqref{zeros-ES-improved} is due to Dickson \cite{Dick}. Obviously
this improves \eqref{integrated-Polya} into
	$$
	N(r,1/f)= \Ce r/2\pi+O(\log r).
	$$

If the multipliers $P_j$ in \eqref{ES} are constants and if certain conditions are valid,
the sectors $S_j(\veps)$ can be replaced with rectangular strips \cite{WH2,Moreno}. We take a quick look at an elementary example and discuss more about
this case in Section~\ref{zeros-in-strips} below.

\begin{example}\label{finitesum}
Let $f(z)=\prod_{j=1}^N(e^z-j)$. Then the orthogonal rays for $\co(W)=[0,N]$ are at angles $\theta_1^\bot=\pi/2$ and $\theta_2^\bot=3\pi/2$.
The zeros of $f$ are precisely at the points $z_{j,n}=\log j +2\pi ni$, where $n\in\Z$ and $j\in\{1, \ldots, N\}$. The zeros lay in the closed strip $0\leq \re (z) \leq \log N$, in fact, on $N$ rays parallel to the orthogonal rays, that is, to the imaginary axis.
It is easy to find that $L_1=|0-N|=N=L_2$ and $n(r,1/f)=N\dfrac{r}{\pi}+O(1)$.
For the function $g(z)=\prod_{j=1}^N(e^{-z}-1/j)$, the situation is similar to the one above except for $\co(W)=[-N, 0]$, since $g(z)/f(z)=\frac{(-1)^N}{N!}e^{-Nz}$ is a unit in ${\cal E}$.
\end{example}

\begin{remark}
One may wonder which exponential sums have all of their zeros only on a single line.
Those sums must be in the class ${\cal E}$, that is, the multipliers must be all constants as in~\eqref{ESC}.
As mentioned above, we will construct such an example in Appendix~\ref{zeta-appendix} by means of a certain modification of partial sums of the Riemann zeta-function.
In the wide class of exponential sums, where multipliers can be polynomials as in~\eqref{ES}, it is essential to consider sectors or logarithmic strips (defined below) around the orthogonal rays as the possible location of the zeros as in Theorem~\ref{Polya-theorem}.
\end{remark}

Schwengeler \cite{Sch, Stein} has proved that the sectors $S_j(\veps)$ in Theorem~\ref{Polya-theorem} can be replaced with logarithmic strips
(log-strips) of the form
    \begin{equation}\label{l-strip}
	\Lambda(\theta^\bot,c)
	=\left\{re^{i\theta}: r>1,\, |\theta-\theta^\bot|<c\frac{\log r}{r}\right\},
	\end{equation}
where $c>0$ is large enough but fixed, and where $\arg(z)=\theta^\bot$ represents an orthogonal ray.
To find the shape of $\Lambda(\theta^\bot,c)$, we choose $\theta^\bot=0$ for simplicity.
If $z=re^{i\theta}$ is on boundary of $\Lambda(0,c)$, we have
	\begin{equation}\label{l-boundary}
	re^{i\theta}=r\cos\left(c\frac{\log r}{r}\right)\pm ir\sin\left(c\frac{\log r}{r}\right)
    \sim r\pm ic\log r,\quad r\to\infty.
	\end{equation}
Hence $\Lambda(0,c)$ is asymptotically the domain between the curves $y=\pm c\log x$, $x\geq 1$.
The distance of these curves from the critical ray $\arg(z)=0$ tends to infinity as $x\to\infty$.
However, for any $\veps>0$ the log-strip $\Lambda(0,c)$ is essentially contained in the sector $S(\varepsilon)=\{z:|\arg(z)|<\veps\}$, see Figure~\ref{fig0}.
Thus Schwengeler's result improves Theorem~\ref{Polya-theorem}.

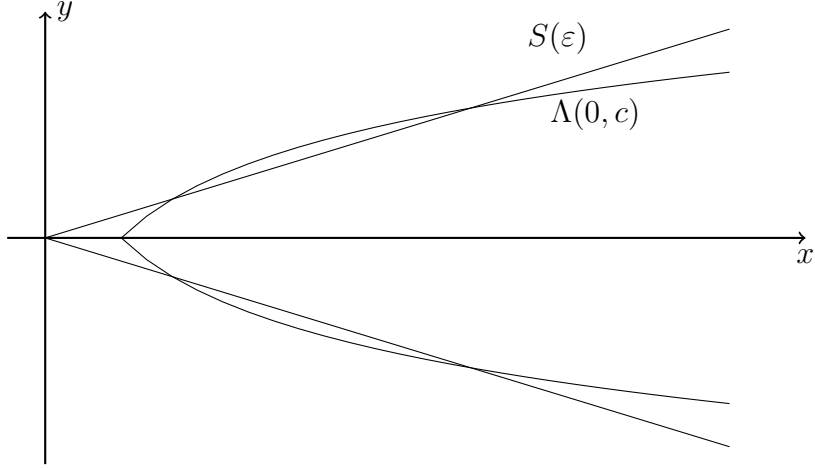
\begin{figure}
    \begin{center}
    \begin{tikzpicture}
    \draw[->,thick](-0.5,0)--(10,0)node[left,below]{$x$};
    \draw[->,thick](0,-3)--(0,3)node[right]{$y$};
    \draw[domain=1:9] plot(\x,{ln(\x)})node[below]{};
     \draw[domain=0:9] plot(\x,\x/3.25)node[above]{};
     \draw[domain=1:9] plot(\x,{-ln(\x)})node[above]{};
     \draw[domain=0:9] plot(\x,-\x/3.25)node[below]{};
     \draw (6.2,3) node[anchor=north west] { $S(\varepsilon)$};
      \draw (6.5,2) node[anchor=north west] { $\Lambda(0,c)$};
     \end{tikzpicture}
    \end{center}
	\begin{quote}
    \caption{The log-strip $\Lambda(0,c)$ is essentially contained in  the sector $S(\varepsilon)$.}\label{fig0}
    \end{quote}
    \end{figure}

\begin{example}
Let $f(z)=e^z-z$. Then the orthogonal rays for $\co(W)=[0,1]$ are at angles $\theta_1^\bot=\pi/2$ and $\theta_2^\bot=3\pi/2$.
By \cite[Example~21]{division}, the zeros of $f$ are precisely on the curves $y=\pm\sqrt{e^{2x}-x^2}$,
which tend asymptotically to the boundary curves $y=\pm e^x$ of the log-strips $\Lambda(\theta_j^\bot, 1)$, $j=1,2$,
in the right half-plane. Hence
the domains in \eqref{l-strip} have the correct order of magnitude.
\end{example}

The remainder of this section is devoted to discussing exponential sums that share infinitely many of their zeros.

\begin{theorem} \textnormal{(Ritt \cite{Ritt})}\label{Ritt-thm}
Assume that $g(z)$ and $h(z)$ are exponential sums in ${\cal E}$ of the form
	$$
	g(z)=\sum_{j=1}^ma_je^{\mu_jz}\quad\textrm{and}\quad
	h(z)=\sum_{j=1}^nb_je^{\nu_jz},
	$$
where $a_j,\mu_j,b_j,\nu_j\in\C$. If $f(z)=g(z)/h(z)$ is entire, then $f$ too is an exponential sum in ${\cal E}$ or reduces to a constant.
\end{theorem}

An alternative proof for Ritt's theorem is due to Lax \cite{Lax}. Rahman \cite{Rahman}
has proved a generalization of Ritt's theorem covering the case of polynomial
coefficients $a_j(z)$ and $b_j(z)$. In Rahman's generalization the multipliers
of the quotient $f$ are rational functions with removable singularities. For
example, if $g(z)=e^z-1$ and $h(z)=ze^z$, then $f(z)=g(z)/h(z)$ is an entire
exponential polynomial but with rational multipliers.

Ritt's theorem leads us to ask whether two exponential sums in ${\cal E}$ with infinitely many zeros in common are both multiples of some third exponential polynomial? According to the literature \cite{DMT,PT},
this question was posed by H.~L.~Montgomery at the 1974 \emph{J\'anos
Bolyai Society Colloquim on Number Theory}, and he attributed it to H.~S.~Shapiro, who in turn had stated it in the form of a conjecture:

\bigskip
\noindent
\textbf{Shapiro's conjecture.} (Shapiro \cite{Shapiro}) \emph{If two exponential sums in ${\cal E}$ have infinitely many zeros in common, they are both multiples of  some third (entire transcendental) exponential sum~in~${\cal E}$.}

\bigskip
Several ideas relevant to settling Shapiro's conjecture are surveyed in the above mentioned papers \cite{PT} by van der Poorten and Tijdeman and~\cite{DMT} by D'Aquino, Macintyre and Terzo.
Recalling Ritt's factorization theorem (Theorem~\ref{Ritt-theorem}) as well as the fact that a common zero of two elements in ${\cal E}$ is a zero of a common factor in ${\cal E}$ of them, one can concentrate on the following two cases when considering Shapiro's conjecture:

\begin{itemize}
\item[(1)] At least one of the two exponential sums is simple.

\item[(2)] Both of the two exponential sums are irreducible.
\end{itemize}

In Case~(1), the conjecture has been proved by van der Poorten and Tijdeman \cite{PT} in 1975, while in 2014 D'Aquino, Macintyre and Terzo \cite{DMT} have a positive solution for Case~(2), assuming Schanuel's conjecture concerning the transcendence degree of complex numbers and their exponents over $\Q$.
Hence, despite of the considerable interest, Shapiro's conjecture remains in doubt.


\begin{example}
Related to Shapiro's conjecture, let us consider exponential sums in
${\cal E}$. For example, if $1-e^{w z}$ and $1-e^{\mu z}$ with $w\mu\neq 0$ have a common zero which is not at the origin, then the ratio $w/\mu$ must be a rational number, say $p/q$. Then they have, in fact, infinitely many common zeros and a common factor $1-e^{\nu z}$ with $\nu=w/p=\mu/q$.

Next, consider the functions $f(z)=1+e^{wz}-2e^{2\mu z}$ and $g(z)=1+e^{\mu z}-2e^{2w z}$ when $w/\mu$ is irrational. Of course, neither $f$ nor $g$
is simple because $\dim \mathrm{supp}(f)=2=\dim \mathrm{supp}(g)$.
Further, both functions are irreducible in~${\cal E}$ since otherwise a possible factorization of them would have the form
$$
(1-ae^{\alpha z})(1-be^{\beta z})=1-ae^{\alpha z}-be^{\beta z}+abe^{(\alpha+\beta)z},
$$
so that there could happen a degeneration of either $\alpha-\beta$ or $\alpha+\beta$, which is impossible.
Can these irreducible elements $f, g \in{\cal E}$ have any common zero other than $z=0$?
The answer is no.
Indeed, the system of linear equations $1+x-2y=0$ and $1+y-2x=0$ has the unique solution $x=y=1$, so that a common zero of $f(z)$ and $g(z)$ is a common zero of $1-e^{wz}$ and $1-e^{\mu z}$.
As we have seen above, they have the unique common zero at the origin only,
and so do $f(z)$ and $g(z)$.

It is not difficult to construct two exponential sums
	$$
	\prod_{j=1}^N(1-a_je^{w_j z})\quad\textrm{and}\quad
	\prod_{j=1}^N(1-b_je^{\mu_j z})
	$$
having exactly $N$ common zeros by taking mutually distinct $w_j$ and $\mu_j$ such that $w_k/\mu_{\ell}\in \R\setminus\Q$ for any pair $(k,\ell)\in\{1, \ldots,  N\}^2$, and by choosing constants $a_j, b_j\in \C\setminus\{0\}$ suitably.
\end{example}

The previous example does not contribute anything to proving Shapiro's conjecture, but it may help the reader to imagine why the two exponential
sums need to have infinitely many zeros in common, and why a complete
proof of Case~(2) is still missing.


\subsection{Zeros in critical strips}\label{zeros-in-strips}

We consider the zeros of exponential sums in ${\cal E}$ of the form
    \begin{equation}\label{exp.eq}
    g(z)=F_1e^{\lambda_1z}+\cdots+F_me^{\lambda_mz},
    \end{equation}
where the frequencies $\lambda_j$ are on one line and $F_j\in\C\setminus\{0\}$.
By appealing to a rotation, we may suppose that $\lambda_j$'s are all real and ordered such that $\lambda_1<\lambda_2<\cdots<\lambda_m$.
Taking the zero-free term $F_1e^{\lambda_1z}$, that is, a unit in ${\cal E}$, as a common factor, we may write $g(z)=F_1e^{\lambda_1z}h(z)$,
where $h(z)$ is another exponential sum sharing its zeros with $g(z)$.
Therefore, without loss of generality, we may suppose that the exponential sum under consideration is in the normalized form
    \begin{equation}\label{normal.eq}
    f(z)=1+H_1e^{w_1z}+\cdots+H_ne^{w_nz}
    \end{equation}
to begin with, where $0<w_1<\cdots<w_m$, $H_j\in\C\setminus\{0\}$, and $n\leq m-1$.
We suppose that $n\geq 2$, since in the case $n=1$ the zeros of $f$ are all on one vertical line.

The point of the normalization in \eqref{normal.eq} is that the zeros of $f$ are in vertical strips $S$, as we will soon see.
If, in addition, the frequencies $w_1, \ldots, w_n$ are linearly independent over the rational numbers, then the vertical projections of the zeros of $f$ on $\R$ form dense subsets of the intervals induced by the intersections of the strips $S$ with $\R$.
A precise statement is Theorem~\ref{MSV-thm} below, which is based on the findings of Moreno \cite{Moreno}.

\begin{theorem}\textnormal{(Mora-Sepulcre-Vidal \cite{SG})}\label{MSV-thm}
Let $f$ be an exponential sum of the normalized form \eqref{normal.eq}, whose frequencies $0<w_1<\cdots<w_n$ are linearly independent over the rational numbers. Then an open interval $(\sigma_0, \sigma_1)$ is contained in
	$$
    R_f:=\overline{\{\textnormal{Re}\,(z): f(z)=0\}}
    $$
if and only if the $n+1$ inequalities
    \begin{equation*}
    1\leq\sum_{j=1}^n|H_j|e^{w_j\sigma},
    \quad |H_k|e^{w_k\sigma}\leq 1+\sum_{j=1, j\neq k}^n|H_j|e^{w_j\sigma},\quad k=1,\ldots,n,
    \end{equation*}
are satisfied for any $\sigma\in(\sigma_0, \sigma_1)$.
\end{theorem}

Note that the exponential sum $f$ in Theorem~\ref{MSV-thm} is not simple, see Definition~\ref{simple-definition}. In particular, the zeros of simple exponential polynomials are on finitely many parallel
lines, see the discussion after Example~\ref{simple-definition2}.

The inequalities in Theorem~\ref{MSV-thm} are of further interest.
Indeed, suppose that $z=x+iy$ satisfies
    \begin{equation}\label{mineq.eq}
    |H_k|e^{w_kx}> \sum_{j\neq k}^n|H_j|e^{w_jx},\quad k=0,1,\ldots,n,
    \end{equation}
where $H_0=1$ and $w_0=0$. Then
    \begin{equation}\label{zerofree.eq}
    |f(z)|\geq |H_k|e^{w_kx}-\sum_{j\neq k}^n|H_j|e^{w_jx}>0,
    \end{equation}
and we find that $f$ has no zeros in the region
    $$
    G:=\{ z=x+iy: x~\text{satisfies}~\eqref{mineq.eq},\, y\in\R\}.
    $$
We call $G$ as a \emph{zero-free region} of $f$.
The zero-free region of $f$ to the extreme left is determined by the
term~$1$, and the one to the extreme right
by $H_ne^{w_nz}$.
They both have an unbounded width, while the rest of the zero-free regions, if any, have a bounded width. Note that $f$ has $n+1$ terms, and hence $f$ has at most $n+1$ zero-free regions.

\begin{example}
\textnormal{(\cite{WH2})}
Let $f(z)=1+e^z+e^{2z}$. Then the two sets
	$$
	\left\{z: \textnormal{Re}\,(z)<\log\frac{\sqrt{5}-1}{2}\right\}
	\quad\textnormal{and}\quad
	\left\{z: \textnormal{Re}\,(z)>\log\frac{\sqrt{5}+1}{2}\right\}
	$$
are the only zero-free regions of $f$ even though $f$ has three terms.
Note that $R_f=\{0\}$, since $f(z)=0$ if and only if $z=\pm \frac{2}{3}\pi i +2m\pi i$, $m\in\mathbb{Z}$.
The system of inequalities in Theorem~\ref{MSV-thm} are satisfied for any $\sigma\in\mathbb{R}$ with
$$
e^{\sigma}\in \left[\frac{\sqrt{5}-1}{2}, \frac{\sqrt{5}+1}{2}\right], \quad \text{or} \quad
\sigma\in \left[\log \frac{\sqrt{5}-1}{2}, \log \frac{\sqrt{5}+1}{2}\right](\ni 0).
$$
\end{example}

In each of the zero-free regions, precisely one inequality in \eqref{mineq.eq} holds.
The exponential term $H_je^{w_jz}$, whose modulus is strictly greater than the sum of the others, is called the \emph{dominating term} in this region.
The boundary $L_k$ of a zero-free region of $f$ with $H_ke^{w_kz}$ being the dominant exponential term is determined by
    \begin{equation*}
    L_k=\left\{z=x+iy: |H_k|e^{w_kx}= \sum_{j\neq k}^n|H_j|e^{w_jx},\, y\in\R\right\}.
    \end{equation*}
The area between two consecutive zero-free regions of $f$
with dominating terms $H_je^{w_jz}$ and $H_ke^{w_kz}$ for $j\neq k$ is called a \emph{critical strip} of $f$, and is denoted by $\Lambda(j,k)$.
Each critical strip is a closed set containing its vertical boundary lines.
All zeros of $f$ lie in these critical strips, with no exceptions.
It is possible that the zeros of $f$ lie on the boundary $L_k$ of a zero-free region, as is illustrated in the following example.

\begin{example}\label{critical-strips-width}
\textnormal{(\cite{WH2})}
The zeros of the exponential sum
	$$
	f(z)=6-5e^z+e^{2z}=(e^z-2)(e^z-3)
	$$
lie on two lines $\textnormal{Re}\,(z)=\log 2$ and $\textnormal{Re}\,(z)=\log 3$.
The zero-free regions of $f$ are $\{z: \textnormal{Re}\,(z)>\log 6\}$, $\{z: \log 2<\textnormal{Re}\,(z)<\log 3\}$ and
$\{z: \textnormal{Re}\,(z)<0\}$. Thus the critical strips of $f$ are $\Lambda(1,2):=\{z:\log 3\leq \textnormal{Re}\,(z)\leq\log 6\}$
and $\Lambda(0,1):=\{z: 0\leq \textnormal{Re}\,(z)\leq\log 2\}$.
The lines $\textnormal{Re}\,(z)=\log 2$ and $\textnormal{Re}\,(z)=\log 3$ are boundary lines of $\Lambda(0,1)$
and $\Lambda(1,2)$, respectively, see Figure~\ref{domain.png}.
    \begin{figure}[ht]\label{zero}
     \begin{center}
     \begin{tikzpicture}
    \draw[->,thick](2,0)--(11,0)node[left,below]{$x$};
    \draw[-](4,-3.5)--(4,3);
    \draw (3.75,3.6) node[anchor=north west] { $\Lambda(0,1)$};
    \draw[-](5,3)--(5,-3.5);
    \draw[-](8,-3.5)--(8,3);
    \draw (7.75,3.6) node[anchor=north west] { $\Lambda(1,2)$};
    \draw[-](9,3)--(9,-3.5);
    \draw[-,dashed](9,2.5)--(10,2);
    \draw[-,dashed](9,1.5)--(10,1);
    \draw[-,dashed](9,0.5)--(10,0);
    \draw[-,dashed](9,-0.5)--(10,-1);
    \draw[-,dashed](9,-1.5)--(10,-2);
    \draw[-,dashed](9,-2.5)--(10,-3);
    \draw[-,dashed](3,2.5)--(4,2);
    \draw[-,dashed](3,1.5)--(4,1);
    \draw[-,dashed](3,0.5)--(4,0);
    \draw[-,dashed](3,-0.5)--(4,-1);
    \draw[-,dashed](3,-1.5)--(4,-2);
    \draw[-,dashed](3,-2.5)--(4,-3);
    \draw[-,dashed](5,2.5)--(8,2);
    \draw[-,dashed](5,1.5)--(8,1);
    \draw[-,dashed](5,0.5)--(8,0);
    \draw[-,dashed](5,-0.5)--(8,-1);
    \draw[-,dashed](5,-1.5)--(8,-2);
    \draw[-,dashed](5,-2.5)--(8,-3);
    \draw[-,dashed](10,4)--(10.5,3.8);
    \draw[-,dashed](10,3.8)--(10.5,3.6)node[right]{zero-free region};
    \draw[-,dashed](10,3.6)--(10.5,3.4);
    \end{tikzpicture}
     \end{center}
    \caption{Zero-free regions and critical strips.}\label{domain.png}
    \end{figure}
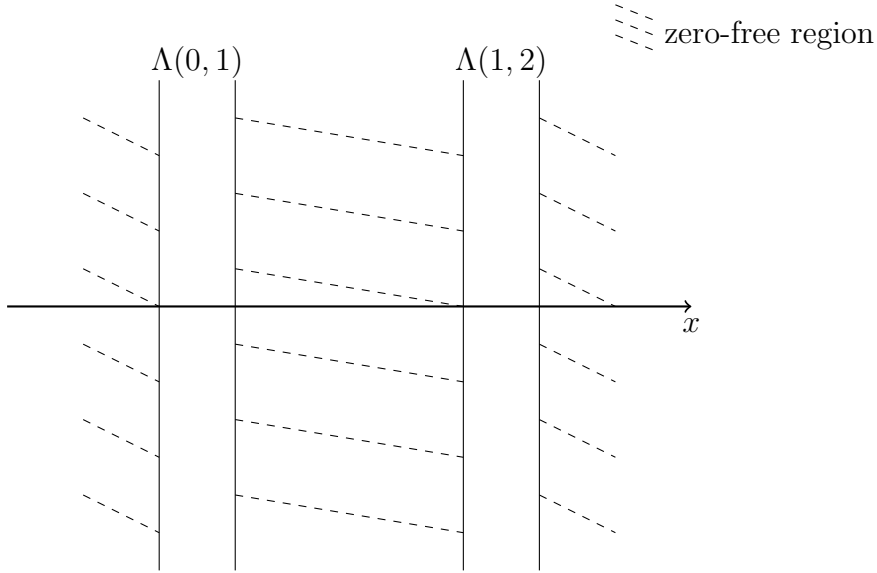
\end{example}

Langer \cite{Langer} considered the zeros of exponential sums $f$ of the form \eqref{normal.eq} in rectangles.
The vertical sides of these rectangles coincide with the boundary lines of one large vertical strip that contains all zeros of $f$.
It is proved in \cite{Langer} that the number $n(r)$ of zeros of $f$ in such a rectangle of height $r$ is subject to the bound
    \begin{equation*}
    \bigg|n(r)-\frac{w_n}{2\pi}r\bigg|\leq n.
    \end{equation*}
The following result reveals the number of zeros of $f$ in each individual critical strip.

\begin{theorem}\label{strips-theorem}
\textnormal{(\cite{WH2})}
Let $f$ be an exponential sum of the form \eqref{normal.eq}, where the multipliers are constants and the frequencies satisfy $0<w_1<\cdots<w_n$.
Then all zeros of $f$ are in finitely many critical strips $\Lambda(j,k)$.
Moreover, let $R$ be any rectangle cut from a critical strip $\Lambda(j,k)$ by two horizontal lines $\textnormal{Im}\,(z)=y_1$ and $\textnormal{Im}\,(z)\, =y_2$ with $y_2-y_1=r>0$.
Then the number $n(r,\Lambda_{jk})$ of zeros of $f$ inside $R$ satisfies
    \begin{equation}\label{n.eq}
    n(r,\Lambda_{jk})=\frac{|w_j-w_k|}{2\pi}r+O(1).
    \end{equation}
\end{theorem}

The asymptotic equation \eqref{n.eq} can be used to show that the zeros of $f$ are asymptotically separated.
Indeed, if $R_1$ and $R_2$ are two matching rectangles of heights $r$ and $r+1$, respectively, then the asymptotic number of zeros
in the rectangle $R_2\setminus R_1$ of height one is asymptotically bounded.

Theorem~\ref{MSV-thm} leads us to asking whether some vertical lines in the critical strips of $f$ could avoid the zeros of $f$?
The following example shows that in fact almost every vertical line has this property in a strong sense.

\begin{example}
\textnormal{(\cite{WH2})}
Let $f$ be as in Theorem~\ref{strips-theorem}, and let $\{z_n\}$ denote the sequence of zeros of $f$ listed according to multiplicities and
ordered with respect to increasing modulus.
Let $\{D_n\}$ be the collection of Euclidean discs
    $$
    D_n:=\left\{z: |z-z_n|<r_n\right\},
    \quad r_n:=(1+|z_n|)^{-1}\log^{-2}(e+|z_n|).
    $$
Note that $\sum_n r_n<\infty$ because the exponent
of convergence of $\{z_n\}$ is equal to one.
It is proved in \cite{WH2} that the set $C\subset\R$ of values~$c$ for which the vertical line $\textnormal{Re}\,(z)=c$ meets infinitely
many discs $D_n$ has linear measure zero.
In other words, almost every vertical line meets at most finitely many discs $D_n$ around the zeros of $f$.
This does not violate the conclusion of Theorem~\ref{MSV-thm} because the radii of these discs tend to zero.
\end{example}

The left and the right half-planes are the zero-free regions of $g(z)=e^z-1$,
while the only critical strip of $g$ is just one vertical line, namely the imaginary axis.
Meanwhile, the critical strips of $f$ in Example~\ref{critical-strips-width} have positive width,
and all the zeros of $f$ lie on some of the boundary lines of the critical strips. Common to these functions is that their frequencies are linearly dependent over the rational numbers, as is the case with
	$$
	h(z)=(e^{(\sqrt{2}+1)z}-1)(e^z-2).
	$$
However, it can be verified by means of computer software that
the vertical boundary lines of the sole critical strip $\Lambda$ of $h$ are
$x\approx -0.316402$ and $x\approx 0.914126$. Meanwhile,
all zeros of $h$ are  lying
on the vertical lines $x=0$ and $x=\log 2\approx 0.693147$ both of which are in the interior of $\Lambda$. This leads us to the following question.

\begin{problem}
Let $f$ be a normalized exponential sum of the form \eqref{normal.eq} having a critical
strip $\Lambda$ of positive width, and suppose that the frequencies $w_1,w_2,\ldots,w_m$,
$m\geq 2$, of $f$ are linearly independent over the rational numbers.
Is it true that every vertical line in the interior
of $\Lambda$ meets at most finitely many zeros of $f$?
\end{problem}

\section{Exponential polynomials}\label{EP-sec}

An \emph{exponential polynomial} is an entire function $f$ of the form
    \begin{equation}\label{EP}
    f(z)=\sum_{j=1}^mP_j(z)e^{Q_j(z)},
    \end{equation}
where $P_j(z)$, $Q_j(z)$ are polynomials for $1 \leq j \leq m$, that is,
$f$ is a linear combination of exponential functions $e^{Q_j(z)}$ over the ring $\C[z]$ of polynomials.
We assume that the polynomials $Q_j(z)$ are pairwise different and normalized such that $Q_j(0)=0$.
In this case the terms $P_j(z)e^{Q_j(z)}$ in~\eqref{EP} are linearly independent over the field of rational functions, and the order of  $f$ is defined by the maximum degree of the $Q_j(z)$.
Observe that polynomials and exponential sums are special cases of exponential polynomials.
In particular, an exponential polynomial of order zero is an ordinary polynomial in~$z$, and an exponential polynomial of order one is an exponential sum. In this section we will focus on the zero distribution of exponential polynomials and on the asymptotic
growth of $T(r,f)$, $N(r,1/f)$, $m(r,f)$ and $m(r,1/(f-c))$, where $f$ is either an exponential polynomial or a quotient of two exponential polynomials.

\subsection{Asymptotic growth of $T(r,f)$ and $N(r,1/f)$}\label{TN-subsection}

If $f$ in \eqref{EP} is transcendental, then it can be written in the normalized form
    \begin{equation}\label{normalized-f}
    f(z)=F_0(z)+F_1(z)e^{w_1z^q}+\cdots+F_n(z)e^{w_nz^q},
    \end{equation}
where $q=\max\{\deg (Q_j)\}\geq 1$ is the order of $f$, the frequencies $w_j$ of $f$ are pairwise distinct non-zero constants, the multipliers $F_j(z)$ are exponential polynomials of order $\leq q-1$ such that $F_j(z)\not\equiv 0$ for $1\leq j\leq n$, and $n\leq m$.
Those terms $P_j(z)e^{Q_j(z)}$ with $\deg(Q_j)\leq q-1$, if any, are included in the term $F_0(z)$.

For convenience, we set $w_0=0$. As in Section~\ref{ES-sec}, we denote  $W=\{\overline{w}_1,\ldots,\overline{w}_n\}$
and $\Ce=C(\co(W))$. In addition, we set $W_0=W\cup\{0\}$ and $\Ce_0=C(\co(W_0))$.

The asymptotic growth for the Nevanlinna characteristic function of a given transcendental exponential polynomial were found by Steinmetz in \cite[Satz~1]{Stein} with $o(r^q)$ as the error term.
We state the following refinement of \cite[Satz~1]{Stein} that has
improved error terms along with multipliers of sub-maximal growth.

\begin{theorem} \textnormal{(\cite{WH})} \label{Stein-thm1}
Let $f$ be given by \eqref{normalized-f}, where the multipliers
$F_j(z)$ are exponential polynomials of order $\rho(F_j)\leq q-p$ for some integer
$1\leq p\leq q$. Then
    \begin{equation}\label{characteristic1}
    T(r,f)=\Ce_0r^q/2\pi+O(r^{q-p}+\log r).
    \end{equation}
If $F_0(z)\not\equiv 0$, then
    \begin{equation}\label{proximity0}
    m(r,1/f)=O(r^{q-p}+\log r),
    \end{equation}
while if $F_0(z)\equiv 0$, then
    \begin{equation}\label{counting1}
    N(r,1/f)=\Ce r^q/2\pi+O(r^{q-p}+\log r).
    \end{equation}
If $F_0(z)\not\equiv 0$ is a polynomial, then \eqref{proximity0} can be replaced with
   \begin{equation}\label{proximity2}
    m(r,1/f)=O(\log r).
    \end{equation}
\end{theorem}

\begin{example}
\textnormal{(Steinmetz \cite{Stein})}
If $f(z)=(1-3e^{iz})e^{z^2}-ze^{-iz^2}$, then
	\begin{eqnarray*}
	T(r,f) &=& \frac{2+\sqrt{2}}{2\pi} r^2+o(r^2),\\
	N(r,1/f) &=& \frac{\sqrt{2}}{\pi} r^2+o(r^2),\\
	m(r,1/(f-c)) &=& O(\log r),\quad c\neq 0.
	\end{eqnarray*}
Thus the Nevanlinna deficiency (see Appendix~\ref{Nevanlinna-appendix})
for the zeros of $f$ is $\delta(0,f)=3-2\sqrt{2}$.
\end{example}

From \eqref{characteristic1} and \eqref{counting1}, we see that the value zero is a Nevanlinna deficient value for $f$ precisely when $\Ce<\Ce_0$. In addition, we conclude that $f$ can have at most one finite deficient value. The inequality $\Ce<\Ce_0$ comes down to the question whether the origin is included in $\co(W)$ or not. The situation is visualized
in Figure~\ref{hull} below.

	\begin{figure}[h]
     \begin{center}
     \begin{tikzpicture}

     \draw[thick,->,thick] (-0.25,0) -- (4,0);
     \draw[thick,->,thick] (0,-0.25) -- (0,4);
     \draw[fill=black] (0,0) circle (1.4pt);
	 \draw[fill=black] (3,1) circle (1.4pt);	
	 \draw[fill=black] (3,2) circle (1.4pt);
	 \draw[fill=black] (1,3) circle (1.4pt);
	 \draw[fill=black] (1.5,1.5) circle (1.4pt);
	 \draw[-] (1.5,1.5) -- (3,1);
	 \draw[-] (3,1) -- (3,2);
	 \draw[-] (3,2) -- (1,3);
	 \draw[-] (1,3) -- (1.5,1.5);
	
	 \draw[thick,->] (5.75,0) -- (10,0);
     \draw[thick,->] (6,-0.25) -- (6,4);
     \draw[fill=black] (6,0) circle (1.4pt);
	 \draw[fill=black] (9,1) circle (1.4pt);	
	 \draw[fill=black] (9,2) circle (1.4pt);
	 \draw[fill=black] (7,3) circle (1.4pt);
	 \draw[fill=black] (7.5,1.5) circle (1.4pt);
     \draw[-] (6,0) -- (9,1);
	 \draw[-] (9,1) -- (9,2);
	 \draw[-] (9,2) -- (7,3);
	 \draw[-] (7,3) -- (6,0);
	
	 \end{tikzpicture}
     \end{center}
    \caption{The circumferences of convex hulls $\co(W)$ and $\co(W_0)$.}\label{hull}
    \end{figure}
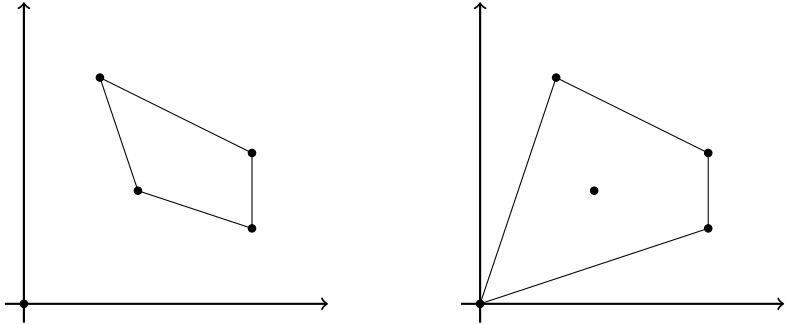

As an easy consequence of \eqref{counting1}, the non-integrated counting function of the
zeros of $f$ has the asymptotic growth
	\begin{equation}\label{n-asymptotic}
	n(r,1/f)=q\Ce r^q/2\pi+o(r^q),
	\end{equation}
see \cite[Corollary~3.3]{WH} for more details. The error term $o(r^q)$
in \eqref{n-asymptotic} can be improved, see the discussion in
Section~\ref{zero-distribution} below.

\begin{remark}
The constant $\Ce_0$ and the set $W_0$ are needed in Theorem~\ref{Stein-thm1} only. From this
point onwards it is convenient to use the same notation as in Section~\ref{ES-sec}. In particular,
the set $W=\{\overline{w}_1,\ldots,\overline{w}_n\}$ contains the conjugated frequencies
of a normalized form exponential polynomial
    \begin{equation}\label{normalized-f2}
    f(z)=F_1(z)e^{w_1z^q}+\cdots+F_n(z)e^{w_nz^q},
    \end{equation}
where the constants $w_j$ are pairwise distinct and one of them may or may not be zero. Moreover,
the points $\overline{w}_1,\ldots ,\overline{w}_s (s\leq n)$ are the vertices of $\co(W)$, and
the rays $\arg(z)=\theta^\bot_j$, $j=1,\ldots ,s$, correspond to the orthogonal
rays for $\co(W)$, and satisfy \eqref{orthogonal-rays-ordered}.
\end{remark}

\subsection{Zeros of exponential polynomials}\label{zero-distribution}

The zero distribution of exponential polynomials is more complicated than that of exponential sums.
For example, if $f(z)=e^{2\pi iz^2}-1$, then the orthogonal rays for the convex hull $\co (W)=[0,-2\pi i]$ are the positive and the negative real axes. However, the zeros of $f$, namely $z_n=\pm\sqrt{n}$ for $n\in\Z$, are precisely on all four coordinate half-axes.
Thus orthogonal rays alone no longer suffice.

From \cite[Lemma~3.2]{division}, we find that each orthogonal ray $\arg(z)=\theta^\bot$ induces $q$ \emph{critical rays} $\arg(z)=\theta_k^*$, where
    \begin{equation}\label{critical-directions}
    \theta_k^*=(\theta^\bot+2k\pi)/q,\quad k=0,\ldots,q-1.
    \end{equation}
Some orthogonal rays and critical rays may coincide.
For example, the critical rays related to the function
$f(z)=e^{2\pi iz^2}-1$ discussed above are at angles $0,\pi/2,\pi,3\pi/2$. In general, the number
of critical rays is $\leq s q$, where $s$ is the number of
vertex points of $\co(W)$ and $q$ is the order of $f$. The critical rays reduce to orthogonal rays in the special case $q=1$. In terms of zero distribution, the critical rays take the role of orthogonal rays, as we will see next. In Appendix~\ref{convex-appendix} below we will show that
the critical rays can be found in terms of the supporting function and the indicator function.

Given $f$ and $W$, we define the modified logarithmic strips
	$$
	\Lambda_p(\theta^*,c)=\left\{z=re^{i\theta} : r>1,\ |\arg(z)-\theta^*|<c\frac{\log r}{r^p} \right\},
	$$
where $c>0$ is large enough but fixed, $p\leq q$ is an integer and $\arg(z)=\theta^*$
represents a critical ray. Similarly as
in \eqref{l-boundary}, we find that $\Lambda_p(0,c)$ is asymptotically the domain between the
curves $y=\pm c\log x/x^{p-1}$, $x\geq 1$. Differing from the situation
in \eqref{l-strip}, the domains $\Lambda_p(\theta^*,c)$
curve asymptotically towards the critical rays  $\arg(z)=\theta^*$ when $p\geq 2$, see Figure~\ref{fig1}.

\begin{figure}[ht]
    \begin{center}
    \begin{tikzpicture}
    \draw[->,thick](-0.5,0)--(10,0)node[left,below]{$x$};
    \draw[->,thick](0,-3)--(0,3)node[right]{$y$};
    \draw[domain=1:9] plot(\x,{ln(\x)})node[below]{$y=\log x$};
     \draw[domain=1:9] plot(\x,{(ln(\x))/\x})node[above]{$y=\frac{\log x}{x}$};
     \draw[domain=1:9] plot(\x,{-ln(\x)})node[above]{$y=-\log x$};
     \draw[domain=1:9] plot(\x,{-(ln(\x))/\x})node[below]{$y=-\frac{\log x}{x}$};
     \draw (0.8,0) node[anchor=north west] {$1$};
    \end{tikzpicture}
    \end{center}
	\begin{quote}
    \caption{The domains $\Lambda_1(0,1)$ and $\Lambda_2(0,1)$.}\label{fig1}
    \end{quote}
    \end{figure}
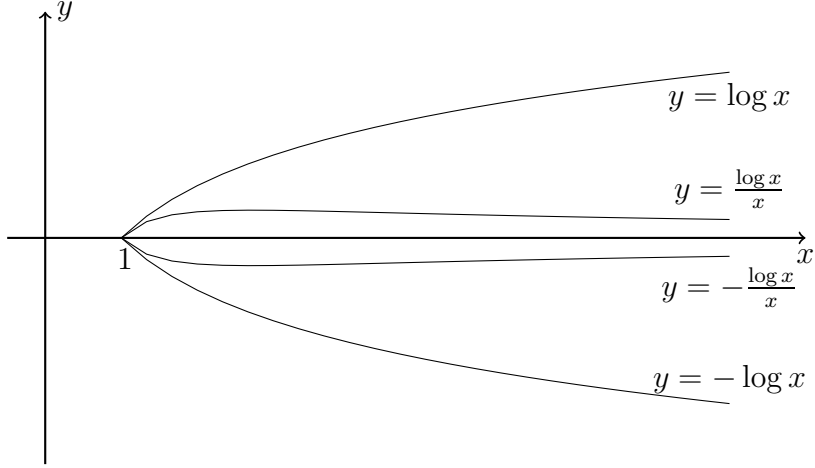

In the following result $N_{\Lambda_p}(r)$ denotes the integrated counting function
of the zeros of $f$ in $|z|\leq r$ lying outside of the finite union of domains $\Lambda_p(\theta^*,c)$.

\begin{theorem} \textnormal{(\cite{division})} \label{zeros-thm}
Let $f$ be an exponential polynomial of the form \eqref{normalized-f2} for $q\geq 1$, where the multipliers $F_j(z)$
are exponential polynomials of order $\rho(F_j)\leq q-p$ for some integer $1\leq p\leq q$.
Then
	$$
    N_{\Lambda_p}(r)=\left\{\begin{array}{rl}
    O\left(r^{q-p}\right),\ & \textnormal{if}\ q-p\geq 1,\\
    O\left(\log r\right),\ & \textnormal{if}\ p=q.
    \end{array}\right.
    $$
\end{theorem}

The case $p=1$ in Theorem~\ref{zeros-thm} is due to Steinmetz \cite[Satz~4]{Stein}.
Roughly speaking, the result says that the majority of zeros of $f$ are in the $\Lambda_p$-domains around the critical rays. The assumption $\rho(F_j)\leq q-p$ for $p=1$ is valid
by the definition of exponential polynomials. In the case $p=q$ the multipliers $F_j(z)$
reduce to polynomials, in which case $f$ has at most finitely
many zeros outside of the domains $\Lambda_p$. This corresponds to the situation with
exponential sums.

\begin{example}\label{Stein-example}
\textnormal{(Steinmetz \cite{Stein})}
An exponential polynomial may have infinitely many zeros outside of the
log-strips $\Lambda_1(\theta^*,c)$. Indeed, if
	$$
	f(z)=\left(e^{2\pi i z^2}-1\right)\left(e^{(1+i)\pi z}-1\right),
	$$
then the orthogonal rays for $\co(W)=[0,-2\pi i]$ are $\arg(z)=0$ and $\arg(z)=\pi$.
The function $f$ has infinitely many zeros $z_n=n+ni$, $n\in\Z$, lying
outside of the log-strips around the critical rays
$\arg(z)=0, \pi/2, \pi, 3\pi/2$, with at most finitely many exceptions. The integrated counting function $N(r)$ of the
points $z_n$ satisfies $N(r)\asymp r$, and hence this example illustrates the
sharpness of Theorem~\ref{zeros-thm} in the case when $q=2$ and $p=1$.
\end{example}

We note that an exponential polynomial $f$ of the form \eqref{EP} satisfies a linear differential
equation of order
    $$
    k\leq \sum_{j=1}^m (1+\deg(P_j))q^{m-j},\quad q=\max_j\{\deg(Q_j)\},
    $$
with polynomial coefficients \cite{VPT}. Hence the multiplicity
of zeros of $f$ cannot exceed $k-1$, see the discussion related to \eqref{zero-of-multiplicity} below. Next we give a sufficient condition
for the majority of zeros of $f$ to be simple.

\begin{theorem}\label{simple-zeros}
\textnormal{(\cite{division})}
Let $f$ be an exponential polynomial of the form \eqref{normalized-f2}, where the multipliers $F_j(z)$
are exponential polynomials of order $\rho(F_j)\leq q-p$ for some integer $1\leq p\leq q$.
Suppose that the points $\overline{w}_1,\ldots,\overline{w}_n$ are either vertex points or
interior points of the convex hull $\co (W)$. Then the majority of zeros of $f$ are simple
in the sense that
    $$
    N_{\geq 2}(r,1/f)=\left\{\begin{array}{rl}
    O\left(r^{q-p}\right),\ & \textnormal{if}\ q-p\geq 1,\\
    O\left(\log r\right),\ & \textnormal{if}\ p=q.
    \end{array}\right.
    $$
\end{theorem}

Alternatively expressed, the points $\overline{w}_1,\ldots,\overline{w}_n$ in
Theorem~\ref{simple-zeros} that are on the boundary
of $\co(W)$ must be vertex points of $\co(W)$. See \cite[Example~3.3]{division}
for different possibilities that may occur if some or all of the conjugated frequencies
are vertex points of $\co(W)$.

\begin{example}
Modifying \cite[Example~3.3(2)]{division}, we find that infinitely many multiple zeros are possible. Indeed, all integers are double zeros of
    $$
    f(z)=\left(e^{2\pi iz^{q-p}}-1\right)\left(e^{2\pi iz^q}-1\right),\quad q-p\geq 1.
    $$
We have $N_{\geq 2}(r,1/f)\asymp r^{q-p}$, while $N(r,1/f)=(2+o(1))r^q$. This illustrates
the sharpness of Theorem~\ref{simple-zeros}.
\end{example}

The following is a generalization of Theorem~\ref{Polya-theorem} by
P\'olya to the higher order case.

\begin{theorem}\textnormal{(\cite{WH})}\label{generalized-polya}
Let $f$ be an exponential polynomial of the form \eqref{normalized-f2}, where the multipliers $F_j(z)$ are exponential polynomials of order $\rho(F_j)\leq q-p$ for some integer $1\leq p\leq q$. Then the number $n_{\lambda_p}(r)$ of zeros of $f$ in
$\{|z|\leq r\}$ and outside of the finitely many domains $\Lambda_p(\theta^*,c)$ satisfies
		$$
    n_{\lambda_p}(r)=\left\{\begin{array}{rl}
    O\left(r^{q-p}\right),\ & \textnormal{if}\ q-p\geq 1,\\
    O\left(\log r\right),\ & \textnormal{if}\ p=q.
    \end{array}\right.
    $$
Moreover, let $\overline{w}_j,\overline{w}_k$ be two consecutive vertex points of
$\co(W)$, and let $\arg(z)=\theta^\bot$ denote the orthogonal ray corresponding to
the interval $[\overline{w}_j,\overline{w}_k]$. Let $\varepsilon>0$, and let
$\arg(z)=\theta^*$ represent any of the $q$ critical rays induced by $\theta^\bot$.
Then the number of zeros $n(r,\Lambda_p(\theta^*,c))$ of $f$ in $\Lambda_p(\theta^*,c)\cap \{|z|\leq r\}$ has the asymptotic growth
	\begin{equation}\label{Lpn}
	n(r,\Lambda_p(\theta^*,c))
	= \frac{|w_j-w_k|}{2\pi}r^q+O\left(r^{q-p}\log^{3+\varepsilon}r\right).
	\end{equation}
\end{theorem}

Since $\overline{w}_j,\overline{w}_k$ are two consecutive vertex points of
$\co(W)$, we have $\sum |w_j-w_k|=\Ce$, where the summation is taken over all
sides of $\co(W)$. Thus \eqref{Lpn} and the fact that each orthogonal ray
induces $q$ critical rays give raise to
	\begin{equation}\label{n-asymptotic2}
	n(r,1/f)= \frac{q\Ce}{2\pi}r^q
	+O\left(r^{q-p}\log^{3+\varepsilon}r\right).
	\end{equation}
This in turn implies
	$$
	N(r,1/f)= \frac{\Ce}{2\pi}r^q
	+S(r),
	$$
where the error term $S(r)=O\left(r^{q-p}\log^{3+\varepsilon}r\right)$ for $q-p\geq 1$ and $S(r)=O\left(\log^{4+\varepsilon}r\right)$ for $q=p$ is slightly weaker than $O(r^{q-p}+\log r)$ obtained in \cite[Theorem~3.2]{WH}, see also Theorem~\ref{Stein-thm1}. The error term in \eqref{n-asymptotic2} improves that in \eqref{n-asymptotic}.

\begin{problem}
What is the best possible error term in \eqref{Lpn}?
\end{problem}

Theorem~\ref{Polya-theorem} by P\'olya shows that if $q=1$, then
$O(\log r)$ suffices for the error term in \eqref{Lpn}. From Dickson's
result \eqref{zeros-ES-improved} it seems most likely that $O(\log r)$
can be further improved to $O(1)$. However, from Example~\ref{Stein-example}, we find that the error term cannot be less than $O(r^{q-p})$ in general when
$q-p\geq 1$. On the other
hand, we are not aware of any examples showing that the general error
term would have to be bigger than $O(r^{q-p}+\log r)$, where the logarithmic
term is meaningful only if $p=q$.

\subsection{Regularly distributed zeros}

It follows from \cite[Lemma~1.3]{GOP} and \cite[Theorem~1.2.1]{Ronkin} that
the zeros of an exponential polynomial are regularly distributed.
Definition~\ref{def2} below explains what this means. The proof of the above
claim relies on a complicated reasoning in \cite{Ronkin}. In Lemma~\ref{regularly-distributed-lemma} below we will give a direct proof that
relies on results discussed earlier in this section.

\begin{definition}\textnormal{(Levin \cite[pp.~89--92]{Levin1}, Ronkin \cite[pp.~6--7]{Ronkin})}\label{def2}
Let $Z=\{z_n\}$ be a sequence of points in $\C$ of exponent of convergence $\lambda\in(0,\infty)$.
Suppose that $Z$ is organized according to increasing modulus, and that the counting function
$n_Z(r)$ has at most a finite type with respect to $\lambda$. Set
    $$
    S(r,\theta_1,\theta_2)=\{z\in\C : |z|<r,\ \theta_1<\arg(z)<\theta_2\},
    $$
and let $n_Z(r,\theta_1,\theta_2)$ be the number of points $z_n$ in $S(r,\theta_1,\theta_2)$ counting multiplicities.
We say that the sequence $Z$ has \emph{angular density} with respect to $\lambda$ if for all $\theta_1,\theta_2$, with
at most countably many possible exceptions, the limit
    \begin{equation}\label{thelimit}
    \lim_{r\to\infty}r^{-\lambda}n_Z(r,\theta_1,\theta_2)
    \end{equation}
exists. If $Z$ has angular density, and either (1) $\lambda\not\in\N$ or (2) $\lambda\in\N$ and, in
addition, the limit
    \begin{equation}\label{thesum}
    \lim_{r\to\infty}\sum_{0<|z_n|<r}z_n^{-\lambda}
    \end{equation}
exists, then $Z$ is called \emph{regularly distributed}.
\end{definition}

Regarding Case (2) in Definition~\ref{def2}, we have two possibilities: either $\lambda=p$ or $\lambda=p+1$,
where $p$ is the genus of $Z$, see Appendix A. If $\lambda=p+1$, then the limit in \eqref{thesum} automatically exists
because $\sum_{n=1}^\infty |z_n|^{-(p+1)}<\infty$ by the definition of genus. We note that the series
in \eqref{thesum} does not
necessarily converge absolutely when $r\to\infty$ even if the counting function $n(r)$ would have a
nice asymptotic growth rate such as $n(r)\asymp r^\lambda$. Indeed,
applying Riemann-Stieltjes integration and integration by parts, we obtain
    $$
    \sum_{0<|z_n|\leq r}|z_n|^{-\lambda}\asymp \int_1^r t^{-\lambda}dn(t)\asymp \int_1^rt^{-1}\, dt=\log r.
    $$
In other words, for the limit in \eqref{thesum} to exist in the case $\lambda=p$, some cancellation must occur. This justifies the terminology ``regularly distributed''. For example, if the $z_n$'s are the zeros of $\sin z$, then $\lambda=p=1$, and the limit in \eqref{thesum} reduces essentially to an alternating harmonic series, which  converges.

\begin{lemma}\label{regularly-distributed-lemma}
If an exponential polynomial $f$ has an infinite zero-sequence $Z$,
then $Z$ is regularly distributed.
\end{lemma}

\begin{proof}
From the oscillation point of view, we may suppose that $\lambda(f)=\rho(f)=q\geq 1$, for otherwise $f$ reduces to the form $f(z)=F(z)e^{wz^q}$, where $F(z)$ is an exponential polynomial of order $\leq q-1$, and where the exponential term $e^{wz^q}$ merely
determines the growth of $f$ but plays no role in oscillation. Let $\{z_n\}$ denote
the zero-sequence of $f$, listed according to multiplicities and organized with respect to increasing modulus. Without loss of generality, we may suppose that $z_n\neq 0$, that is, $f(0)\neq 0$.

Let $\theta_1<\theta_2$. Since countably
many exceptions for $\theta_1,\theta_2$ are allowed, we may suppose that
$\theta_1,\theta_2\neq\theta^*$ for all
critical angles $\theta^*$. This leaves us two cases to consider.\\[-20pt]
\begin{itemize}
\item[(1)] Suppose that $S(\theta_1,\theta_2)$ doesn't contain any of the critical rays $\arg(z)=\theta^*$.
Let $n_{\Lambda_1}(r)$ denote the non-integrated counting function of the
zeros of $f$ in $|z|\leq r$ lying outside of the finite union of domains $\Lambda_1(\theta^*,c)$. Then
    $$
    n_Z(r,\theta_1,\theta_2)\leq n_{\Lambda_1}(r)\leq \frac{1}{\log 2}\int_r^{2r}\frac{n_{\Lambda_1}(t)}{t}\, dt
    \leq \frac{1}{\log 2}N_{\Lambda_1}(2r)=O\left(r^{q-1}+\log r\right)
    $$
by Theorem~\ref{zeros-thm}, so that the limit in \eqref{thelimit} exists and is equal to zero.\\[-20pt]
\item[(2)] Suppose that $S(\theta_1,\theta_2)$ contains at least one of the critical rays $\arg(z)=\theta^*$.
It is clear that $S(\theta_1,\theta_2)$ essentially contains $\Lambda_1(\theta^*,c)$ for each
$\theta^*\in (\theta_1,\theta_2)$.  Hence the limit in \eqref{thelimit} exists by
Theorems~\ref{zeros-thm} and \ref{generalized-polya}.\\[-20pt]
\end{itemize}
We conclude that the sequence $Z$ has angular density.

Since $\lambda=\lambda(f)=q\in\N$, it remains to check that the limit in \eqref{thesum} exists. By
\eqref{n-asymptotic}, we have $n(r)\asymp r^q$. Hence the integral $\int_0^\infty t^{-q-1}n(t)\, dt$
diverges, which implies that the sum $\sum_{|z_n|>0}|z_n|^{-q}$ diverges \cite[Lemma~2.5.5]{Boas}.
On the other hand, by Riemann-Stieltjes integration and integration by parts, we have
    $$
    \sum_{|z_n|\geq e}\frac{1}{|z_n|^q\log^\alpha |z_n|}=\int_e^{\infty}\frac{dn(t)}{t^q\log^\alpha t}
    \leq (q+\alpha)\int_e^\infty\frac{n(t)}{t^{q+1}\log^\alpha t}<\infty,
    $$
where $\alpha>1$ is fixed. Thus the sequence $Z$ has genus $p=q$.

From \cite[p.~7]{GO}, if $|z|=r<R$,
	$$
	\frac{d^q}{dz^q}\log f(z)\Big|_{z=0}=\frac{2q!}{2\pi R^q}\int_0^{2\pi}
	\log |f(Re^{i\theta})|e^{-iq\theta}\, d\theta+(q-1)!\sum_{|z_n|<R}
	\left(\frac{\bar{z}_n^q}{R^{2q}}-\frac{1}{z_n^q}\right).
	$$
Since $|\log \alpha|=\log^+\alpha+\log^+\frac{1}{\alpha}$ for $\alpha>0$, we have
by \eqref{characteristic1},
	\begin{eqnarray*}
	\bigg|\sum_{|z_n|<R}\frac{1}{z_n^q}\bigg|
	&\leq & \frac{2q}{R^q}\left(m(r,f)+m(r,1/f)\right)+
	\sum_{|z_n|<R}\left|\frac{\bar{z}_n^q}{R^{2q}}\right|+O(1)\\
	&\leq & \frac{4q}{R^q}T(r,f)+\frac{n(R)}{R^q}+O(1)=O(1),
	\end{eqnarray*}
and the assertion follows.
\end{proof}

\subsection{Quotients of exponential polynomials} \label{section3.4} 

For a quick motivation to quotients of exponential polynomials, we recall
an open problem by Hayman (see Problem~2.27 in \cite{Hayman,HL}):
\begin{quote}
\emph{Let $\phi_1,\ldots, \phi_n$ be entire functions of the form
	\begin{equation}\label{phi}
	\phi=\sum e^{f_\nu} \big/ \sum e^{g_\nu},
	\end{equation}
where $f_\nu,g_\nu$ are entire functions.
Does there exist an entire function $f$, not of the form \eqref{phi},
satisfying an algebraic equation $f^n+\phi_1f^{n-1}+\cdots +\phi_{n}=0$?}
\end{quote}
In general, functions of the form \eqref{phi} are not entire.
As an example, Hayman notes that $f(z)=\frac{\sin \pi z^2}{\sin \pi z}$ is
entire but not of the form $\sum e^{f_{\nu}}$, although it is a ratio of such functions.

Another problem related to Hayman's problem will be discussed in
Section~\ref{d-roots-sec} below.
Here we proceed to generalize Theorem~\ref{Ritt-thm} by Ritt to exponential polynomials of arbitrary order $q\geq 1$
instead. The simple example
	$$
	(e^{2\pi iz^q}-1)/(e^{2\pi iz}-1)e^{z^q},\quad q\geq 2,
	$$
shows that the coefficients of the entire quotient may be quotients of exponential polynomials themselves, but of lower order and
with removable singularities. Before stating the generalization of Ritt's theorem, we define two function classes for $p\in\N\cup\{0\}$:
	\begin{eqnarray*}
	P_p&=&\{f: f\, \textrm{is an exponential polynomial and}\, \rho(f)\leq p\},\\
	R_p&=&\{f/g: f,g\in P_p,\, g\not\equiv 0\}.
	\end{eqnarray*}
Then $P_0$ and $R_0$ consist of polynomials and rational functions, respectively.

\begin{theorem}\textnormal{(\cite{division})}\label{lax-rahman}
Let  $n,m,q\in\N$, and let $p\leq q-1$ be a non-negative integer. Set	
	$$
	g(z)=\sum_{j=1}^ma_j(z)e^{\mu_jz^q}\quad\textrm{and}\quad
	h(z)=\sum_{j=1}^nb_j(z)e^{\nu_jz^q},
	$$
where $a_j,b_j\in P_p$ are non-vanishing. If $f(z)=g(z)/h(z)$ is entire, then $f$ too is an
exponential polynomial or reduces to a constant, but with coefficients in $R_p$.
\end{theorem}

We note that the exponential polynomials $e^{2\pi i z^2}-1$ and $e^{2\pi i z} -1$ share infinitely many zeros, namely the integers, and yet their ratio is not an exponential polynomial.
Hence a higher order analogue of Shapiro's conjecture would have to involve more than just ``infinitely many zeros''.
Note that if an exponential sum $f$ has infinitely many zeros, then the exponent of convergence of those zeros is equal to one, same as the order of $f$.

One possible formulation for a higher order analogue could be as follows:
\begin{quote}
\emph{Suppose that two transcendental exponential polynomials $f$ and $g$ of orders $p$ and $q$ share a sequence of zeros, whose exponent of convergence is $\min\{p,q\}$. Then $f$ and $g$ are both multiples of some third exponential polynomial.}
\end{quote}
Yet another formulation relies on the following definition.

\begin{definition}
Let $f$ and $g$ be entire functions. Then $N(r,f,g)$ denotes the integrated counting function of the common zeros of $f$ and $g$. Every common zero $z_0$ with $|z_0|\leq r$ of multiplicity $\mu>0$
for $f$ and of multiplicity $\nu>0$ for $g$ is calculated $\min\{\mu,\nu\}$
times.
\end{definition}

\medskip\noindent
\textbf{Generalized Shapiro's conjecture.} \emph{Suppose that $f$ and $g$
are exponential exponential polynomials of the form
	$$
	f(z)=\sum_{j=0}^m F_j(z)e^{w_jz^{q}}\quad\text{and}\quad
	g(z)=\sum_{j=0}^k G_j(z)e^{\lambda_jz^{q}},
	$$
where $F_j(z)$ and $G_j(z)$ are exponential polynomials of order $\max\{\rho(F_j), \rho(G_j)\}\leq q-p$ for some integer $1\leq p\leq q$. If $f$ and $g$ are not multiples of the same transcendental exponential polynomial of order $q$, then $N(r,f,g)=O(r^{q-p}+\log r)$.}

\begin{remark}
Related to the generalized Shapiro's conjecture, Theorem~\ref{simple-zeros} shows that an exponential polynomial $f$ and its derivative $f'$ are not
multiples of the same transcendental exponential polynomial of the same order, provided that the conjugates $\overline{w}_1,\ldots,\overline{w}_m$ of the frequencies of $f$ are either vertex points or interior points of $\co (W)$.
Note that a simple exponential sum $(1+\alpha e^{wz})^n$ in ${\cal E}$ can have an exponential sum as the reciprocal of its logarithmic derivative, that is, $\frac{1}{nw}(1+\alpha^{-1}e^{-wz})$ for $n\in\N$ and $\alpha, w\in\C\setminus\{0\}$, and they are dual to each other.
When $n\geq 2$, the points $-\overline{jw}$, $1\leq j \leq n-1$, are all  boundary points of $\co(W)=[0, -\overline{nw}]$, but none of them are vertex points.
\end{remark}

Let $f=g/h$, where
	$$
	g(z)=\sum_{j=0}^n G_j(z)e^{w_jz^q}\quad\textrm{and}\quad
	h(z)=\sum_{j=0}^n H_j(z)e^{w_jz^q},
	$$
and where $w_0=0$ and $w_j\neq w_k$ for $j\neq k$.
We allow that some of the multipliers $G_j(z)$ or $H_j(z)$ may vanish identically (more than one may vanish),
but we suppose that the matching multipliers $G_j(z)$ and $H_j(z)$ do not both vanish
identically for any $j$. Let $W_g$ and $W_h$ denote the sets of conjugated
frequencies of $g$ and $h$, and define $W_f=W_g\cup W_h$. Then
$W_f=\{\overline{w}_0,\overline{w}_1,\ldots,\overline{w}_n\}$.

If both $g$ and $h$ have all the frequencies $w_j$, then $f$ is
said to be in the \emph{normal form} \cite{GM}. Following the
reasoning in \cite{GM}, we will show in
Section~\ref{Stein-ratio-proof-sec} that, for suitable constants
$a,b\in\C$, the numerator and the denominator of the quotient
	\begin{equation}\label{all-frequencies}
	S(z)=\frac{1}{f(z)-a}-b
	=\frac{h(z)-bg(z)+abh(z)}{g(z)-ah(z)}
	=\frac{\sum_{j=0}^n \tilde G_j(z)e^{w_jz^q}}{
	\sum_{j=0}^n \tilde H_j(z)e^{w_jz^q}}
	\end{equation}
have all the frequencies $w_j$. In other words, any quotient of
exponential polynomials can be transformed to a quotient that is
in the normal form.

The asymptotic growth of $T(r,f)$, $N(r,f)$, $m(r,f)$ and $m(r,1/(f-c))$
for the quotient $f=g/h$ were found by Steinmetz in \cite[Satz~1]{Stein2}, with $o(r^q)$
representing the error term.  The proof of \cite[Satz~1]{Stein2} relies on a theorem of
Ahlfors-Frostman \cite[p.~276]{Nevanlinna}, according to which
	$$
	N\left(r,\frac{1}{f-a}\right)\sim T(r,f),\quad r\to\infty,
	$$
for a given transcendental meromorphic function $f$ and for all values $a$ outside of a
possible exceptional set
of capacity zero. As $T(r,f)\asymp r^q$, the use of this result leads to the aforementioned
error term $o(r^q)$. Theorem~\ref{Stein-ratio-thm} below is a refinement of
\cite[Satz~1]{Stein2} having improved error terms even in the
standard case $p=1$.  The proof is independent of the Ahlfors-Frostman
theorem.

\begin{theorem}\label{Stein-ratio-thm}
Let $f=g/h$ be defined as above, where the multipliers $G_j(z)$ and $H_j(z)$ are exponential polynomials of order $\leq q-p$ for some integer
$1\leq p\leq q$. Then
\begin{itemize}
\item[\textnormal{(a)}] $T(r,f)=\frac{C(\co(W_f))}{2\pi} r^q-N(r,g,h)+O(r^{q-p}\log^2 r)$,
\item[\textnormal{(b)}]
	$m(r,f)=\frac{C(\co(W_f))-C(\co(W_h))}{2\pi}r^q
	+O(r^{q-p}\log^2 r),$
\item[\textnormal{(c)}]
	$m\left(r,1/(f-c)\right)
	=\frac{C(\co(W_f))-C(\co(W_{g-ch}))}{2\pi}r^q
	+O(r^{q-p}\log^2 r),$ $c\in\C$,
\item[\textnormal{(d)}] $f$ has at most $s\leq n+1$ deficient values,
where $s$ is the number of vertex points of $\co(W_f)$. Moreover,
if $a\in\C$ (resp.~$a=\infty$) is a deficient value of $f$, then there exist a vertex point $\overline{w}_j$ of $\co(W_f)$ such that the multipliers $G_j(z)$ and $H_j(z)$ of $e^{w_jz^q}$ (resp.~the
multiplier $H_j(z)$ of $e^{w_jz^q}$) satisfy $G_j(z)\equiv a H_j(z)$
(resp.~$H_j(z)\equiv 0$).
\end{itemize}
\end{theorem}

The proof of Theorem~\ref{Stein-ratio-thm} is postponed to
Section~\ref{Stein-ratio-proof-sec}.
The example in \cite[p.~463]{Stein2} shows that $f$ can have $n+1$ deficient values, so that the assertion in Part (d) is sharp in the case when $s=n+1$. Note that the original result \cite[Satz~1]{GM} does not say anything about the vertex points. The assertion in Part (a) is meaningful, provided that
	\begin{equation}\label{meaningful}
	2\pi\limsup_{r\to\infty}\frac{N(r,g,h)}{r^q}<C(\co(W_f)).
	\end{equation}
The following refinement of \cite[Satz~1]{GM} shows that,
under certain conditions, $N(r,g,h)$ is a small function compared
to $f$, in which case \eqref{meaningful} clearly holds.	

\begin{theorem}\label{common-zeros-thm}
Suppose that the ratio $f=g/h$ is in the normal form, where the multipliers $G_j(z)$ and $H_j(z)$ are exponential polynomials of order $\leq q-p$ for some integer $1\leq p\leq q$. If the points $\overline{w}_0,\ldots,\overline{w}_n$ are either vertex points or interior points of $\co (W_f)$, and if $G_0H_{n}-G_{n}H_0\not\equiv 0$ and $G_jH_{j+1}-G_{j+1}H_j\not\equiv 0$ for every $j=1,\ldots, n-1$, then
	$$
	N(r,g,h)=\left\{\begin{array}{rl}
    O\left(r^{q-p}\right),\ & \textnormal{if}\ q-p\geq 1,\\
    O\left(\log r\right),\ & \textnormal{if}\ p=q.
    \end{array}\right.
	$$
\end{theorem}
	
The proof of Theorem~\ref{common-zeros-thm} is a simple modification
of the proof of Theorem~\ref{simple-zeros} (see \cite[Theorem~3.1]{division}), or
of the original proof of \cite[Satz~1]{GM}. The details are omitted.

\begin{problem}
Could the error term $O(r^{q-p}\log^2 r)$ in
Theorem~\ref{Stein-ratio-thm} be replaced with a smaller error
term such as $O(r^{q-p}+\log r)$?
\end{problem}

\subsection{Proof of the quotient theorem}\label{Stein-ratio-proof-sec}

In this section we prove Theorem~\ref{Stein-ratio-thm}. This lengthy proof serves three purposes:
\begin{itemize}
\item[(1)] The Ahlfors-Frostman theorem used in proving
\cite[Satz~1]{GM} is known to be complicated. Our reasoning
is based on much simpler deductions.
\item[(2)] The proof is partially based on a reasoning in
\cite{WH}, which has errors. In the proof we will point out
what these errors are, and correct them with a new auxiliary
result, see Lemma~\ref{exceptionalset-lemma} below.
\item[(3)] In addition to Lemma~\ref{exceptionalset-lemma}, the proof contains rather standard technical
manipulations of exponential polynomials, and seeing them
will help the reader to get an intuitive idea of the methods behind the theorems on exponential polynomials.   
\end{itemize}

The proof of Part~(d) can essentially be found in \cite{Stein2}, but it is included
here for the convenience of the reader as it is brief.

\bigskip
\noindent
\textbf{Proof of (d).}
For $a\in\C$, write	
	\begin{equation}\label{all-frequencies-numerator}
	f(z)-a=\frac{g(z)-ah(z)}{h(z)}
	=\frac{\sum_{j=0}^n (G_j(z)-aH_j(z))e^{w_jz^q}}{\sum_{j=0}^n H_j(z)e^{w_jz^q}}.
	\end{equation}
If $a$ is a deficient value for $f$, then we must have $C(\co(W_{g-ah}))<C(\co(W_f))$. The only way the circumference of $\co(W_f)$
can get smaller is when one of its vertex points $\overline{w}_j$
is removed.
Thus there exists at least one but at most $n+1$ indices $j$ for which $G_j(z)-aH_j(z)\equiv 0$. However, for any constant $a_0\neq a$, we then have $G_j(z)-a_0H_j(z)\not\equiv 0$.
Similarly, if the value at infinity is a deficient value for $f$, then
$C(\co(W_{h}))<C(\co(W_f))$, and consequently $H_j(z)\equiv 0$ for at least one but for at
most $n+1$ indices $j$. \hfill$\Box$

\bigskip
\noindent
\textbf{Preparations for the proof of (b).}
Suppose first that both the numerator $g$ and the denominator
$h$ have all the frequencies $w_0,\ldots,w_n$.
Then $C(\co(W_f))=C(\co(W_g))=C(\co(W_h))$. Let $\arg(z)=\theta^*_k$
denote the critical rays for either of $g$ or $h$, where
$k=0,\ldots,\ell$, $\ell\leq sq$, and $s$ is the number of vertex
points of $\co(W_f)$. Define
	$$
	E_k(r)=\big\{\theta\in [0,2\pi] : z=re^{i\theta}\in
	\Lambda_p(\theta_k^*,c)\big\},\quad k=0,\ldots,\ell,
	$$
and $E(r)=\cup_{k=0}^\ell E_k(r)$. Let $\{z_n\}$ denote the sequence
consisting of the zeros of $g$ and $h$ and of their multipliers  $G_0(z),\ldots,G_n(z)$ and $H_0(z),\ldots,H_n(z)$ that do not vanish identically. Then $\{z_n\}$ has exponent of convergence~$q$,
and we may suppose that the points $z_n$ are ordered with respect to increasing modulus. For $\kappa (x)=x^q\log^2(x+e)$, let $F$ denote the collection of
Euclidean discs
	\begin{equation}\label{discs-dn}
	D_n=\{z:|z-z_n|\leq 1/\kappa (|z_n|)\},
	\end{equation}
and let $I$ be the circular projection of $F$ on the positive real axis.
Then by Riemann-Stieltjes integration, it can easily be verified
that $I$ has finite linear measure, see \cite[Remark~1]{WH}.
If $z\not\in F$ lies
between two consecutive log-strips $\Lambda_p(\theta_k^*,c)$,
determined by $\co(W_f)$, then, by \cite[Lemma~7.2]{WH}, there
exists a unique integer $j_0\in\{0,\ldots,n\}$ such that
	\begin{equation}\label{reps}
	\begin{split}
	g(z) &= G_{j_0}(z)e^{w_{j_0}z^q}(1+\veps_{q-p}(r)),\\
	h(z) &= H_{j_0}(z)e^{w_{j_0}z^q}(1+\veps_{q-p}(r)),
	\end{split}	
	\end{equation}	
where $|e_{q-p}(r)|\leq B\exp(-r^{q-p}\log r)$, $|z|=r$ and $B>0$ is a constant. Thus, for such values of $z$, \cite[Corollary~4.2]{WH}
applied to $H_{j_0}(z)$ yields
	$$
	|f(z)|\leq \left|\frac{G_{j_0}(z)}{H_{j_0}(z)}
	\right|(1+|\veps_{q-p}(r)|)
	=\exp\left(O\left(r^{q-p}\log r\right)\right).
	$$
Now, applying \cite[Corollary~4.2]{WH} to $h$, we obtain
	\begin{eqnarray*}
	m(r,f)&=&\frac{1}{2\pi}\int_{E(r)}
	\log^+|f(re^{i\theta})|\, d\theta
	+\frac{1}{2\pi}\int_{[0,2\pi]\setminus E(r)}
	\log^+|f(re^{i\theta})|\, d\theta\\
	&=& O\left(r^{q}\log r\right)\int_{E(r)}d\theta
	+O\left(r^{q-p}\log r\right)\\
	&=& O\left(r^{q-p}\log^2 r\right),\quad r\not\in I.
	\end{eqnarray*}

Suppose then that the denominator $h$ has all the frequencies, but
the numerator $g$ does not. For a suitable constant $a\in\C$,
the function
	\begin{equation}\label{normal-T}
	T(z)=f(z)-a=\frac{g(z)-ah(z)}{h(z)}
	\end{equation}
is a ratio of two exponential polynomials and has all the frequencies both in the numerator and in the denominator, see \eqref{all-frequencies-numerator}. Thus, from the discussion above,
	$$
	m(r,T)=O\left(r^{q-p}\log^2 r\right),\quad r\not\in I,
	$$
and consequently
	$$
	m(r,f)=O\left(r^{q-p}\log^2 r\right),\quad r\not\in I.
	$$

It remains to assume that the numerator $g$ has all the frequencies, but the denominator $h$ has not. We may suppose that the missing
frequencies of $h$ are vertex points of $\co(W_g)$, for otherwise
this case reduces to the first case discussed above. Further, we
may suppose that the origin is an interior point of $\co(W_g)$
because we may multiply both $g$ and $h$ by a suitable exponential
term that has no contribution to $m(r,f)$. Now the log-strips for
$g$ and for $h$ may be different, but nevertheless they cover
only a small portion of the complex plane. The unique dominating
exponential terms still exist for $g$ and $h$, but they may not be
equal as opposed to the situation in \eqref{reps}.

Let $h_g(\theta)$ and $h_h(\theta)$ denote the Phragm\'en-Lindel\"of
indicator functions of $g$ and $h$, and let $k_g(\theta)$ and
$k_h(\theta)$ denote the supporting functions of $\co(W_g)$
and $\co(W_h)$, respectively. See Appendices \ref{Nevanlinna-appendix}
and \ref{convex-appendix} for the definitions of these functions.
Since the origin
is an interior point of $\co(W_g)$, the supporting functions
and hence the indicator functions are all non-negative by \eqref{kqh}
below. Following the proof of \cite[Theorem~3.1]{WH} (or the original proof in \cite[Satz~1]{Stein}),
using the $2\pi$-periodicity of the supporting functions, the identity \eqref{kqh} as well as Theorem~\ref{circumference-thm} below, we obtain
	\begin{eqnarray*}
	m(r,f) &=& \frac{r^q}{2\pi}\int_0^{2\pi}
	\big(h_g(\theta)-h_h(\theta)\big)\, d\theta
	+O\left(r^{q-p}\log^2 r\right)\\
	&=& \frac{r^q}{2\pi}\int_0^{2\pi}
	\big(k_g(q\theta)-k_h(q\theta)\big)\, d\theta
	+O\left(r^{q-p}\log^2 r\right)\\
	&=& \frac{r^q}{2\pi}\int_0^{2\pi}
	\big(k_g(\varphi)-k_h(\varphi)\big)\, d\varphi
	+O\left(r^{q-p}\log^2 r\right)\\
	&=&\frac{C(\co(W_f))-C(\co(W_h))}{2\pi}r^q
	+O(r^{q-p}\log^2 r),\quad r\not\in I.
	\end{eqnarray*}

We have now proved that in all possible cases,
	\begin{equation}\label{I}
	m(r,f)=\frac{C(\co(W_f))-C(\co(W_h))}{2\pi}r^q
	+O(r^{q-p}\log^2 r),\quad r\not\in I,
	\end{equation}
where $I=\cup_{n=1}^\infty I_n\subset [0,\infty)$ is an exceptional set of finite linear measure, and the sets $I_n$ are closed intervals of
the form $I_n=\big[r_n-1/\kappa (r_n),r_n+1/\kappa (r_n)\big]$ for $r_n=|z_n|$. Standard methods do not apply to remove the set $I$ because $m(r,f)$ is not in general non-decreasing. The key idea is to move on to
considering $T(r,f)$, which is non-decreasing, and then come back
to considering $m(r,f)$. In this direction, we point out that \cite[Lemma~4.3]{WH} has a mistake in the proof, and consequently
the use of this lemma leads to an error term of magnitude $o(r^q)$.
Therefore we need a new approach for dealing with exceptional sets.

\medskip
\begin{remark}
Suppose that the numerator $g$ of $f=g/h$ does not have all the
frequencies. Then the reasoning in \eqref{all-frequencies-numerator}
and \eqref{normal-T} shows that, for a suitable $a\in\C$, the numerator of $f-a=\frac{g-ah}{h}$ has all the frequencies. Consequently, the denominator of $\frac{1}{f-a}=\frac{h}{g-ah}$ has all the frequencies. If the numerator $h$ of the quotient $\frac{1}{f-a}$ does not have
all the frequencies, we may use this process again
and subtract from it a suitable constant $b\in\C$.
This shows that the quotient $S$ in \eqref{all-frequencies}
is in the normal form.
\end{remark}

\medskip
\noindent
\textbf{A new approach on exceptional sets.}
The rest of the proof of Theorem~\ref{Stein-ratio-thm} requires two
lemmas. The first one is outlined without a proof in \cite[p.~412]{CGHR}, but we include the short and easy proof here for convenience.

\begin{lemma}\label{convex-lemma}
\textnormal{(\cite{CGHR})}
Let $\varphi:(0,\infty)\to (0,\infty)$ be differentiable and
eventually convex. Then there exists a constant $R>0$ such that
	$$
	\varphi(x)\leq 2\varphi(x+\varphi(x)),\quad x\geq R.
	$$
\end{lemma}

\begin{proof}
From eventual convexity and differentiability, there exists a constant
$R_0>0$ such that
	$$
	\varphi(x+\varphi(x))\geq
	\varphi(x)+\varphi'(x)\varphi(x),\quad x>R_0.
	$$
Hence it suffices to prove that there exists an $R>R_0$
such that $\varphi'(x)\geq-1/2$ for all $x\geq R$. Suppose on the contrary to this claim that there exists a strictly
increasing sequence $\{x_n\}$ of positive real numbers such that $\varphi'(x_n)< -1/2$ for all $n$.
Since $\varphi'(x)$ is increasing by convexity, we have $\varphi'(x)\leq -1/2$ for all $x\geq x_1$.
But then $\varphi(x)\to-\infty$, which contradicts the fact that $\varphi(x)>0$ for all $x>0$.
\end{proof}

The exceptional set $I$ appearing in \eqref{I} will be dealt
with the following lemma.

\begin{lemma}\label{exceptionalset-lemma}
Let $\{r_n\}\subset [1,\infty)$ be a strictly increasing sequence tending to infinity.
Suppose that $\kappa:(0,\infty)\to [1,\infty)$ is any non-decreasing and differentiable
function such that $1/\kappa (x)$ is eventually convex and that
$\sum_{n=1}^\infty \frac{1}{\kappa (r_n)}<\infty$. Denote
	$$
	I=\bigcup_{n=1}^\infty I_n,
	$$
where $I_n=[r_n-1/\kappa (r_n),r_n+1/\kappa (r_n)]$. Suppose that $u,v:(0,\infty)\to (0,\infty)$ are non-decreasing functions such that
	$$
	u(r)\leq v(r),\quad r\not\in I.
	$$
Suppose that at most a finite uniform number $K\in\N$ of consecutive
intervals $I_n$ can overlap. (That is, if the intervals $I_n$
form clusters, there are at most $K$ intervals in each cluster.)
Let $\alpha=\alpha(r)=1+\frac{5K}{r\kappa (r)}$. Then there exists a constant
$R>0$ such that
	$$
	u(r)\leq v(\alpha r),\quad r\geq R.
	$$
\end{lemma}

\begin{proof}
The function $\varphi(x)=1/\kappa (x)$ satisfies the assumptions of Lemma~\ref{convex-lemma}, so
there exists a constant $R>0$ such that
	\begin{equation}\label{less2}
	\kappa\left(x+\frac{1}{\kappa (x)}\right)\leq 2\kappa (x),\quad x\geq R.
	\end{equation}

Let $J=[0,\infty)\setminus I$. We claim that $[r,\alpha r]\cap J\neq\emptyset$ for all $r\geq R$. Suppose on the contrary to this claim that there exists
a strictly increasing sequence $\{t_n\}\subset [R,\infty)$
tending to infinity such that $[t_n,\alpha t_n]\subset I$ for all $n\in\N$,
where $\alpha=\alpha(t_n)$.

Suppose first that $[t_n,\alpha t_n]\subset I_m$ for some $m$, that is,
	$$
	r_m-\frac{1}{\kappa (r_m)}\leq t_n<\alpha t_n\leq r_m+\frac{1}{\kappa (r_m)}.
	$$
But now, using \eqref{less2},
	\begin{eqnarray*}
	\meas([t_n,\alpha t_n])&=&t_n(\alpha-1)\geq 5\kappa (t_n)^{-1}\geq 5k\left(r_m+\frac{1}{\kappa (r_m)}\right)^{-1}\\
	&\geq& \frac52\kappa (r_m)^{-1}=\frac54\meas(I_m),
	\end{eqnarray*}
which is a contradiction.	

Suppose that $[t_n,\alpha t_n]\subset I_m\cup I_{m+1}$ for some $m$, where $I_m\cap I_{m+1}\neq\emptyset$.
Further we may suppose that
	$$
	r_m+\frac{1}{\kappa (r_m)}< r_{m+1}+\frac{1}{\kappa (r_{m+1})},
	$$
for otherwise $[t_n,\alpha t_n]\subset I_m$, which is a contradiction by the discussion above.
For the same reason, the case $t_n\not\in I_m$ also results in a contradiction.
Hence $t_n\in I_m$, and we have
	$$
	r_m-\frac{1}{\kappa (r_m)}\leq t_n \leq r_m+\frac{1}{\kappa (r_m)}< r_{m+1}+\frac{1}{\kappa (r_{m+1})}.
	$$
But now $K\geq 2$, and using \eqref{less2} twice gives us
	\begin{eqnarray*}
	\meas([t_n,\alpha t_n])&=&t_n(\alpha-1)\geq 2\cdot 5\kappa (t_n)^{-1}\\
	&\geq& 5\kappa\left(r_m+\frac{1}{\kappa (r_m)}\right)^{-1}+5\kappa\left(r_{m+1}+\frac{1}{\kappa (r_{m+1})}\right)^{-1}\\
	&\geq& \frac52\kappa (r_m)^{-1}+\frac52\kappa (r_{m+1})^{-1}
	\geq \frac54 \meas(I_m\cup I_{m+1}),
	\end{eqnarray*}
which is a contradiction.	

Suppose that we have covered the cases where $[t_n,\alpha t_n]$ is a subset of at most $K$ consecutive intervals, the conclusion being a contradiction in each case for $1\leq K \leq M$  for a fixed integer $M\geq 2$. Now suppose that $[t_n,\alpha t_n]\subset I_m\cup I_{m+1}\cup\cdots\cup I_{m+M}$ for some index $m \in\N$, where $I_j\cap I_{j+1}\neq\emptyset$ for every $j$ involved.
We may suppose that
	$$
	r_m+\frac{1}{\kappa (r_m)}< r_{m+1}+\frac{1}{\kappa (r_{m+1})}<\cdots <r_{m+M}+\frac{1}{\kappa (r_{m+M})},
	$$
for otherwise $[t_n,\alpha t_n]$ belongs to a union of $\leq M$ intervals, which is a contradiction.
The same conclusion holds if $t_n\not\in I_m$. Hence $t_n\in I_m$, and we have, similarly as above,
	\begin{eqnarray*}
	\meas([t_n,\alpha t_n])&=&t_n(\alpha-1)\geq (M+1)\cdot 5\kappa (t_n)^{-1}\\
	&\geq &\frac52\kappa (r_m)^{-1}+\frac52\kappa (r_{m+1})^{-1}+\cdots+\frac52\kappa (r_{m+M})^{-1}\\
	&\geq& \frac54\meas(I_m\cup I_{m+1}\cdots\cup I_{m+M}),
	\end{eqnarray*}
which is a contradiction. This covers the case for $K=M+1$,
and thus for any~$K\in \N$.	

Finally, let $r\geq R$, and take $t\in [r,\alpha r]\cap J$. Then
	$$
	u(r)\leq u(t)\leq v(t)\leq v(\alpha r)=v(r+5K/\kappa (r))
	$$
by the monotonicity of $u$ and $v$.
\end{proof}

We will now continue the proof of Theorem~\ref{Stein-ratio-thm}.

\bigskip
\noindent
\textbf{Proof of (a).}	
From Theorem~\ref{Stein-thm1} and from the definition of $N(r,g,h)$,
	\begin{equation}\label{common-zeros}
	\begin{split}
	N(r,f) &= N\left(r,\frac{g}{h}\right)
	=N\left(r,\frac1h\right)-N(r,g,h)\\
	&=\frac{C(\co(W_h))}{2\pi} r^q-N(r,g,h)+O(r^{q-p}+\log r).
	\end{split}
	\end{equation}
From \eqref{I}, we find that
	\begin{equation}\label{TI}
	T(r,f)=\frac{C(\co(W_f))}{2\pi} r^q-N(r,g,h)
	+O(r^{q-p}\log^2 r),	\quad r\not\in I,
	\end{equation}
where $I=\cup_{n=1}^\infty I_n\subset [0,\infty)$ is an exceptional set of finite linear measure consisting of closed intervals $I_n=\big[r_n-1/\kappa (r_n),r_n+1/\kappa (r_n)\big]$ for $\kappa (x)=x^q\log^2(x+e)$
and $r_n=|z_n|$.

Recall that the sequence $\{z_n\}$ consists of the zeros of $g$
and $h$ and those of their multipliers, is ordered with respect to increasing modulus, and has exponent of convergence~$q$. In
fact, if $n_{\{z_n\}}(r)$ denotes the number of points $z_n$ in
$|z|\leq r $, we infer from \eqref{n-asymptotic} that
	\begin{equation}\label{n-asymptotic3}
	n_{\{z_n\}}(r)=C_0r^q+o(r^q)
	\end{equation}
for some constant $C_0>0$.

We now use a reasoning due to Whittaker \cite{Whittaker}: Associate with the disc
$D_1$ in \eqref{discs-dn} any discs $D_n$ which touch or intersect
$D_1$, add to the group any discs which touch or intersect these new discs, and so on. The group so formed will consist of a finite number of discs, because none of their centres can be more than the distance
	$$
	\frac{1}{\kappa (r_1)}+\sum_{n=2}^\infty\frac{2}{\kappa (r_n)}\leq 2H
	$$
away from $z_1$, where
	$$
	H=\sum_{n=1}^\infty \frac{1}{\kappa (r_n)}.
	$$
Indeed, if a closed disc $\overline{D}(z_1,2H)$ of radius $2H$ centred at $z_1$ would contain infinitely many points $z_n$, then it would
contain an accumulation point of the points $z_n$, which is impossible.
Next, take the first point $z_n$ not in the first group and form a second group, and so on. Following Whittaker, these groups of discs are called ``nebulae''. The diameter $d(N)$ of the nebula $N$ is at most $2H$ and tends to zero as the distance $l$ of $N$ from the origin tends to infinity, because
	\begin{eqnarray*}
	d(N)&\leq& \sum_{|z_n|\geq l}\frac{2}{\kappa (r_n)}
	=2\int_{l}^\infty\frac{dn_{\{z_n\}}(t)}{\kappa (t)}
	\leq 2\int_{l}^\infty
	\frac{n_{\{z_n\}}(t)\kappa'(t)}{\kappa (t)^2}\, dt\\
	&\leq& 2(2+q)\int_{l}^\infty
	\frac{n_{\{z_n\}}(t)\, dt}{t^{q+1}\log^2(t+e)}
	\leq \frac{C_1}{\log l}\to 0,\quad l\to \infty,
	\end{eqnarray*}
where $C_1=2(2+q)(C_0+1)>0$ is a uniform constant independent on $n$
and $C_0>0$ is the constant from \eqref{n-asymptotic3}. Since
$\{z_n\}$ is a union of zero sequences of finitely many exponential
polynomials, we may use the reasoning in \cite[p.~408]{VPT}
to conclude that the number of points $z_n$ in each nebula is
uniformly bounded, say by a constant $K>0$.

From \eqref{TI}, there are now two possibilities depending on
the growth of $N(r,f,g)$:
\begin{itemize}
\item[(I)] If $T(r,f)=O(r^{q-p}\log^2 r)$ for $r\not\in I$,
then the assertion in Part~(a) follows by standard
methods, see \cite[Lemma~1.1.1]{Laine}.
\item[(II)]
If $T(r,f)=C_2r^q+O(r^{q-p}\log^2 r)$ for some $C_2>0$ and for $r\not\in I$, then the assertion in Part~(a) follows by Lemma~\ref{exceptionalset-lemma}. Indeed, we have
	$$
	T(r,f)\leq C_2\left(r+\frac{5K}{\kappa (r)}\right)^q
	+O(r^{q-p}\log^2 r)
	=C_2r^q+O(r^{q-p}\log^2 r)
	$$
and	
	\begin{eqnarray*}
	C_2r^q+O(r^{q-p}\log^2 r) &=&
	C_2\left(r-\frac{2}{r^q}\right)^q+O(r^{q-p}\log^2 r)\\
	&=& T\left(r-\frac{2}{r^q},f\right)
	\leq T\left(r-\frac{2}{r^q}+\frac{5K}{\kappa (r-2/r^q)},f\right)\\
	&\leq& T\left(r-\frac{2}{r^q}+\frac{1}{(r-2/r^q)^q},f\right)
	\leq T(r,f)
	\end{eqnarray*}
for all $r$ large enough.
\end{itemize}
This completes the proof of Part~(a). \hfill$\Box$

\bigskip
\noindent
\textbf{Proof of (b) and (c).}
Since $m(r,f)=T(r,f)-N(r,f)$, the assertion in Part~(b) follows
from \eqref{common-zeros} and Part~(a). Further, since
$\frac{1}{f-c}=\frac{h}{g-ch}$, the assertion in Part~(c) follows
from Part~(b). \hfill$\Box$

\begin{remark}
As already indicated above, \cite[Lemma~4.3]{WH} has a mistake in the proof, and consequently the use of this lemma in proving
\cite[Theorem~3.1]{WH} leads to an error term of magnitude $o(r^q)$.
This can be avoided by using Lemma~\ref{exceptionalset-lemma} instead,
together with the fact that, for any given exponential polynomial, the number of discs in each nebula is uniformly bounded.
\end{remark}

\subsection{The $d$-th roots of exponential polynomials}\label{d-roots-sec} 

Motivated by the Green-Griffiths-Lang conjecture for projective spaces with moving targets, Guo, Sun and Wang proved the following result on an exponential polynomial which is a $d$-th power of an entire function.

\begin{theorem}
\textnormal{(Guo-Sun-Wang \cite{GSW0})}\label{Guo-Sun-Wang}
Let $f$ be an exponential polynomial of order~$q$ of the form \eqref{EP}.
Assume that $f = g^d$ for some integer $d \geq 2$ and some entire function $g$.
Then $g$ is also an exponential polynomial of order~$q$.
In addition, if the multipliers $P_j(z)$ of $f$ are constant functions for all $1 \leq i \leq m$, then $g$ is an exponential polynomial of order~$q$ with constant multipliers as well.
\end{theorem}

We note that if an exponential polynomial is decomposed as a product of two entire functions, the entire components are not necessary exponential polynomials. For example, the exponential polynomial $f=(e^z-1)^2$ can be decomposed as a product of two entire functions $g=(e^z-1)/z$ and $h=z(e^z-1)$, but $g$ is not an exponential polynomial because the coefficients do not
belong to $\C[z]$. Thus $g$ being a single entire function and not a product
of them is essential in Theorem~\ref{Guo-Sun-Wang}.

We see that Theorem~\ref{Guo-Sun-Wang} is an affirmative partial answer
to Hayman's problem discussed in Section~\ref{section3.4}. This leads us
to state a new open problem as follows.

\begin{problem} \label{Problem3}
Let $\phi_1, \ldots , \phi_d$ be exponential polynomials of the form~\eqref{EP} and of order at most $q$.
Assume that an entire function $g$ satisfies an algebraic equation of the form
\begin{equation} \label{algebraic}
g^d+\phi_1g^{d-1}+\cdots+\phi_{d-1}g+\phi_d=0.
\end{equation}
Is $g$ an exponential polynomial of order at most $q$?
\end{problem}

If $d=2$, this is actually the case, since an application of Theorem~\ref{Guo-Sun-Wang} to
	$$
	g^2+\phi_1g+\phi_2=\Bigl(g+\frac{1}{2}\phi_1\Bigr)^2
	+\phi_2-\frac{1}{4}\phi_1^2=0
	$$
implies that $g+\frac{1}{2}\phi_1$ is also an exponential polynomial of order equal to the order of $\phi_2-\frac{1}{4}\phi_1^2$.
Differing from Theorem~\ref{Guo-Sun-Wang}, we have to take into account a possible degeneration in the order even if all the $\phi_j$
are of the same order~$q$.

Finally, let us consider the possibility of a factorization of an exponential polynomial into a product $g_1^{d_1}g_2^{d_2}$ of two entire functions $g_1, g_2$ for some integers $d_1, d_2$.  In the particular case when $d_1=d_2$, Theorem~\ref{Guo-Sun-Wang} says that the product $g_1g_2$ should be an exponential polynomial, while if $d_1+d_2=0$, it says that the ratio $g_1/g_2$ should be so. These observations are related to Theorem~\ref{lax-rahman} and to Generalized Shapiro's conjecture in Section~\ref{section3.4}.

\section{Exponential polynomials and ODE's}\label{ODE-sec}

It is well-known that if the coefficients $A_0(z),\ldots,A_{n-1}(z)$ are entire, then all solutions of
    \begin{equation}\label{lde}
      f^{(n)}+A_{n-1}(z)f^{(n-1)}+\cdots +A_1(z)f'+A_0(z)f=0
     \end{equation}
are entire as well.
A theorem by Wittich~\cite{Laine, Wittich} states that all solutions of \eqref{lde} are of finite order of growth if and only if the coefficients are polynomials.
If $A_k(z)$ is the last transcendental coefficient so that $A_{k+1}(z),\ldots,A_{n-1}(z)$ are polynomials, then a theorem by Frei \cite{Frei, Laine} states that \eqref{lde} possesses at most $k$ linearly independent solutions of finite order.
The possible orders of solutions of~\eqref{lde} with polynomial coefficients are
discussed in Appendix~\ref{DE-appendix}.
In this section we discuss solutions of~\eqref{lde} having completely regular growth.
Exponential polynomials turn out to be typical examples of such solutions.
Thus, in view of the previous section, we will discuss Problem~\ref{Problem3}
for linear differential equations instead of the algebraic
equations~\eqref{algebraic}. We then proceed to discuss the role
of exponential polynomials in the theory of non-linear differential equations
such as the Riccati equations, Malmquist type equations,
Tumura-Clunie type equations, and binomial equations.

\subsection{Solutions of completely regular growth}\label{crg-solutions-sec} 

Following Levin \cite[pp.~139--140]{Levin1} or Ronkin \cite[p.~6]{Ronkin},
we say that an entire function $f$ is of \emph{completely regular growth} (c.r.g.)
if it has at most finite type with respect to its order $\rho=\rho(f)\in (0,\infty)$,
and if there exists a sequence of Euclidean discs $D(z_k,r_k)$ satisfying
	\begin{equation}\label{r.eq}
	\sum_{|z_k|\leq r}r_k=o(r)
	\end{equation}
such that
	\begin{equation}\label{crg}
	\log |f(re^{i\theta})|= (h_f(\theta)+o(1))r^{\rho},
	\quad re^{i\theta}\not\in\bigcup_k D(z_k,r_k),
	\end{equation}
as $r\to\infty$ uniformly in $\theta$. A set $E\subset\C$ which can be covered by a sequence of discs $D(z_k,r_k)$ satisfying \eqref{r.eq} is known as a $C_0$-set. The circular projection of a $C_0$-set on the positive real axis has zero upper linear density.
Recall that the upper and lower linear
densities of a set $E\subset [0,\infty)$ are given, respectively, by
    $$
    \overline{\operatorname{dens}}(E)=\limsup_{r\to\infty}\frac{\int_1^r \chi_E(t)\, dt}{r-1}
    \quad\textrm{and}\quad
    \underline{\operatorname{dens}}(E)=\liminf_{r\to\infty}\frac{\int_1^r \chi_E(t)\, dt}{r-1},
    $$
where $\chi_E$ denotes the characteristic function of the set $E$.

\begin{example}
\textnormal{(\cite{GOP})}
If $f(z)=e^z$, then $\log |f(re^{i\theta})|=r\cos\theta$, while if $f$ is either
of the trigonometric functions $\sin z=(e^{iz}-e^{-iz})/(2i)$ or $\cos z=(e^{iz}+e^{-iz})/2$, then $\log |f(re^{i\theta})|=r|\sin\theta|+o(1)$.
In general, exponential polynomials (of all orders) are of c.r.g.
\end{example}

The following result was proved independently by Petrenko and Steinmetz.

\begin{theorem}\label{crg-thm}
\textnormal{(Petrenko \cite[p.~110]{Petrenko}, Steinmetz \cite{Stein3})}
All transcendental solutions of \eqref{lde} with polynomial coefficients are of c.r.g.
\end{theorem}

Conversely to Theorem~\ref{crg-thm}, we recall from \cite{VPT} that any exponential polynomial is a solution of some equation of the form \eqref{lde} with polynomial
coefficients. Not all  transcendental solutions of \eqref{lde} with polynomial
coefficients are exponential polynomials. For
example, the Airy integral (of order 3/2) is a well-known solution of $f''-zf=0$.

The result in Theorem~\ref{crg-thm} motivated Gol'dberg and Ostrovskii to pose the following question \cite[p.~300]{Havin}:
	\begin{quote}
	\emph{If the coefficients $A_0(z),\ldots,A_{n-1}(z)$ 
	of \eqref{lde} are of c.r.g., and if \eqref{lde} possesses a
	transcendental entire solution $f$ of finite order of growth, 
	then is it true that $f$ is of c.r.g.?}
	\end{quote}
In a recent manuscript, Bergweiler has answered this question in the negative.

\begin{theorem}\label{GOP-thm}
\textnormal{(Bergweiler \cite{Berg})}
Given $0<\sigma<\rho<1$, there exist entire functions $A(z)$ and $B(z)$ of order $\rho$
and of c.r.g.~and an entire function $f$ of order $\sigma$ and of lower order $0$ such that
	\begin{equation}\label{ldeAB}
	f''+A(z)f'+B(z)f=0.
	\end{equation}
\end{theorem}

\begin{problem}
Is a result analogous to Theorem~\ref{GOP-thm} valid for
integer orders $\sigma$ and $\rho$? Alternatively, we
may strengthen the assumption and ask: If the coefficients $A_0(z),\ldots,A_{n-1}(z)$ of \eqref{lde} are exponential polynomials, and if \eqref{lde} possesses a transcendental entire solution $f$ of finite order of growth, then is it true that $f$ is of c.r.g.?
\end{problem}

An interested reader may find several examples of
equations \eqref{lde} that have exponential polynomials
as coefficients and solutions in \cite{dual2, GOP, WGH}.
 
If some of the coefficients of \eqref{lde} are not of c.r.g., then $f$ is not necessarily of c.r.g. Indeed, Gol'dberg has shown \cite[p.~300]{Havin} that if $f$ is any entire function with zeros of multiplicity $\leq n-1$, then $f$ is a solution of some differential equation of the form \eqref{lde} with
entire coefficients. Conversely, it is well-known that the zeros of non-trivial entire solutions $f$
of \eqref{lde} are of multiplicity $\leq n-1$. To prove this, suppose that
$f$ has a zero of multiplicity $k\geq n$ at $z_0\in\C$. Write \eqref{lde}
in the form
	\begin{equation}\label{zero-of-multiplicity}
	-A_0(z)=\frac{f^{(n)}}{f}+A_{n-1}(z)\frac{f^{(n-1)}}{f}
	+\cdots+A_1(z)\frac{f'}{f}.
	\end{equation}
Then the right-hand side of \eqref{zero-of-multiplicity} has a pole of multiplicity
$n$ at $z_0$, which is a contradiction because $A_0(z)$ is entire.

Theorem~\ref{indicators-thm} below gives a condition for the
solutions of the second order differential equation	\eqref{ldeAB} with entire coefficients to be of infinite order. This result shows that if $B(z)$ dominates the growth of
$A(z)$ on just one ray, the solutions cannot be of c.r.g.

\begin{theorem}\label{indicators-thm}
\textnormal{(\cite{GOP})}
Let $A(z)$ and $B(z)$ be entire such that $\rho(A)=\rho(B)\in (0,\infty)$, and assume that
both are of finite type. If \eqref{ldeAB} possesses a solution $f\not\equiv 0$ of finite order, then
    \begin{equation}\label{indicators}
    h_B(\theta)\leq \max\{0,h_A(\theta)\},\quad \theta\in [-\pi,\pi).
    \end{equation}
In particular, if there exists a $\theta_0\in [-\pi,\pi)$
such that $\max\{0,h_A(\theta_0)\}<h_B(\theta_0)$, then all non-trivial solutions of \eqref{ldeAB} are of infinite order.
\end{theorem}

If $\rho(A)<\rho(B)$, then all solutions $f\not\equiv 0$ of
\eqref{ldeAB} are of infinite order of growth \cite[Corollary~1]{Gundersen2},
while if $\rho(A)>\rho(B)$, then $f$  satisfies $\rho(f)\geq\rho(A)$ \cite[Theorem~2]{Kwon}.

\subsection{Periodic and subnormal solutions}\label{periodic-subnormal-sec}

Recall that any periodic entire function $f$ of period~$c\in\C\setminus\{0\}$ can be represented as a composition $g\left(\exp\Bigl(\frac{2\pi i}{c}z\Bigr)\right)$ of an analytic function $g(w)=\sum_{n=-\infty}^{+\infty}a_n w^n$ in the punctured plane $0<|w|<\infty$ and of the exponential function $e^{(2\pi i/c)z}$.  This result can be found, for example, in Ahlfors' book \cite{Ahlfors}, where such an $f$ is called a simply periodic function.

If, in addition to $c$-periodicity, we assume that $f$ is of finite order, especially if neither of $w=0$ and $w=\infty$ is an essential singular point of $g(w)$, then $g(w)$ becomes a so-called Laurent polynomial
	$$
	g(w)=\sum_{n=s}^{t}a_n w^n,\quad s, t\in \Z,\ s\leq t,
	$$
and the order of $f(z)$ attains the smallest value, that is, one.
This form of periodic functions typically appear both as solutions and as coefficients of differential equations, which we will observe in this section.

It is useful to write a Laurent polynomial $g(w)$ in the form $g(w)=g_1(w)+g_2(1/w)$ for two polynomials
	$$
	g_1(w)=\sum_{n=0}^{t}a_n w^n\quad
	\textrm{and}\quad g_2(w)=\sum_{m=1}^{|s|}a_{-m} w^m,
	$$
where $a_n=\frac{1}{2\pi i} \int_{|w|=r} g(w)w^{-n-1}dw$ for any $r>0$.
Then $f$ is determined completely as the sum $g_1\bigl(e^{(2\pi i/c)z}\bigr)+g_2\bigl(e^{-(2\pi i/c)z}\bigr)$, which is a simple element of ${\cal E}$.
This exponential sum~$f$ is of regular growth in the sense of Nevanlinna and
clearly of c.r.g.

One of the classical questions on linear differential equations
is whether or when the periodicity of the coefficients runs in the family of its solutions. For example, in 1983, Bank and Laine~\cite{Bank-LaineJRAM}
gave a representation of solutions of simply periodic second order linear differential equations, which accelerated related studies as found in~\cite{HILL, Shimomura2002}. Note that the case when the coefficients are doubly periodic
is very different since as elliptic functions they are not entire functions anymore and have nothing to do with exponential polynomials.
In fact, take two complex numbers $\alpha$, $\beta$ with $\frac{\alpha}{\beta} \in \R\setminus \Q$, and a Laurent polynomial $Q(w_1,w_2)$ in two variables, that is,
	$$
	Q(w_1,w_2)=\sum_{n_1=s_1}^{t_1}
	\sum_{n_2=s_2}^{t_2}w_1^{n_1}w_2^{n_2},\quad
	s_j,t_j\in \Z,\ s_j\leq t_j,\ j=1,2.
	$$
Then $Q(e^{\alpha z}, e^{\beta z})$ belongs to~${\cal E}$ but it is not simple or irreducible in general. Moreover, $Q(e^{\alpha z}, e^{\beta z})$ is never doubly periodic, or even simply periodic in general.

The discussions above lead us to be concerned with finite-order solutions of~\eqref{ldeAB} when $A(z)$ and $B(z)$ are finite-order entire functions of the same period, say~$c$. That is, $A(z)$ and $B(z)$ are of the form $g_1\bigl(e^{(2\pi i/c)z}\bigr)+g_2\bigl(e^{-(2\pi i/c)z}\bigr)$ for two polynomials $g_1$ and $g_2$. As one of the simplest cases, we take $g_1(w)\equiv 0$ and $g_2(w)=w$ for $A(z)$ and $g_1(w)\equiv \alpha$ and $g_2(w)\equiv 0$ for $B(z)$ in~\eqref{ldeAB}.

\begin{example}\label{Frei-example}
\textnormal{(Frei \cite{Frei2})}
The differential equation
	\begin{equation}\label{Frei-eqn}
	f''+e^{-z}f'+\alpha f=0,\quad \alpha\in\C\setminus\{0\},
	\end{equation}
has a subnormal solution $f$ (that is, $f\not\equiv 0$ and $\log T(r,f)=o(r)$) if and only if
$\alpha=-m^2$ for a positive integer $m$. This solution $f$ is a polynomial in $e^z$ of degree $m$, that is, an exponential sum of the form
	\begin{equation}\label{Frei-f}
	f(z)=1+C_1e^z+\cdots+C_me^{mz}, \quad C_j\in \mathbb{C}.
	\end{equation}
\end{example}

In Frei's example, the coefficients are periodic functions with period
$c=2\pi i$. When the existence of the solution $f$ is known,
the coefficients $C_j$ can be determined by a direct computation \cite{dual2}.
One may wonder, however, whether there is well-grounded information why the solution $f$ of~\eqref{Frei-eqn} is obtained as a polynomial in~$e^z$
for the special parameter $\alpha=-m^2$ only.
This uncertainty ought to be cleared out in the next examples.

\begin{example}\label{Fuchsian-Hermite-Laguerre-etc}
The change of a variable $x=e^z$ gives us the differential operators
	$$
	\frac{d}{dx}=e^{-z}\frac{d}{dz} \quad\textnormal{and}\quad
	\frac{d^2}{dx^2}=-e^{-2z}\frac{d}{dz}+e^{-2z}\frac{d^2}{dz^2}.
	$$
This enables us to transform Frei's equation \eqref{Frei-eqn} to the equation
	\begin{equation}\label{Fuchsian}
	x^2g''(x)+(x+1)g'(x)+\alpha g(x)=0
	\end{equation}
of Fuchsian type.
But this equation is special in the sense that it permits regular singular points at the origin and at the point at infinity.
Because of this fact, equation \eqref{Fuchsian} admits a polynomial solution $g(x)=\sum_{n=0}^{m}a_nx^n$, where the coefficients $a_n$ are given by the recurrence equation
	$$
	a_{n+1}=-\frac{n^2+\alpha}{n+1}a_n, \quad n=0, 1, \ldots, m-1,
	$$
together with $\alpha=-m^2$.
Taking $a_0=1$, we have the solution $f(z)=g(e^z)$ in the previous example with $C_n=a_n$.
\end{example}

\begin{example}
The discussion in the previous example leads us to look for more second order linear differential equations with polynomial coefficients and (Laurent) polynomial solutions. Monographs such as~\cite{Beal-Wong} or \cite{Kristensson}
contain detailed definitions and results on second-order equations of Fuchsian type, that is, hypergeometric differential equations
	$$
	z(z-1)u''(z)+[(a+b+1)z-c]u'(z)+abu(z)=0
	$$
with at most three regular singular points and with parameters $a,b, c$.
In order for Laurent polynomial solutions to occur, two of the regular singular points $z=0, 1, \infty$ need to be confluent into $0$ and $\infty$.

As polynomial solutions of such confluent hypergeometric differential equations, Hermite polynomials $H_n$
and Laguerre polynomials $L_n^{(\alpha)}$
are known to satisfy the second-order differential equations
	\begin{equation}\label{Hermite}
	H_n''(x)-2xH_n'(x)+2nH_n(x)=0
	\end{equation}
and
	\begin{equation}\label{Laguerre}
	x[L_n^{(\alpha)}]''(x)+(\alpha+1-x)[L_n^{(\alpha)}]'(x)
	+nL_n^{(\alpha)}(x)=0,
	\end{equation}
respectively, see \cite{Beal-Wong}.
Again by the change of a variable $x=e^z$, we may transform the Hermite differential equation \eqref{Hermite} to
	\begin{equation} \label{hermite}
	f''(z)-(2e^{2z}+1)f'(z)+2ne^{2z}f(z)=0
	\end{equation}
with $f(z)=H_n(e^z)$, and the Laguerre differential equation
\eqref{Laguerre} to
	\begin{equation} \label{laguerre}
	f''(z)-(e^{z}-\alpha)f'(z)+ne^{z}f(z)=0
	\end{equation}
with $f(z)=L_n^{(\alpha)}(e^z)$.
Note that both of the coefficient functions in \eqref{hermite} as well as in \eqref{laguerre} contain the exponential function, while only one of the coefficient functions in Frei's equation \eqref{Frei-eqn} contains the exponential function.
\end{example}

We have seen how well the change of variables work for Hermite or Laguerre equations in producing second-order linear differential equations with periodic coefficients. The discussion goes parallel with $\alpha \in \C\setminus\{0\}$ for the change of variables $x=e^{\alpha z}$.
Other changes of variable would not bring our desired second-order linear differential equations having exponential polynomials for coefficients and solutions.

The general method originates in Wittich's paper~\cite{Wittich1967}
published in 1967, where, in relation to Frei's equation,  subnormal solutions of
	$$
	w''+p(e^z)w'+q(e^z)w=0
	$$
are considered for $p(t), q(t)\in\C[t]$. By $t=e^z$ and $v(t)=w(z)$, this equation transforms into
	\begin{equation}\label{Wittich-eqn}
	t^2v''+t\bigl(1+p(t)\bigr)v'+q(t)v=0.
	\end{equation}
Then the main result in \cite{Wittich1967} claims that the existence of a subnormal solution of the former equation is equivalent to the existence of a solution $v(t)=t^{\sigma}P(t)$ of the latter equation with $P(t)\in\C[t]$ and with $\sigma$ satisfying $\sigma^2+p(0)\sigma+q(0)=0$.

Note that the modified Frei's equation \eqref{Fuchsian} corresponds to modified Wittich's equation \eqref{Wittich-eqn} with $p(x)=1/x$, so that the formula for $\sigma$ is obtained after the substitution $z=-\zeta$ and $f(z)=g(\zeta)$ in~\eqref{Frei-eqn} to obtain $g''(\zeta)-e^{\zeta}g'(\zeta)+\alpha g(\zeta)=0$.
In this case, one has $p(0)=0$, $q(0)=\alpha$ and $\sigma^2=-\alpha$.

Shimomura~\cite{Shimomura2002} uses Wittich's method to obtain oscillation results for $n$-th order linear differential equations  of the form
	$$
	w^{(n)}+R_{n-1}(e^z)w^{(n-1)}+\cdots+R_1(e^z)w'+R_0(e^z)w=0,
	$$
where the coefficients $R_j(t)$ are rational functions for $0\leq j\leq n-1$. In other words, the coefficients are meromorphic periodic functions. It is very important that each $R_j(t)$ is permitted to have poles other than $z=0, \infty$ so that the situation is not restricted to Laurent polynomials only.

Next, we look at concrete examples of subnormal solutions for third-order linear differential equations. These examples form the starting point for the
discussions on dual exponential polynomials in Section~\ref{dual-sec} below.


\begin{example}\label{third-order-example}
\textnormal{(\cite{WGH})} (a) The function $f(z) = 16 - 27e^{-2z} + 27e^{-3z}$ satisfies
	$$
	f''' + 9^{-1}(9 + 9e^z + 4e^{2z})f'' - 5f' + 3f = 0.
	$$
The coefficient $9^{-1}(9 + 9e^z + 4e^{2z})$ and the solution $f$ both have two transcendental terms.

(b) The function $f(z)=\exp\left(z^2/2 + z\right) + z + 1$ is a second order solution of
    $$
    f''' + \left(\exp\left(-z^2/2 - z\right) - z - 1\right)f''
    - f' - (z + 1)f = 0.
    $$
\end{example}

We proceed to briefly discuss subnormal solutions to non-homogeneous linear
differential equations. The paper \cite{GE} by Gundersen and Steinbart
initiated the research in this direction, and is cited by many further
papers. Not surprisingly, the subnormal solutions turn out to be
exponential polynomials. The paper \cite{GE} contains seven interesting
examples of such solutions. We settle for recalling one of the main
results in \cite{GE}.

\begin{theorem}
\textnormal{(Gundersen-Steinbart \cite{GE})}
Suppose that $f$ is a subnormal solution of
	$$
	f''+P_1(e^z)f'+P_2(e^{z})f=Q_1(e^z)+Q_2(e^{-z}),
	$$
where $P_1,P_2,Q_1,Q_2\in \C[z]$ and $\deg(P_2)>\deg(P_1)$. Then there
exist $S_1,S_2\in \C[z]$ such that
	$$
	f(z)=S_1(e^z)+S_2(e^{-z}).
	$$
\end{theorem}

We move on to considering non-homogeneous equations of the form
	\begin{equation}\label{GSW-eqn}
	f^{(n)} + P_{n-1}(z)f^{(n-1)} + \cdots + P_0(z)f = h(z)e^{Q(z)},
	\end{equation}
where $n \in \N$, the coefficients $P_0(z), P_1(z), \ldots , P_{n-1}(z)$ are polynomials, $P_0(z)\not\equiv 0$, $h(z) \not\equiv 0$ is entire,
$Q(z)$ is a non-constant polynomial, and $\lambda(h) = \rho(h) < \deg (Q)$ holds. In~\cite{GSW2} Gundersen, Steinbart and Wang give a condition on the leading coefficient of the polynomial $Q(z)$ yielding the maximum number of linearly independent solutions $f$ of \eqref{GSW-eqn} satisfying $\lambda (f) < \rho(f)$, and determine those equations for which the maximum number of such solutions is attained. To illustrate the sharpness of their main results,
several examples of exponential polynomial solutions of
\eqref{GSW-eqn} with $h$ being a polynomial or an exponential polynomial
are given in \cite{GSW2}.

The paper \cite{GSW2} is motivated by Gao's paper~\cite{Gao},
published in 1989, which concerns with differential equations of the form 	
	\begin{equation}\label{Gao-eqn}
	f''+a_0(z)f=P_1(z)e^{P_0(z)}
	\end{equation}
for polynomials $a_0, P_0, P_1\not\equiv 0$ under the assumption
that $\deg (P_0)<1+\frac{1}{2}\deg (a_0)$. It is proved in \cite{Gao}
that if $\deg (P_1) < \deg (a_0)$, then every solution $f$ of \eqref{Gao-eqn} satisfies
	\begin{equation} \label{Gao}
	\lambda(f)=\rho(f)=1+\frac{1}{2}\deg (a_0),
	\end{equation}
while if $\deg (P_1) \geq \deg (a_0)$, then either $f$ satisfies \eqref{Gao} or $f$ is of the form $f = Qe^{P_0}$, where $Q$ is a polynomial.
Six examples of exponential polynomial solutions of \eqref{Gao-eqn}
are given in \cite{Gao}. Gao's results are refined in~\cite{GSW3} along
with many further interesting examples.

It is natural to consider a more general equation, where
the right-hand side of \eqref{GSW-eqn} is replaced with a transcendental entire function $H(z)$ of finite order, that is,
	\begin{equation}\label{GSW-gen}
	f^{(n)} + P_{n-1}(z)f^{(n-1)} + \cdots + P_0(z)f =H(z).
	\end{equation}	
All solutions $f\not\equiv 0$ of \eqref{GSW-gen} are clearly entire and of finite order $\rho(f)\geq \rho(H)$ \cite{GSW3}. Further, it is proved in \cite[Theorem~2.6]{GSW3} that
	$$
	\rho(f)-\lambda(f)\leq \rho(H)-\lambda(H).
	$$
In particular, if $\rho(H)=\lambda(H)$ (that is, the value zero
is not a Borel exceptional value for~$H$), then $\rho(f)=\lambda(f)$.

\begin{theorem}\label{def-value-thm}
We have $\delta(0,f)\leq \delta(0,H)$ for any solution $f\not\equiv 0$ of \eqref{GSW-gen}. In particular, if
$\delta(0,H)=0$ (that is, the value zero is not a deficient
value for $H$), then $\delta(0,f)=0$. If $f$ and $H(z)$ are both
exponential polynomials, then
	\begin{equation}\label{perimeters1}
	C(\co(W_H))/C(\co(W_{0H})) \leq C(\co(W_f))/C(\co(W_{0f})).
	\end{equation}
\end{theorem}

\begin{proof}
Since $f$ is of finite order, we have
	\begin{eqnarray*}
	T(r,H)&=&m(r,H)=m\left(r,\left(\frac{f^{(n)}}{f}
	+P_{n-1}\frac{f^{(n-1)}}{f}+\cdots +P_0\right)\cdot f\right)\\
	&\leq & m(r,f)+O(\log r)=T(r,f)+O(\log r)
	\end{eqnarray*}
and
	\begin{eqnarray*}
	m\left(r,\frac{1}{f}\right)
	&=&m\left(r,\left(\frac{f^{(n)}}{f}
	+P_{n-1}\frac{f^{(n-1)}}{f}+\cdots +P_0\right)\cdot\frac{1}{H}\right)\\
	&\leq & m\left(r,\frac{1}{H}\right)+O(\log r).
	\end{eqnarray*}
Since $H(z)$ is transcendental, it follows from \eqref{GSW-gen}
that $f$ is transcendental as well. This gives us $T(r,f)/\log r\to\infty$ and $T(r,H)/\log r\to \infty$ as $r\to\infty$. Thus
the first assertion follows from the two inequalities above. This
implies
	$$
	\limsup_{r\to\infty}\frac{N(r,1/H)}{T(r,H)}\leq
	\limsup_{r\to\infty}\frac{N(r,1/f)}{T(r,f)}.
	$$
The second assertion follows from this and Theorem~\ref{Stein-thm1}.
\end{proof}

\begin{example}
Regarding Theorem~\ref{def-value-thm}, if $P(z)\not\equiv 0,1$ is
any polynomial, then the function $f(z)=e^{2z}+e^z$ is a solution of
	$$
	f''-\left(P(z)+2\right)f'+2P(z)f=\left(P(z)-1\right)e^z,
	$$
and in this case we have a strict inequality $\delta(0,f)=1/2<1=\delta(0,H)$. A strict inequality may hold even if
$H(z)$ has infinitely many zeros. Indeed, 	
the function $f(z)=e^{2z^2}+e^{z^2}+e^z$ solves
	$$
	f''-4zf'+(4z-1)f=(4z+3)e^{2z^2}-(4z^2-4z-1)e^{z^2},
	$$
and in this case $\delta(0,f)=0<1/2=\delta(0,H)$. Finally,
the function $f(z)=e^{2z^2}+e^{z^2}$ solves
	$$
	f''+f'-(4/3)(z+1)f=(2/3)(6z^2+z+1)(e^{z^2}+4e^{2z^2}),
	$$
and in this case the equality $\delta(0,f)=1/2=\delta(0,H)$ holds.
These examples also illustrate that either an equation or a strict inequality in \eqref{perimeters1} may hold.
\end{example}

Finally we note that, in some cases, the change of variable method works for non-homogeneous equations also. For example, the following result is proved in \cite{CC}:
{\it Let $\lambda\in\C\setminus\{0\}$, and let $H(z)$ be a
non-vanishing polynomial. Then the differential equation $4y''(t) - \lambda^2 y(t) = H(t)$ has a special solution $y_0(t)$, which is a non-vanishing polynomial.} The usual change of a variable $t=e^z$ shows that $f(z)=y_0(e^z)$ is a special solution of
	$$
	f''(z)-f'(z)-\frac{\lambda^2}{4}e^{2z}f(z)=\frac{1}{4}e^{2z}H(e^z).
	$$

\subsection{Dual exponential polynomials}\label{dual-sec} 

As already mentioned in Section~\ref{crg-solutions-sec}, there are many examples of linear differential equations with coefficients and solutions being exponential polynomials \cite{dual2, GOP, WGH}. In this subsection we will take a closer look at the papers \cite{dual2, WGH}.

Examples \ref{Frei-example} and \ref{third-order-example} above, as
well as many further examples in~\cite{dual2, GOP, WGH} of the same nature, all illustrate a certain duality behavior between the solution~$f$ and one of the coefficients. This leads to the following definition.

\begin{definition}
\textnormal{(\cite{dual2, WGH})}
Let an exponential polynomial $f$ be given in the normalized form \eqref{normalized-f}.
If the non-zero frequencies ${w}_1,\ldots,{w}_m$ of $f$ all lie on a fixed ray $\arg(w)=\theta$, then $f$ is called a
\emph{one-sided exponential polynomial}.
If $g$ is another
one-sided exponential polynomial of the same order $q$ as $f$ such that the non-zero frequencies of $g$ all lie on the opposite ray $\arg(w)=\theta+\pi$, then $f$ and $g$ are called \emph{dual exponential polynomials}.
\end{definition}

Examples \ref{Frei-example} and~\ref{third-order-example} contain three examples on the duality
between a solution and one of the coefficient functions, both being
exponential polynomials. The concept of duality was introduced to describe this
phenomenon for linear differential equations.

\begin{remark}
In the original references~\cite{dual2, WGH}, the term \emph{simple exponential polynomial} was used instead of \emph{one-sided exponential polynomial}.  But here in this survey we have already used the term \emph{simple} for exponential sums in ${\cal E}$ in a different meaning.
To avoid any confusion, we thus give up on the term \emph{simple} in defining dual exponential polynomials.
\end{remark}

Next we state a general result about dual exponential polynomials, which
is motivated by Example~\ref{third-order-example} and many other examples
of the same nature.

\begin{theorem}\textnormal{(\cite{WGH})}\label{general.theorem}
Suppose that $f$ is an exponential polynomial solution of \eqref{lde}, where for precisely one index $\mu\in\{1,\cdots,n-1\}$, $a_\mu(z)$ is a transcendental exponential polynomial such that for all $j\neq\mu$, $a_j(z)$ is an exponential polynomial satisfying
$\rho(a_j)<\rho(a_\mu)$.
Then either $f$ is a polynomial of degree $\leq \mu-1$ or $f$ and $a_\mu(z)$ are dual exponential polynomials of order $q\in\N$.
In the latter case, $f$ has the normalized representation
	\begin{equation}\label{fsolutioncd.eq}
	f(z)=S(z)+F_1(z)e^{w_1z^q}+\cdots +F_m(z)e^{w_mz^q},
	\end{equation}
where the non-zero constants $w_j$ have the same argument, $m\in\N$, and $S(z)$ is a polynomial satisfying $S(z)\not\equiv 0$ and $\deg(S)\leq \mu-1$.

In particular, if $\rho(a_\mu f^{(\mu)})<q$ and $a_j(z)$ are polynomials for $j\neq \mu$, then
	\begin{equation}\label{sp.eq}
	f(z)=S(z)+Q(z)e^{P(z)}\quad\text{and}\quad a_\mu(z)=R(z)e^{-P(z)},
	\end{equation}
where $P(z), Q(z), R(z)$ are polynomials and $\deg(P)=q$.
\end{theorem}

\begin{remark}
By a careful examination of the proof of \cite[Theorem~2]{WGH}, we find
that the assumption
	``$a_j(z)$ is an exponential polynomial
	satisfying $\rho(a_j)<\rho(a_\mu)$ for all $j\neq\mu$''
can be weakened to
	``$a_j(z)$ is an entire function
	satisfying $T(r,a_j)=S(r,a_\mu)$ for all $j\neq\mu$''. Indeed, one needs to use Niino's theorem \cite[Theorem~1.53]{YY}
	instead of Borel's theorem when proving \cite[Lemma~2]{WGH}.
\end{remark}

We proceed to state a stronger version of Theorem~\ref{general.theorem} in the second order case under weaker assumptions.
For this purpose, we need two definitions.

\begin{definition}\label{commensurable}
\textnormal{(Langer \cite[p.~214]{Langer})}
Let $f$ in \eqref{normalized-f} be a
one-sided exponential polynomial.
If there exists a constant $w\in\C\setminus\{0\}$ such that $w_j/w$ is a positive integer for every $j=1,2,\ldots,m$, then the (non-zero) frequencies of $f$ are said to be \emph{commensurable}, and $w$ is called a \emph{common factor}.
\end{definition}

As observed in \cite{dual2}, $f(z)=e^{\pi z}+3e^{2\pi z}+ze^{3\pi z}$ and $g(z)=e^{4iz}+e^{6iz}$ are
one-sided exponential polynomials, both of their frequencies are commensurable, and examples of common factors are $\pi,\pi/2$ for $f$ and $i,2i$ for $g$.
Thus a common factor is not unique.
In fact, $g\in{\cal E}$ is a simple exponential sum since $\dim\mathrm{supp}(g)=1$, while $f(z)$ can be treated similarly.
Hence, $f\in{\cal E}$ given by~\eqref{normalized-f} is simple, if its frequencies $w_j$ $(1\leq j\leq m)$ are commensurable.

\begin{definition}\label{s-commensurable}
\textnormal{(\cite{dual2})}
Suppose that $f$ and $g$ are dual exponential polynomials with commensurable frequencies
$\{w_j\}$ $(j>0)$ and $\{\lambda_i\}$ $(i>0)$, respectively, sharing the same common factor $w$ but with opposite signs.  If the points $w_j+\lambda_i$
are on one ray including the origin for all $i,j>0$, then
$f$ and $g$ are called \emph{strongly dual exponential polynomials}.
\end{definition}

As observed in \cite{dual2}, the functions $f(z)=1+ze^z+2e^{3z}$ and $g(z)=1-e^{-z}$ are strongly dual exponential polynomials, while $f(z)$ and $h(z)=g(z)+2z^2e^{-2z}$ are not.
We are now ready to state an improvement of Theorem~\ref{general.theorem} in the second order case.

\begin{theorem}\label{dual-thm}
\textnormal{(\cite{dual2, WGH})}
Suppose that $f$ and $A(z)$ in \eqref{ldeAB} are transcendental exponential polynomials, and that $B(z)$ is an entire function satisfying $T(r,B)=o(T(r,A))$.
Then the following assertions~hold.
\begin{itemize}
\item[\textnormal{(a)}]
$f$ and $A(z)$ are dual exponential polynomials of order $q\in\N$, $f$ has the normalized representation
	\begin{equation*}
	f(z)=c+F_1(z)e^{w_1z^q}+\cdots +F_m(z)e^{w_mz^q},\quad m\in\N,\ c\in\C\setminus\{0\},
	\end{equation*}
and $B(z)$ is an exponential polynomial of order $\rho(B)\leq q-1$.
\item[\textnormal{(b)}]
The frequencies of $f$ are commensurable if and only if the frequencies of $A(z)$ are commensurable.
In both cases, $f$ and $A(z)$ are strongly dual exponential \mbox{polynomials.}
\item[\textnormal{(c)}]
If $\rho(Af')<q$, then $q=1$ and
    \begin{equation}\label{q=1}
   A(z)=a e^{-wz}, \quad B(z)=-w^2 \quad\text{and} \quad f(z)=c\left(1+\frac{w}{a}e^{wz}\right),
    \end{equation}
where $w=w_1$ and $a \in\C\setminus\{0\}$.
\end{itemize}
\end{theorem}

Regarding the assumption $\rho(Af')<\rho(f)$ in Theorem~\ref{dual-thm}(c), we recall Example~\ref{Frei-example}, which shows that $q=1$ may hold even if $\rho(Af')=\rho(f)$. Next we give another condition implying $\rho(f)=\rho(A)=q=1$.

\begin{theorem}\label{one-term}
\textnormal{(\cite{dual2})}
Suppose that $f(z)=F_0(z)+F_1(z)e^{wz^q}$ is a solution of \eqref{ldeAB}, where $A(z)$ is an exponential polynomial and $B(z)$ is an entire function satisfying $T(r,B)=o(T(r,A))$.
Then $q=1$, and there are constants $c,b\in\C\setminus\{0\}$ and a non-trivial polynomial $P(z)$ such that
	$$
	f(z)=c+be^{wz},\ A(z)=\frac{b}{c}P(z)-w+P(z)e^{-wz}\
	\text{and}\ B(z)=-\frac{wb}{c}P(z).
	$$
\end{theorem}

\begin{problem}
Under the assumptions of Theorem~\ref{dual-thm}, is it always true that $q=1$ and $B(z)$ is a polynomial?
\end{problem}

As observed in~\cite{dual2}, this question is very fragile in the sense that the desired conclusion is not valid if a minor modification in the assumptions of Theorem~\ref{dual-thm} is performed.
Indeed, the function $f(z)=e^{z^2}+1$ satisfies the differential equations
	\begin{equation*}
	\begin{split}
	f''+\left(\frac{e^{-z^2}-1}{2z}-2z\right)f'-f &= 0,\\
	f''-\frac{e^{-z^2}(z-1)+4z^2+z+1}{2z}f'+(z-1)f &= 0,
	\end{split}
	\end{equation*}
where the transcendental coefficients are entire exponential polynomials but with rational multipliers because $z=0$ is a removable singularity for both.

\subsection{Non-linear ODE's} 

Next we  discuss results on Riccati equations \cite{BGL},
Malmquist type equations \cite{GL}, and Tumura-Clunie type equations \cite{TC,Zhang1,Zhang2}. The reason why we focus on these papers is that each of them contains a very nice reference list on the corresponding study and provides
with many interesting examples in terms of exponential sums.
In addition, these papers along with \cite{Gundersen2019} contain several
open questions that have a connection to exponential sums.   	

We begin with Riccati equations
	\begin{equation}\label{Riccati-eqn}
	f'=a_0(z)+a_1(z)f+a_2(z)f^2,\quad a_2(z)\not\equiv 0.
	\end{equation}
The following problem is considered in the paper \cite{BGL} by Bank, Gundersen and Laine:
\begin{quote}
{\it Under which conditions a Riccati equation with meromorphic coefficients, which are not all entire, actually admits meromorphic solutions in the complex plane?}
\end{quote}
Here ``meromorphic function'' means meromorphic in the whole complex plane.
An elementary transform by Wittich \cite{BGL, Laine},
	$$
	f=\frac{1}{a_2(z)}u-\frac{a_1(z)}{2a_2(z)}-\frac{a_2'(z)}{2a_2(z)^2},
	$$
transforms \eqref{Riccati-eqn} to the special form
	\begin{equation}\label{Riccati-special}
	u'=A(z)+u^2
	\end{equation}
with $A(z)$ meromorphic. So, without loss of generality, we may focus on \eqref{Riccati-special}.

\begin{theorem}
\textnormal{(Bank-Gundersen-Laine \cite{BGL})}
\begin{itemize}
\item[\textnormal{(a)}] If $A(z)$ is a transcendental entire function of finite order, then \eqref{Riccati-special} admits at most two distinct meromorphic solutions of finite order.

\item[\textnormal{(b)}] If \eqref{Riccati-special} admits a meromorphic solution having at least one pole of multiplicity~$\geq 2$, then \eqref{Riccati-special} admits at most two distinct meromorphic solutions.
\end{itemize}
\end{theorem}

The following example illustrates the sharpness of Part (a).

\begin{example}
\textnormal{(Bank-Gundersen-Laine \cite{BGL})}
Equation \eqref{Riccati-special} with $A(z)=-\frac{1}{4}-\frac{e^{2z}}{4}$ has two meromorphic solutions $f(z)=\frac{1\pm e^z}{2}$ of finite order.
This example arises from the equation $u'(t)=-\frac{1}{4t}-\frac{t}{4}+\frac{1}{t}u(t)^2$ and its rational solutions  $u(t)=\frac{1\pm t}{2}$ through the change of variable $z=\log t$ and $u(t)=f(\log t)$.
This equation also permits the finite-order solution $u(t)=\frac{1}{2}(1+t)-t[1+Ce^{-t}]^{-1}$, $C\in\C$, and the corresponding solution $f(z)=u(e^z)$ is of infinite order but of hyper-order~$1$.

Equation \eqref{Riccati-special} with  $A(z)=-\frac{e^z+1}{(e^z-1)^2}$ has
a meromorphic solution $f(z)=\frac{1}{e^z-1}$ of finite order.
This example arises from the equation $u'(t)=-\frac{t+1}{t(t-1)^2}+\frac{1}{t}u(t)^2$ and its rational solution $u(t)=\frac{1}{t-1}$ similarly as above.
This rational $u(t)$ is not a Laurent polynomial, while the above $A(z)$ and $f(z)$ are ratios of two (simple) exponential sums.
\end{example}

The conclusion of Part (b) does not say anything about possible rational solutions.
The following example illustrates the sharpness of Part (b), and shows
that rational solutions may appear.

\begin{example}
\textnormal{(Bank-Gundersen-Laine \cite{BGL})}
The Riccati differential equation
	$$
	u'(t)=\frac{2\alpha-\alpha^2}{4t^2}-\frac{1}{4t^{2\alpha}}+u(t)^2
	$$
admits exactly two rational (and Laurent polynomial) solutions $u(t)=-\frac{\alpha}{2t}\pm \frac{1}{2t^{\alpha}}$, where $\alpha\geq 2$ is an integer. This implies that the equation
	$$
	f'(z)=\frac{2\alpha-\alpha^2}{4}e^{-z}
	-\frac{1}{4}e^{-(2\alpha-1)z}+e^zf(z)^2
	$$
admits the solutions $f(z)=-\frac{\alpha}{2}e^{-z}\pm \frac{1}{2}e^{-\alpha z}$, which are simple elements in~${\cal E}$.

The Riccati differential equation
	$$
	u'(t)=3t^{-4} - t^{-6} +u(t)^2
	$$
admits exactly one meromorphic solution, namely the Laurent polynomial
$u(t)=-t^{-3}$. This implies that the equation
	$$
	f'(z)=3e^{-3z} - e^{-5z} +e^zf(z)^2
	$$
admits an exponential-function solution $f(z)=-e^{-3z}$.

Note that in both examples above, the latter equation may have another meromorphic solution $g(z)$, which is not translated back to a meromorphic solution $v(t)$ of the corresponding former equation by means of the
change of variable $z=\log t$.
\end{example}

\begin{problem}
If $A(z)$ is an exponential polynomial or a ratio of two exponential polynomials, what solutions are available to the equation $f'(z)=A(z)+f(z)^2$? Can some of those solutions be exponential polynomials or possibly rations of two exponential polynomials?
\end{problem}

\begin{remark}
Recall that the Riccati equation \eqref{Riccati-special} can be
transformed to the homogeneous linear differential equation $v''(t)+A(t)v(t)=0$ by means of $u(t)=-\frac{v'(t)}{v(t)}$.
Therefore the above problem can be transformed into this normalized
linear equation in $v(t)$. The function $f(z)=u(e^z)$ satisfies $f'(z)=e^zA(e^z)+e^zf(z)^2$. Define a function $g$ by $f(z)=-e^{-z}\frac{g'(z)}{g(z)}$. Then $g$ is a solution of $g''(z)-g(z)+e^{2z}A(e^z)g(z)=0$. When both $A(t)$ and $u(t)$ are rational (and at most Laurent polynomials), this transform once again takes a rational case to
a periodic case.
\end{remark}

We move on to considering Malmquist type equations
	\begin{equation} \label{Malmquist}
	u'=\frac{\sum_{k=0}^{n}A_k(z)u^k}{\sum_{k=0}^{m}B_k(z)u^k },
	\end{equation}
where the coefficients $A_k(z)$ and $B_k(z)$ are meromorphic functions and the right-hand side of~\eqref{Malmquist} is irreducible as a rational function in~$u$.
Malmquist~\cite{Malmquist} proved that if a differential equation of the form~\eqref{Malmquist} with rational coefficients $A_k(z)$, $B_k(z)$ admits a transcendental meromorphic solution, then this equation reduces to a Riccati equation or to a linear equation. See~\cite{GL} for extensions of Malmquist's theorem.

Theorem~2 in \cite{GL} focuses on more specified differential equations
	\begin{equation} \label{non-Riccati}
	f'=\sum_{k=0}^{n} \exp(-q_k z)P_k(e^z)f^k,
	\end{equation}
where each $q_k\geq 0$ is an integer, each $P_k(z)$ is a polynomial in~$z$, $P_n(z)=z^q$, and $n\geq 3$. Here the coefficients are obviously in~${\cal E}$. Then \cite[Theorem~2]{GL} claims that any meromorphic solution $f$ is a simple exponential sum $f(z)=e^{-\alpha (z/m)}Q\left(e^{z/m}\right) \in {\cal E}$, where $\alpha\geq 0$ and $m>0$ are integers, and $Q$ is a polynomial.  This answers ``almost yes'' to a question posed by E.~Hille:
\begin{quote}
{\it Is a meromorphic solution of a non-Riccati differential equation of the form~\eqref{Malmquist} necessarily a rational function of the coefficients $A_k(z)$, $B_k(z)$?}
\end{quote}
A meromorphic solution of a non-Riccati equation~\eqref{non-Riccati}, whose coefficients are all Laurent polynomials in $e^z$, is necessarily a Laurent polynomial in $e^{z/m}$ for some positive integer $m$. However, examples in \cite{GL} show that the answer
to Hille's question in general is no.

Hille's question naturally leads to the following question \cite{GL}:
\begin{quote}
{\it Is a meromorphic solution $f$ of~\eqref{Malmquist}, where $A_k$ and $B_k$ are rational functions in $e^z$, necessary of the form $f(z)=R(e^{\beta z})$ for some rational function $R(w)$ and for some $\beta\in\Q$?}
\end{quote}
If \eqref{Malmquist} reduces to a Riccati equation or to a first-order
linear equation, the answer to this question is in the negative.
See \cite[Example~5.3]{BGL} and the equation $f'=e^zf$ for counterexamples,
respectively.

Four questions on the solutions of~\eqref{Malmquist} are stated in \cite{Gundersen2019}. Question~6.1 in \cite{Gundersen2019} is about the possibility of non-Riccati and non-linear equation~\eqref{Malmquist} to possess infinitely many distinct meromorphic solutions. Riccati equations
and linear differential equations are known to have this property.
It is remarked in \cite{Gundersen2019} that the answer to the above
question is no if all the coefficients $A_k$ and $B_k$ are rational.
Question~6.2 in \cite{Gundersen2019} considers equations of the form
$f'=R(e^z, f)$, where $R(u,v)$ is rational in both arguments such that
the equation does not reduce to neither a Riccati equation nor to a linear equation. The question is if the meromorphic solutions $f$ are always of the form $f(z)=S(e^{cz})$ for some rational function $S(w)$ and
for some constant $c\in\Q$. To illustrate this question, two examples are given in \cite{Gundersen2019}, namely
	$$
	f(z)=e^{-z/2}\in {\cal E}
	\quad\textnormal{and}\quad R(u,v)=-\frac{1}{2}uv^3,
	$$
	$$
	f(z)=\frac{e^z}{z(e^z-1)}\not\in{\cal E}
	\quad\textnormal{and}\quad R(u,v)=v/(1-u)+(1/u-1)v^2.
	$$
Questions~6.3 and~6.4 in \cite{Gundersen2019} concern the number of distinct or linearly independent meromorphic solutions of the equation
	$$
	u'=P_0(t)+P_1(t)u+P_2(t)u^2+\cdots +P_n(t)u^n, \quad n\geq 3,
	$$
where each $P_k(t)$ is a polynomial in~$t$ with $P_n(t)\not\equiv 0$.
Malmquist's theorem states that every meromorphic solution $u$ of this equation must be a rational function. If $u$ is such a solution, then
$f(z)=u(e^z)$ solves
	$$
	f'=e^{-z}P_0(e^z)+e^{-z}P_1(e^z)f+e^{-z}P_2(e^z)f^2+\cdots
	+ e^{-z}P_n(e^z)f^n.
	$$

Differential equations of Tumura-Clunie type \cite{TC} are of the form
	\begin{equation}\label{TC-eqn}
	f^n+P(z,f)=h(z),
	\end{equation}
where $n\geq 2$ is an integer, $h$ is a meromorphic function in $\C$, and
$P(z,f)$ is a differential polynomial in $f$ and its derivatives with
small meromorphic coefficients (compared to $f$). The name for these
equations arises from a theorem of Tumura published in 1937 the proof of
which was completed by Clunie in 1962 \cite{TC}. The special case of
\eqref{TC-eqn} of the form
	$$
	f^3+\frac34 f''=h(z),
	$$
where $h(z)=-\frac14 \sin(3z)$, was studied by Yang and Li in 2004 \cite{CC-Li}. They proved that this equation has precisely three entire solutions, namely
	$$
	f_1(z)=\sin z,\quad f_2(z)=\frac{\sqrt{3}}{2}\cos z-\frac12\sin z,
	\quad f_3(z)=-\frac{\sqrt{3}}{2}\cos z-\frac12\sin z.
	$$
The known solutions of the general equation \eqref{TC-eqn} are typically exponential polynomials
with rational coefficients, see \cite{TC} and the references therein.
Apart from exponential polynomials, entire solutions of \eqref{TC-eqn} of every half-integer order also exist \cite{TC}. In these examples, the function $h$ is a solution of a second order linear differential equation with rational coefficients.  Entire solutions of arbitrary positive rational order are not known, but such solutions could possibly be found by allowing $h$ to be a solution of a higher order linear differential equation with rational coefficients.

\begin{problem}
For an arbitrary $\rho\in\Q_+$, construct an equation of the form \eqref{TC-eqn}
having an entire solution $f$ of order $\rho(f)=\rho$.
\end{problem}

\subsection{A glimpse at binomial equations and PDE's}

We recall a theorem of Hayman from 1959.

\begin{theorem}
\textnormal{(Hayman \cite{GY})}
If $f$ is a
meromorphic function in $\C$ such that $f$ and $f''$ have only finitely
many zeros and poles, then $f(z)=R(z)e^{P(z)}$, where $R(z)$ is a rational function and $P(z)$ is a polynomial. Of these the only functions for
which $f$ and $f''$ have no zeros are $e^{az+b}$ and $1/(az+b)^n$, where
$n$ is a positive integer and $a, b\in\C$.
\end{theorem}

Langley has proved that the assumption on finitely many zeros can be
removed. Hayman himself showed that if $f$ has no zeros, then $f=1/g$,
where $g$ is an entire function such that
	$$
	f''(z)=\frac{2g'(z)^2-g(z)g''(z)}{g(z)^3}.
	$$
This leads to a question when a differential polynomial $2g'(z)^2-g(z)g''(z)$
has no zeros for an entire $g$. This question was studied by Langley and
Mues \cite{GY}. This background motivated Gundersen and Yang to study
binomial differential equations of the form $ff''-a(f')^2=H$ in \cite{GY}.
The main result in \cite{GY} concerns the special case
	\begin{equation}\label{GY-eqn}
	ff''-a(z)(f')^2=b(z)e^{2c(z)},
	\end{equation}
where $a(z),b(z),c(z)$ are polynomials such that $b(z)\not\equiv 0$
and $c(z)$ is non-constant. If $b(z)$ is non-constant, then all entire
solutions $f$ are of the form $f(z)=p(z)e^{c(z)}$ for $p(z)\in\C[z]$,
while if $a(z)$ and $b(z)$ are constants, that is, if $f f'' - a(f')^2$ is zero-free and of finite order, then
	$$
	f(z)=\alpha e^{\lambda z}, \ \alpha e^{\lambda z^2+\gamma z},
	\ (\alpha z+\beta)e^{\lambda z}, \ \alpha e^{2\lambda z}+\beta,
	\ \alpha e^{\gamma z}+\beta e^{(2\lambda -\gamma)z}
	$$
are the only possible entire solutions for suitable constants $\alpha, \beta, \gamma$ and $\lambda$. There are five open questions about the entire
solutions of \eqref{GY-eqn} and of related equations in \cite{GY}.
Three examples in \cite{GY} show that functions of the form $f(z)=\alpha e^{\lambda z}+\beta e^{\mu z} \in {\cal E}$ solve each of these equations.
The closing question in \cite{GY} is whether these equations have entire solutions other than $\alpha e^{\lambda z}+\beta e^{\mu z}$ and $p(z)e^{c(z)}$.

The relation of irreducible exponential sums in~${\cal E}$ to differential equations may go beyond ODE's, while PDE's deviate from our objectives
in this survey. Therefore, we briefly consider an elementary example just to see how difficult it is to deal with irreducible exponential sums as solutions of ODE's and PDE's.

\begin{example}
Consider the exponential sum $f(z)=1+e^{(- \log 2) z}+e^{(-\log 3) z}+e^{2(-\log 2) z}$. This irreducible element of~${\cal E}$ is the partial sum
$1+1/2^z+1/3^z+1/4^z$ of the Riemann zeta-function  and solves the fourth-order homogeneous linear differential equation
	$$
	D(D+\log 2)(D+\log 3)(D+\log 4)f(z)=0, \quad D=\frac{d}{dz}.
	$$
It is also a straightforward calculation to see that $f$ solves the
second-order non-homogeneous linear differential equation
	\begin{equation} \label{4thPS}
	f''(z)+(\log 12)f'(z)+(\log 3)(\log 4)f(z)=(\log 3)(\log 4) + (\log 2)		\left(\log\frac32\right)\frac{1}{2^z}.
	\end{equation}
Generally, if $f(z)=Q(e^{\mu z}, e^{\nu z})$ for a polynomial $Q(u,v)$ in
both arguments, then
	$$
	f'(z)=\mu e^{\mu z}Q_u(e^{\mu z}, e^{\nu z})
	+\nu e^{\nu z}Q_v(e^{\mu z}, e^{\nu z}),
	$$
	\begin{eqnarray*}
	f''(z)&=&\mu^2e^{\mu z}Q_u(e^{\mu z}, e^{\nu z})
	+\nu^2e^{\nu z}Q_v(e^{\mu z}, e^{\nu z})+\\
      	&  & +\mu^2e^{2\mu z}Q_{uu}(e^{\mu z}, e^{\nu z})
      	+2\mu\nu e^{(\mu+\nu) z}Q_{uv}(e^{\mu z}, e^{\nu z})
        +\nu^2e^{2\nu z}Q_{vv}(e^{\mu z}, e^{\nu z}).
	\end{eqnarray*}
Taking $Q(u,v)=1+u+v+u^2$ and $(\mu, \nu)=(-\log 2, -\log 3)$, we see that the above $f$ satisfies
	$$
	f'(z)=(-\log 2)e^{(-\log 2)z}\left(1+2e^{(-\log 2)z}\right)
	+(-\log 3)e^{(-\log 	3)z},
	$$
	$$
	f''(z)=(-\log 2)^2e^{(-\log 2)z}\left(1+2e^{(-\log 2)z}\right)
	+(-\log 3)^2e^{(-\log 3)z}+2(-\log 2)^2e^{2(-\log 2)z}.
	$$
The equation~\eqref{4thPS} is obtained by deleting exponentials $e^{(-\log 3)z}$ and $e^{2(-\log 2)z}$ from them.
One can find more interesting examples by observing partial derivatives of other $Q(u,v)$ or polynomials in more variables, but it is reasonable for us to skip this attempt now.
\end{example}

\section{Exponential polynomials in the oscillation theory}\label{EPO-sec}

This section may be recognized as a succession of Section~\ref{ODE-sec}.
Instead of focusing on questions on growth and duality, we deal with the
number of zeros of solutions, although the examples below will include
some glimpses back to the contents of Section~\ref{ODE-sec}.

Most questions involve second order linear differential equations
	\begin{equation}
	g''(z)+p(z)g'(z)+q(z)g(z)=0,\label{st}
	\end{equation}
where $p(z)$ and $q(z)$ are entire. If $g$ is a non-trivial solution,
then \eqref{st} can be transformed to a normalized equation
	\begin{equation}\label{LDE2}
	f''+A(z)f=0
	\end{equation}
by means of
	\begin{equation}\label{transformation-ince}
	g(z)=f(z) \exp\left( -\frac{1}{2} \int^z p(t)dt \right)
	\quad \text{and}\quad
	A(z)= q(z) - \frac{1}{2} p'(z) -\frac{1}{4} p(z)^2.
	\end{equation}
Since $f$ and $g$ share the same zeros, this transformation allows us
to focus on \eqref{LDE2} in place of \eqref{st}.
Note that if $p(z)$ and $q(z)$ are polynomials, then so is $A(z)$.

We are unaware of the origin of the transformation \eqref{transformation-ince},
but at least it can be traced back to the book
of Ince \cite[p.~394]{Ince}, which was originally published in 1926.

Conversely to the situation above, suppose that equation
\eqref{LDE2} with an entire coefficient $A(z)$ is our starting
point, possessing linearly independent solutions $f_1$ and $f_2$. Then $f_1^{n-1}, f_1^{n-2}f_2,\ldots ,f_1f_2^{n-2}, f_2^{n-1}$
are linearly independent solutions of \eqref{lde} with
some entire coefficients and with $A_{n-1}(z)\equiv 0$ by \cite[Theorem~4]{CGHR2}. This allows us to construct examples
of higher-order equations from examples regarding \eqref{LDE2}.

\subsection{Examples on transformed equations}

Ozawa \cite{Ozawa} has proved that the differential equation
	$$
	w''+e^{-z}w'+(az+b)w=0,\quad a,b\in\C,\ a\neq 0,
	$$
has no finite order entire solutions, that is, $\rho(w)=\infty$ holds for any non-trivial solution~$w$. This equation transforms into
	$$
	f''+\left(-\frac{1}{4}e^{-2z}+\frac{1}{2}e^{-z}+az+b\right)f=0
	$$
via the transformation \eqref{transformation-ince}. It would be tempting
to believe that $\lambda(f)=\infty$, but this is not guaranteed even
though $\rho(w)=\infty$. Indeed, $\lambda(w)<\infty$ could still hold.

Section~4 in \cite{Gundersen2019} contains many open questions regarding
the zeros of solutions of \eqref{LDE2}, while Section~5 in \cite{Gundersen2019}
contains questions on the growth of solutions \eqref{st}. Almost all
examples illustrating these research questions are given in terms
of exponential sums.

It turns out that finding the number of zeros of solutions is somewhat more
delicate than finding the growth of solutions. We proceed to take a look at some further examples.

\begin{example}\label{example5}
\textnormal{(Bank-Laine-Langley \cite{BLL2})}
For any $\gamma\in\C\setminus\{0\}$, the equation
    \begin{equation} \label{BLL}
    f''-\left(\gamma^2 e^z-\frac{\gamma}{2} e^{z/2}+\frac{1}{4}\right)f=0
    \end{equation}
possesses linearly independent solutions
    \begin{eqnarray*}
    f_1(z)&=&\exp\left(2\gamma e^{z/2}-z/2\right),\\
    f_2(z)&=&(4\gamma e^{z/2}+1)\exp\left(-2\gamma e^{z/2}-z/2\right),
    \end{eqnarray*}
where $f_1$ has no zeros, $\lambda(f_2)=1$ and $\overline{\lambda}(c_1f_1+c_2f_2)=\infty$ for any constants $c_1,c_2\in\C\setminus\{0\}$.
\end{example}

Among the research questions in Section~\ref{oscillation-problems-sec}
below are to find conditions for $A(z)$ guaranteeing that $f$ has no zeros,
or $\lambda(f)\geq\rho(A)$, or $\lambda(f)=\infty$. These three basic
cases are visible in Example~\ref{example5} and in many other similar examples.

We recall Question 4.3 in \cite{Gundersen2019}:
\begin{quote}
\emph{Is it possible
to characterize non-constant periodic entire functions $A(z)$ of period
$\omega$ and rational in $e^{\alpha z}$, where $\alpha=\frac{2\pi i}{\omega}$,
such that \eqref{LDE2} possesses two linearly independent solutions $f_1,f_2$
satisfying $\lambda(f_1f_2)<\infty$?}
\end{quote}
In other words, $A(z)=R(e^{ \frac{2\pi i}{\omega} z})$, where $R(w)$ is rational in $w$. This is precisely the situation
in Example~\ref{example5}. As discussed in Section~\ref{periodic-subnormal-sec}, such a rational function $R(w)$ needs to be a Laurent polynomial
in order for $A(z)$ to be entire.

\begin{example}
The equation~\eqref{BLL} is transformed by $g(z):=f(z)\exp\left(2\gamma e^{z/2}+z/2\right)$ to an equation
	\begin{equation} \label{HL}
	g''(z)-(2\gamma e^{z/2}+1)g'(z)+\gamma e^{z/2}g(z)=0,
	\end{equation}
which resembles the Hermite equation \eqref{hermite} and the Laguerre equation \eqref{laguerre}. The linearly independent solutions $f_1,f_2$ are replaced with a zero-free solution $g_1(z)=\exp\left(4\gamma e^{z/2}\right)$ and a subnormal solution $g_2(z)=4\gamma e^{z/2}+1$.
Their linear combination $g=c_1g_1+c_2g_2$ over $\C$ can be neither zero-free nor subnormal, unless $c_1c_2=0$. In fact, $\lambda(g)=\infty$.

Observe that the subnormal solution $g_2$ and the two coefficients of~\eqref{HL} are all linear polynomials in the exponential function $e^{z/2}\in{\cal E}$, yet $g_2$ is not dual with either of the coefficients.
However, there is no violation with Theorem~\ref{dual-thm} because the assumption $T(r,B)=o(T(r,A))$ is not valid in this case.

Another transformation of~\eqref{BLL} is by means of $h(z):=f(z)\exp\left(2\gamma e^{z/2}\right)$. We obtain
	\begin{equation} \label{H}
	h''(z)-2\gamma e^{z/2}h'(z)-\left(\frac{\gamma}{2}e^{z/2}
	-\frac{1}{4}\right) h(z)=0
	\end{equation}
having a zero-free solution $h_1(z)=\exp\left(4\gamma e^{z/2}-z/2\right)$ and a subnormal solution $h_2(z)=4\gamma +e^{-z/2}$ as a fundamental solution base.
The subnormal solution $h_2$ does have a dual relationship with the coefficients of~\eqref{H} in this case either.

Noting that the coefficient functions in \eqref{HL} and \eqref{H} are all non-constant linear functions in~$e^{z/2}$ and thus equally strong in terms of growth, one sees that in general it is impossible to judge whether there exists a dual solution to a second order linear differential equation with non-constant exponential sums coefficients.
\end{example}

\begin{example}\label{example3}
Returning to Frei's differential equation~\eqref{Frei-eqn} in Section~\ref{periodic-subnormal-sec}, one sees it has the normalized form
	$$
	f''-\left(\frac{1}{4}e^{-2z}-\frac{1}{2}e^{-z}-\alpha \right)f=0.
	$$
For $\alpha=-1$ the functions $g_1(z)=\exp\left( e^{-z}+z\right)$ and
$g_2(z)= e^{z}+1$ are known linearly independent solutions of Frei's equation.
They are transformed to the solutions
$f_1(z)=\exp\left( \frac{3}{2}e^{-z}+z\right)$ and $f_2(z)= (e^{z}+1)\exp\left(\frac{1}{2} e^{-z}\right)$ of the normalized equation.

Similarly as above, another transformation $g(z)=f(z)\exp\left(-\frac{1}{2}e^{-z}-z\right)$ gives a zero-free solution $g_1(z)=\exp\left( e^{-z}\right)$ and a subnormal (but not dual) solution $g_2(z)= 1+e^{-z}$ to the corresponding equation
	$$
	g''(z)+(e^{-z}+2)g'(z)+e^{-z}g(z)=0,
	$$
whose coefficients are once again equally strong. These equations of the form \eqref{st} are in the same category as the Hermite \eqref{hermite} and the Laguerre \eqref{laguerre} differential equations.
\end{example}

Chiang and Ismail \cite{CI} have studied in detail the relationship between second order ODE's with periodic coefficients and special functions and orthogonal polynomials from the oscillation theory point of view. In particular, see \cite[Remark~1.11, p.~733]{CI}.

Our last example is about a third order linear differential equation, where
a transformation to another third order differential equation can also be made.

\begin{example}
Hinkkanen, Ishizaki, Laine and Li \cite{HILL} have considered a third order ODE of the form
	\begin{equation}\label{third}
	f'''(z)-Kf'(z)+e^zf(z)=0, \quad K\in\C.
	\end{equation}
If a non-trivial solution $f$ of \eqref{third} satisfies $\log^+ N(r,1/f)=o(r)$, then there exist $r,s\in \Z$ with $r+s\geq 0$ such that $K=\frac{(r+s+1)^2}{9}$. The solution $f$ takes one of the forms
	$$
	f_j(z)=e^{-(s+1)z/3}\psi(e^{z/3})\exp\left(c_je^{z/3}\right),
	\quad c_j^3+27=0, \ j=1,2,3,
	$$
where $\psi(t)=\sum_{n=-r}^{s} d_n t^n$ is a Laurent polynomial over $\C$. The transformation $g(z)=f(z)\exp(-c_je^{z/3})$ for some $j\in\{1,2,3\}$
transforms \eqref{third} to
	\begin{equation}\label{standard-third}
	\begin{split}
	g'''(z)&+c_je^{z/3}g''(z)+\left(\frac{c_j^2}{3}e^{2z/3}
	+\frac{c_j}{3}e^{z/3}-K\right)g'(z)\\
	&+\frac{c_j}{3}e^{z/3}\left(\frac{c_j}{3}e^{z/3}
	+\frac{c_j}{9}-K\right)g(z)=0.
	\end{split}
	\end{equation}
Then the corresponding $f_j$ is transformed to the subnormal solution $g_j(z)=e^{-(s+1)z/3}\psi(e^{z/3})$ for the fixed $j$. In fact, $g_j$ is a Laurent polynomial in $e^{z/3}$ of the form
	$$
	g_j(z)=e^{-(s+1)z/3}\sum_{n=-r}^s d_n e^{nz/3}
	=\sum_{m=-(r+s)}^{0} d_{m+s} e^{(m-1)z/3},
	$$
which is a dual exponential sum to the coefficients in \eqref{standard-third}.
\end{example}

The remaining part of Section~\ref{EPO-sec} will focus on the oscillation
theory regarding \eqref{LDE2}. Examples~\ref{example5}--\ref{example3}
turn out to be useful from this perspective.


\subsection{Oscillation problems and some known results}\label{oscillation-problems-sec}

In this section we discuss the oscillation theory.
Motivated by Examples~\ref{example5}--\ref{example3}, we formulate three typical questions regarding the oscillation of non-trivial solutions $f$ of
    \begin{equation}
    f''+A(z)f=0\label{lde2}
    \end{equation}
with an entire coefficient $A(z)$ as follows: Under which conditions for $A(z)$ do we have

\medskip
(Q1) $f$ has no zeros,

(Q2) $\lambda(f)\geq\rho(A)$,

(Q3) $\lambda(f)=\infty$?

\medskip
\noindent
These questions can also be re-phrased for solution bases $\{f_1,f_2\}$ as follows:
Under which conditions for $A(z)$ do we have

\medskip
(Q1') $f_1,f_2$ have no zeros,

(Q2') $\max\{\lambda(f_1),\lambda(f_2)\}\geq\rho(A)$,

(Q3') $\max\{\lambda(f_1),\lambda(f_2)\}=\infty$?

\medskip
\noindent
The cases $0<\lambda(f)<\rho(A)$, $\rho(A)<\lambda(f)<\infty$, and $\lambda(f)=\rho(A)\not\in\N$
are also possible \cite{HILT2}, but there do not seem to be many examples of these situations.

Regarding (Q1) and (Q1'), we recall from \cite[p.~356]{Bank-Laine} that \eqref{lde2} has
two linearly independent zero-free solutions $f_1,f_2$ if and only if $A(z)$ can be represented as
    \begin{equation}
    -4A(z)=e^{2\varphi(z)}+{\varphi}'(z)^2-2\varphi^{\prime\prime}(z), \label{A1.3}
    \end{equation}
where $\varphi$ is a non-constant entire function.
Here $A(z)$ is an exponential polynomial if and only if $\varphi$ is a polynomial.
This implies that if $A(z)$ is an exponential polynomial with at least two exponential terms, then \eqref{lde2} cannot have a zero-free solution base. However, even if $A(z)$ has multiple
exponential terms, equation \eqref{lde2} can still have zero-free
solutions, see Examples~\ref{example5} and \ref{example3}.

Regarding (Q2') and (Q3'), we point out the papers \cite{BLL1, BLL2, Edmund}. From the exponential polynomials perspective, the following special case of \cite[Theorem 1]{BLL1} is intriguing:

\begin{theorem}
\textnormal{(Bank-Laine-Langley \cite{BLL1})}
Suppose that $A(z)$ is an exponential polynomial of the non-normalized form
    $$
    A(z)=P(z)+P_1(z)e^{Q_1(z)}+\cdots+P_n(z)e^{Q_n(z)},
    $$
where $P(z)$ either vanishes identically or is a polynomial such that $\frac{\deg(P)+2}{2}<\rho(A)=\max\{\deg (Q_j)\}$.
If $\{f_1,f_2\}$ denotes a solution base for \eqref{lde2}, then $\max\{\lambda(f_1),\lambda(f_2)\}=\infty$.
\end{theorem}

An example of a condition guaranteeing (Q2) is $\overline{N}(r,1/A)=S(r,A)$, see \cite[Satz~1]{BFL}.
An improvement of this result is stated as follows.

\begin{theorem}\label{K-thm-IT}
\textnormal{(\cite{Ishizaki-Tohge})}
Let $K>4$, and let $A(z)$ be a transcendental entire function of order $\rho(A)$ such that
    \begin{equation}\label{K}
    K\overline{N}\left(r,\frac{1}{A}\right)\leq T(r,A)+S(r,A).
    \end{equation}
Then every non-trivial solution $f$ of \eqref{lde2} satisfies $\lambda(f)\geq\rho(A)$.
\end{theorem}

Regarding (Q1), we mention the following result, known as the $\frac{1}{16}$-theorem.

\begin{theorem}\label{1/16}
\textnormal{(Bank-Laine-Langley \cite{BLL2})}
Let $P$ be a polynomial of degree $\deg(P)>0$, let $Q$ be an
entire function of order $\rho(Q)<\deg P$, and let $A(z)=e^{P(z)}+Q(z)$.
If \eqref{lde2} admits a non-trivial solution $f$ such that $\lambda(f)<\deg(P)$. Then $f$ has no zeros, $Q$ is a polynomial and
	$$
	Q=-\frac{1}{16}(P')^2+\frac{1}{4}P''.
	$$
\end{theorem}

Regarding (Q3), we point out the papers \cite{Bank-Langley, BLL2, CLW}. In particular, we recall the following special case of the main result in \cite{Bank-Langley}.

\begin{theorem}
\textnormal{(Bank-Langley \cite{Bank-Langley})} Suppose that
\begin{itemize}
\item[{\rm (a)}] $P(z)$, $Q(z)$ are polynomials such that $P(z)=a_nz^n+\cdots+a_0$ and $\deg(Q)+2<2n$,
\item[{\rm (b)}] $\Pi(z)$ is an entire function with $\rho(\Pi)<n$,
\item[{\rm (c)}] there exists $\theta_0\in\R$ with $\re\left(a_ne^{in\theta_0}\right)=0$
           and a positive $\veps$ such that $\Pi(z)$ has only finitely many zeros in the sector $|\arg(z)-\theta_0|<\veps$.
\end{itemize}
Then every solution $f\not\equiv 0$ of
    $$
    f''+\left(\Pi(z)e^{P(z)}+Q(z)\right)f=0
    $$
satisfies $\lambda(f)=\infty$.
\end{theorem}

Next we consider \eqref{lde2} from the point of view of (Q2) and (Q3) in the
case when $A(z)$ has two exponential terms. Specifically, we consider
    \begin{equation}
    f''+\left(e^{P_1(z)}+e^{P_2(z)}+Q(z)\right)f=0,\label{1.1}
    \end{equation}
where $P_1(z)=\zeta_1z^n+\cdots$ and $P_2(z)=\zeta_2z^m+\cdots $ are non-constant polynomials
such that $e^{P_1(z)}$ and $e^{P_2(z)}$ are linearly independent, and $Q(z)$ is an entire function
of order $<\max\{n,m\}$. We summarize \cite[Theorem~1]{Ishizaki} and \cite[Theorem~3.2]{Ishizaki-Tohge}
as follows.

\begin{theorem}\label{Ishizaki-Tohge}
\textnormal{(\cite{Ishizaki, Ishizaki-Tohge})}
Let $f\not\equiv 0$ be a solution of \eqref{1.1}.
If $n\neq m$, then $\lambda(f)=\infty$, while if $n=m$, we have the following four assertions.
\begin{itemize}
\item[\textnormal{(a)}] If $\zeta_1=\zeta_2$, then $\lambda(f)\geq n$.
\item[\textnormal{(b)}] If $\zeta_1\neq\zeta_2$ and $\zeta_1/\zeta_2$ is non-real, then $\lambda(f)=\infty$.
\item[\textnormal{(c)}] If $0<\zeta_1/\zeta_2<1/2$, then $\lambda(f)\geq n$.
\item[\textnormal{(d)}] If $3/4<\zeta_1/\zeta_2<1$ and $Q(z)\equiv 0$, then $\lambda(f)\geq n$.
\end{itemize}
\end{theorem}

If the points $0$, $\zeta_1$, $\zeta_2$ in the complex plane are not collinear, we have $\lambda(f)=\infty$ for every solution $f\not\equiv 0$ of \eqref{1.1} by Theorem~\ref{Ishizaki-Tohge}~(b), see also \cite[Corollary~4.2]{BLL2}.
Suppose then that $0,\zeta_1,\zeta_2$ are collinear.
If $\zeta_1,\zeta_2$ are on the opposite sides of the origin, then as an immediate consequence of \cite[Theorem 4.3]{BLL2}, as remarked in \cite{HILT},
we find that every solution $f\not\equiv 0$ of \eqref{1.1} satisfies $\lambda(f)=\infty$.
If $\zeta_1,\zeta_2$ are on the same side of the origin,
we may assume that $0<\zeta_1/\zeta_2<1$.
This leads us to consider equations of the form
    \begin{equation}\label{1.1b}
    f''+\left(e^{P(z)}+e^{\rho P(z)}+Q(z)\right)f=0.
    \end{equation}
There are examples (see \cite[Examples~7.1,~7.2]{HILT2} and \cite{Ishizaki}) of zero-free solutions in the following two cases: (1) $\rho=1/2$ and $Q(z)\equiv 0$; (2) $\rho=3/4$ and $Q(z)\not\equiv 0$. These examples are also related to Parts~(c) and~(d) of Theorem~\ref{Ishizaki-Tohge}.

It is natural to ask if it is possible to obtain $\lambda(f)=\infty$ in place
of $\lambda(f)\geq n$ in Part~(c) and Part~(d) of Theorem~\ref{Ishizaki-Tohge}, and is it possible to ignore the condition $Q(z)\equiv0$ in Part~(d)?
These questions will be answered in Theorem~\ref{new} below, where
we will also discuss the gap $1/2<\rho<3/4$, which corresponds to the case $m=1$.

\begin{theorem}\textnormal{(\cite{HILT2})}
\label{new}
Suppose that $P(z)$ and $Q(z)$ are polynomials satisfying
    \begin{equation}\label{degrees}
    \begin{split}
    \deg(Q)+2 < 2\deg(P),\quad &\text{if}\ \deg(P)\geq 2,\\
    Q(z)\equiv 0,\quad &\text{if}\ \deg(P)=1,
    \end{split}
    \end{equation}
and that $\rho$ is a constant satisfying
    \begin{equation}\label{sig}
    0<\rho<\frac{1}{2}\quad\text{or}\quad \frac{2m-1}{2m}<\rho<\frac{2m+1}{2(m+1)}
    \end{equation}
for some $m\in\N$. Then every solution $f\not\equiv 0$ of \eqref{1.1b} satisfies $\lambda(f)=\infty$.
\end{theorem}

If the assumptions \eqref{degrees} and \eqref{sig} in Theorem~\ref{new} are violated, then \eqref{1.1b} may posses zero-free solutions.

\begin{example}\textnormal{(\cite{HILT2})}
(a) Let $\rho=1/2$, and let $\varphi(z)$ be a polynomial of degree $\deg(\varphi)\geq 2$. Then
    $$
    f(z)=\exp\left(i\int^z e^{-2\varphi(\zeta)+i\zeta}\, d\zeta+\varphi(z)\right)
    $$
is a zero-free solution of \eqref{1.1b}, where the polynomials $P(z)=-4\varphi(z)+2iz$ and $Q(z)=-\varphi'(z)^2-\varphi''(z)$ satisfy $\deg(P)=\deg(\varphi)\geq 2$ and $\deg(Q)+2=2\deg(P)$. Observe that if $\deg(P)=\deg(\varphi)=1$,
then $Q(z)$ is a nonzero constant.
\vspace{0.25cm}

\noindent
(b) Let $\rho=3/4$. Then
    $$
    f(z)=\exp\left(-2e^{-\frac{i}{2}z}-2e^{-\frac{i}{4}z}+\frac{i}{8}z\right)
    $$
is a zero-free solution of \eqref{1.1b}, where $P(z)=-iz$ and $Q(z)=-1/64$.
\end{example}

The behavior of solutions of \eqref{1.1b} depends on the properties of the coefficient
	$$
	A(z)=e^{P(z)}+e^{\rho P(z)}+Q(z), \quad 0<\rho<1.
	$$
Using Theorem~\ref{Stein-thm1}, we find that if $Q(z)$ is a non-constant polynomial, then $A(z)$ has no finite deficient values, while if $Q(z)\equiv c\in\C$, then $c$ is the only finite deficient value of
$A(z)$ with $\delta(c,A)=\rho$. Even if $c=0$ would be a deficient value for $A(z)$, it is not always guaranteed that the situation in
(Q2) let alone in (Q3) would hold. Finally, if $Q(z)$ is a transcendental entire function, then the assumption \eqref{degrees} in Theorem~\ref{new} makes no sense. However, way may still ask whether \eqref{1.1b} could have a zero-free solution if $\rho$ satisfies $\frac{2m-1}{2m}<\rho<\frac{2m+1}{2(m+1)}$ for some $m\in\N$.

We discuss the question above in the case when $P$ is linear.
The reasoning in this direction becomes overly technical even if $\deg(P)=1$ as described in \cite{HILT2}.  We re-write \eqref{1.1b} in this case as
    \begin{equation}
    f''+\left(H_1e^{z}+H_2e^{\rho z}+q(z)\right)f=0,\label{8.1}
    \end{equation}
where $q(z)$ is an entire function of order less than $1$ and $H_1,H_2\in\C\setminus\{0\}$.
Below we consider the cases
    \begin{eqnarray}
    1/2&<&\rho<3/4\quad (m=1),\label{8.2}\\
    3/4&<&\rho<5/6\quad (m=2).\label{8.205}
    \end{eqnarray}
Observe that for the case \eqref{8.205} with $q(z)\equiv0$,
there are no zero-free solutions to \eqref{8.1} by Part~(d) of Theorem~\ref{Ishizaki-Tohge}. Hence we may assume that $q(z)\not\equiv 0$ in \eqref{8.1}.

\begin{theorem}\textnormal{(\cite{HILT2})}\label{main}
Equation \eqref{8.1} under either of the conditions \eqref{8.2} or \eqref{8.205}
possesses no zero-free solutions.
\end{theorem}

It seems that in the general case
$\frac{2m-1}{2m}<\rho<\frac{2m+1}{2(m+1)}$ there are no zero-free solutions either, but the
method used in~\cite{HILT2} can only give one specific interval at the time, and the amount of considerations increases exponentially along with $m$. The complexity of the
case $\frac{5}{6}<\rho<1$ is illustrated in~\cite[Section~6]{HILT2}.
The proof of Theorem~\ref{main} yields general forms of zero-free solutions in the
cases $\rho=1/2$ and $\rho=3/4$, including two examples given in \cite{Ishizaki, Ishizaki-Tohge}. The details are given in~\cite[Section~7]{HILT2}, and recent progresses are seen in \cite{Zhang2}.

\begin{problem}
If $\frac{2m-1}{2m}<\rho<\frac{2m+1}{2(m+1)}$, prove that \eqref{8.1} possesses no zero-free solutions. Find general forms of zero-free solutions in the cases $\rho=\frac{2m-1}{2m}$ for $m\geq 3$.
\end{problem}

\subsection{Oscillation results in terms of convex hulls}

We proceed to consider oscillation results in terms of convex hulls. General results
about convex sets and convex hulls can be found in Appendix~\ref{convex-appendix}.
Our first result is an obvious consequence of Theorems~\ref{Stein-thm1} and \ref{K-thm-IT}, where $A(z)$ in \eqref{lde2} is an exponential polynomial and has $m\geq2$ exponential terms.

\begin{corollary}\textnormal{(\cite{HILT})}\label{n-thm}
Let $A(z)$ be an exponential polynomial of the form
    \begin{equation*}\label{A}
    A(z)=H_1(z)e^{\zeta_1z^n}+\cdots +H_m(z)e^{\zeta_mz^n},\quad m\geq 2,
    \end{equation*}
where the functions $H_j(z)$ are either
exponential polynomials of order $<n$ or ordinary polynomials in $z$. Denote
$W=\left\{\bar{\zeta}_1,\ldots ,\bar{\zeta}_m\right\}$ and $W_0=W\cup\{0\}$,
and suppose that
    \begin{equation}\label{perimeters}
    C(\co(W_0))>4C(\co (W)).
    \end{equation}
Then every non-trivial solution $f$ of \eqref{lde2} satisfies $\lambda(f)\geq n$.
\end{corollary}

\begin{remark}
(1) Geometrically speaking, if the convex hull $\co (W)$ has a large circumference,
then it needs to be lay sufficiently far away from
the origin for \eqref{perimeters} to hold.

(2) Similarly as in Theorem~\ref{Ishizaki-Tohge}, suppose that $W=\left\{\bar{\zeta}_1,\bar{\zeta}_2\right\}$ and $\zeta_1=\rho\zeta_2$ for $0<\rho<1$. Then the assumption \eqref{perimeters} reduces to
    $$
    2|\zeta_2|>2\cdot 2(|\zeta_2|-|\zeta_1|)=4(1-\rho)|\zeta_2|,
    $$
from which $\rho>1/2$. Thus Corollary~\ref{n-thm}
generalizes Theorem~\ref{Ishizaki-Tohge}(d), and the result is sharp in the sense that an equality in \eqref{perimeters} cannot hold \cite{Ishizaki}.

(3) More exponential terms of order $n$ can be added to the coefficient $A(z)$ without affecting the assertion $\lambda(f)\geq n$, for as long as the conjugates of the frequencies of the new exponential terms belong to $\co (W)$. This stems from the fact that such an ``addition'' has no affect on \eqref{perimeters}.
\end{remark}

Regarding (Q2'), we proceed to formulate a result of perturbation type,
which shows that the small term of the exponential
polynomial coefficient $A(z)$ plays a role in the oscillation theory.

\begin{theorem}\label{perttu-thm}
\textnormal{(Bank-Laine-Langley \cite{BLL2})}
Let $A(z)$ be an exponential polynomial of order $n$, and let $f_1,f_2$ be two linearly independent
solutions of \eqref{lde2} with $\max\{\lambda(f_1),\lambda(f_2)\}<n$. Then, for any entire function
$B(z)\not\equiv 0$ of order $\rho(B)<n$, any two linearly independent solutions $g_1,g_2$ of the
differential equation
    $$
    g''+(A(z)+B(z))g=0
    $$
satisfy $\max\{\lambda(g_1),\lambda(g_2)\}\geq n$.
\end{theorem}

\begin{remark}
Prior to Theorem~\ref{perttu-thm}, it was shown in \cite[Corollary~1]{BLL1} that $\max\{\lambda(f_1),\lambda(f_2)\}=\infty$ for any linearly independent solutions $f_1,f_2$ of \eqref{lde2}, provided that $A(z)$ is an exponential polynomial of the form
	\begin{equation*}\label{exp-poly}
    	A(z)=P_1(z)e^{Q_1(z)}+\cdots +P_k(z)e^{Q_k(z)},
    	\end{equation*}
where the functions $P_j(z)$ and $Q_j(z)\not\equiv$ const.~are polynomials.
That is, the assumption $\max\{\lambda(f_1),\lambda(f_2)\}<n$ in Theorem~\ref{perttu-thm}
is possible only if at least one of the polynomials $Q_j(z)$ is a constant.
\end{remark}

The next result generalizes Theorem~\ref{Ishizaki-Tohge}(c).

\begin{theorem}\textnormal{(\cite{HILT})}\label{n-thm-c}
Let $A(z)$ be an exponential polynomial of the normalized form
    \begin{equation}\label{A0}
    A(z)=H_0(z)+H_1(z)e^{\zeta_1z^n}+\cdots +H_m(z)e^{\zeta_mz^n},\quad m\geq 2.
    \end{equation}
Denote $B(z)=A(z)-H_1(z)e^{\zeta_1z^n}$, and suppose that $h_A(\theta)>2h_B(\theta)$ whenever $h_B(\theta)>0$, and
that $C(\co(W_0^A))>2C(\co(W_0^B))$.
Then $\lambda(f)\geq n$ for any solution $f\not\equiv 0$ of \eqref{lde2}.
\end{theorem}

The necessity of the assumptions in Theorem~\ref{n-thm-c} are
discussed in the next example.

\begin{example}
\textnormal{(\cite{HILT})}
Let $f(z)=\exp\left(e^{z^n}\right)$, and denote $C(z)=f'(z)/f(z)=nz^{n-1}e^{z^n}$. Then $f$ is a zero-free solution
of \eqref{lde2}, where
    \begin{eqnarray*}
    A(z)&=&-C'(z)-C(z)^2\\
    &=&-\left(n(n-1)z^{n-2}+n^2z^{2(n-1)}\right)e^{z^n}-n^2z^{2(n-1)}e^{2z^n}.
    \end{eqnarray*}
Defining $B(z)=A(z)+C(z)^2=-C'(z)$, we have $h_A(\theta)=2h_B(\theta)$ whenever $h_B(\theta)>0$,
and $C(\co(W_0^A))=2C(\co(W_0^B))=4$.
\end{example}

\section{Exponential polynomials and O$\Delta$E's}

We aim to classify the finite order meromorphic solutions of non-linear
differential-difference equations
    \begin{equation}\label{wen-equation1}
    f(z)^n+a_{n-1}f(z)^{n-1}+\cdots+a_1f(z)+q(z)e^{Q(z)}f^{(k)}(z+c)=P(z),
    \end{equation}
where $q(z)$ is entire, $Q(z),P(z)$ are polynomials, $k\geq 0$ and $n\geq 2$ are integers, $c\in\C\setminus\{0\}$, and $a_j\in\C$ for $j=1,2,\ldots,n-1$. These equations arise from the equations
	$$
	f^n+L(z,f)=h(z)
	$$
studied by Laine and Yang in \cite{Y-L}. Here $L(z,f)$ is a linear differential-difference polynomial
in $f$ with meromorphic coefficients of growth $S(r,f)$, the function
$h(z)$ is meromorphic, and $n\geq 2$ is an integer. If $n-k\geq 2$, then
meromorphic solutions of \eqref{wen-equation1} of moderate growth
reduce to entire functions. We will give a precise statement of this claim
together with a proof, which is a simple modification of the reasoning
in \cite[p.~1296]{WHL}.

\begin{lemma}\label{meromorphic-entire-lemma}
Suppose that $n-k\geq 2$, $q(z)$ is a polynomial, and that $f$ is a meromorphic solution of \eqref{wen-equation1} of hyper-order $\rho_2(f)<1$. Then $f$ is entire.
\end{lemma}

\begin{proof}
Suppose first that $k=0$. If $z_0$ is a pole of $f$ of multiplicity $t\geq 1$, then, by \eqref{wen-equation1},
$z_0+c$ is a pole of $f$ of multiplicity $\geq nt$ except for a zero of the polynomial $q(z)$, which are finitely many.
Continuing inductively, we see that $z_m=z_0+mc$ is a pole of $f$ of multiplicity $\geq n^mt$.
Since $n\geq 2$, the hyper-exponent of convergence $\lambda_2(\{z_m\})$ of the sequence $\{z_m\}$ satisfies $\lambda_2(\{z_m\})\geq 1$.
Since $\lambda_2(\{z_m\})\leq \rho_2(f)$ is true in general, and since
$\rho_2(f)<1$, it follows that $f$ is entire.

Suppose then that $k\in \N$ but $n-k\geq 2$.
If $z_0+c$ is a pole of $f^{(k)}$ of multiplicity $\geq nt$, $f$ has also its pole at $z_0+c$ of multiplicity $\geq nt-k\geq (n-k)t$.
Thus as the above, we see that $z_m=z_0+mc$ is a pole of $f$ of multiplicity $\geq (n-k)^mt\geq 2^m t$.
\end{proof}

If  $k=n-1\geq 1$, non-entire meromorphic solutions of finite-order of \eqref{wen-equation1} may exist.

\begin{example}\label{liu.ex}
\textnormal{(Liu \cite{kailiu})}
The non-linear differential-difference equation
    $$
    f(z)^2-e^{-z}f'(z+2\pi i)=0
    $$
has a meromorphic solution $f(z)=1/(1-e^z)$ for $(n,k)=(2,1)$.
\end{example}

\begin{example}
The function $f$ in~Example~\ref{liu.ex} has period~$2\pi i$ and it solves the Riccati equation $f'(z)=e^{z}f(z)^2$.
If we make a change of a variable $x=e^z$ as usual, and thus
$\frac{d}{dz}=x\frac{d}{dx}$, this transforms into $h'(x)=h(x)^2$ with $h(x)=1/(1-x)=f(z)$.
Defining $g(z)=f(z)-\frac{1}{2}a_1$ for $a_1\in\C$, one sees that this $g$ has period
$2\pi i$ and it is a meromorphic solution to
	$$
	g(z)^2+a_1g(z)-e^{-z}g'(z+2\pi i)=-\frac{1}{4}{a_1}^2,
	$$
which is of the form $\eqref{wen-equation1}$.
\end{example}

\begin{example}
The function $f(z)=\sqrt{2}i/\cos z$ solves
	$$
	f(z)^3+f(z)+f''(z+2\pi)=0,
	$$
which is \eqref{wen-equation1} in the case $(n,k)=(3,2)$.
However, $f$ has period $2\pi$, thus $f''(z+2\pi)=f''(z)$, so that the appearance of~$c$ in~\eqref{wen-equation1} seems to mean nothing in this example. Further, it is easy to see the non-existence of non-constant rational solutions to the transformed equation $h(x)^3-\frac{1}{2}h(x)+\frac{1}{2}xh'(x)+\frac{1}{2}x^2h''(x)=0$ by observing the degree of such~$h(x)$.
\end{example}

\begin{example}
The function $f(z)=1/\cos z$ satisfies
	$$
	f(z)^3-\frac{1}{2}f(z)+\frac{1}{2}f''(z+\pi)=0,
	$$
and this example is appropriate for our purpose since it can be modified
to confirm the existence of non-periodic solutions in the case $(n,k)=(3,2)$.
Indeed, choosing $f(z)=1/\cos z-a_2/3$ for $a_2\in\C$, the equation above
takes the form \eqref{wen-equation1}.
\end{example}

Returning to the case $(n,k)=(2,1)$, the existence of non-periodic solutions is still unknown.

\begin{problem}\label{-5.p}
Give an example of a solution to $\eqref{wen-equation1}$ for $(n,k)=(2,1)$ which is not of period~$c$ or prove the non-existence of such examples.
\end{problem}

Note that the solutions in the examples above are rational in the exponential function. If a non-periodic, finite-order and entire solution exists to one of the discussed equations, it may be given as an exponential sum which is neither simple nor irreducible. This leads us to the following problem, for which
Lemma~\ref{meromorphic-entire-lemma} is a partial solution.

\begin{problem}\label{5.p}
Which conditions imply that every meromorphic solution of \eqref{wen-equation1} is entire?
\end{problem}

By the proof of \cite[Theorem 1.1(a)]{WHL}, we find that every entire and finite-order solution~$f$ of \eqref{wen-equation1} satisfies $\rho(f)=\deg(Q)$
and is of mean type. Hence it seems plausible that entire solutions of \eqref{wen-equation1} could be exponential polynomials of the form
    \begin{equation}\label{f.eq}
    f(z)=H_0(z)+H_1(z)e^{w_1z^q}+\cdots+H_m(z)e^{w_mz^q}.
    \end{equation}
Indeed, the exponential polynomial solutions are classified  in~\cite{WHL} in the special case
    \begin{equation}\label{wen-equation2}
    f(z)^{n}+q(z)e^{Q(z)}f(z+c)=P(z),
    \end{equation}
where $q(z),Q(z),P(z)$ are polynomials, $n\geq 2$ is an integer and $c\in\C\setminus\{0\}$.
Entire finite-order solutions of \eqref{wen-equation2} do exist, while a solution does not have
to be unique. Indeed, as observed in \cite{WHL}, the functions $f_1(z)=e^z+1$ and $f_2(z)=e^z-1$ both solve
    \begin{equation}\label{example-eqn}
    f(z)^2-2e^zf(z-\log 2)=1.
    \end{equation}
Note that $f_1$ and $f_2$ are both periodic but not of period~$-\log 2$.
Although we are interested in finite-order entire solutions of \eqref{wen-equation2}, we note that
\eqref{wen-equation2} possesses infinite-order entire solutions also.
For example, the function $f(z)=e^{e^{z}}+e^{-e^{z}}$
solves
	$$
	f(z)^2-f(z+\log 2+\pi i)=2.
	$$
It turns out that every solution $f$ of the
form \eqref{f.eq} reduces to a function that belongs to one of the following two classes
of transcendental entire functions:
    \begin{eqnarray*}
    \Gamma_{1} &=& \{h=e^{\alpha(z)}+d\, :\,
    d\in\C\ \textrm{and}\ \alpha\ \textrm{polynomial}, \alpha\neq \textrm{const.}\},\\
    \Gamma_{0} &=& \{h=e^{\alpha(z)}\, :\, \alpha\ \textrm{polynomial}, \alpha\neq\textrm{const.}\}.
    \end{eqnarray*}

\begin{theorem}\label{WHL-result}
\textnormal{(\cite{WHL})}
Let $n\geq 2$ be an integer, $c\in\C\setminus\{0\}$, and let $q(z),Q(z),P(z)$ be polynomials such
that $Q(z)$ is not a constant and $q(z)\not\equiv 0$. If $f$ is an exponential polynomial
solution of the form \eqref{f.eq} of \eqref{wen-equation2}, then $f\in \Gamma_1$. Moreover, if $f\in\Gamma_1\backslash \Gamma_0$, then $\rho(f)=1$.
\end{theorem}

Li and Yang \cite{nanli} classified the exponential polynomial solutions of
    \begin{equation}\label{wen-equation3}
     f(z)^2+a_{1}f(z)+q(z)e^{Q(z)}f^{(k)}(z+c)=P(z),
    \end{equation}
where $q(z),Q(z),P(z)$ are polynomials, and $a_1, c\in\C\setminus\{0\}$.

\begin{theorem}\label{Li-result}
\textnormal{(Li-Yang \cite{nanli})}
Let $a_1, c\in\C\setminus\{0\}$ and let $q(z),Q(z),P(z)$ be polynomials such
that $Q(z)$ is not a constant and $q(z)\not\equiv 0$. Suppose that $f$ is an exponential polynomial
solution of the form \eqref{f.eq} of \eqref{wen-equation3}. Then the
following is true:
\begin{itemize}
\item[\textnormal{(a)}]
If $m\geq 2$, there exist $i_0, j_0\in\{1,\ldots,m\}$ such that $w_{i_0}=2w_{j_0}$.
\item[\textnormal{(b)}]
If $m=1$, then $f\in\Gamma_1$. Moreover, if $f\in\Gamma_1\setminus\Gamma_0$, then $\rho(f)=1$.
\end{itemize}
\end{theorem}

Liu \cite{kailiu} proceeds to classify the exponential polynomial solutions of
    \begin{equation}\label{wen-equation4}
    f(z)^n+q(z)e^{Q(z)}f^{(k)}(z+c)=P(z),
    \end{equation}
where $q(z)$ is entire, $Q(z),P(z)$ are polynomials, $k$ and $n\geq 2$ are integers, and $c\in\C\setminus\{0\}$. We denote
    \begin{eqnarray*}
    \Gamma^d_{1} &=& \{h=d_1(z)e^{\alpha(z)}+d_2(z)\, :\,
    d_1, d_2\ \textrm{and}\ \alpha\ \textrm{polynomial}, \alpha\neq \textrm{const.}\},\\
    \Gamma^d_{0} &=& \{h=d_1(z)e^{\alpha(z)}\, :\, d_1\ \textrm{and}\ \alpha\ \textrm{polynomial}, \alpha\neq\textrm{const.}\}.
    \end{eqnarray*}

\begin{theorem}
\textnormal{(Liu \cite{kailiu})}
Let $n\geq 2$ and $k\geq 1$ be an integers,  $c\in\C\setminus\{0\}$, and let $q(z),Q(z),P(z)$ be polynomials such
that $Q(z)$ is not a constant and $q(z)\not\equiv 0$. If $f$ is an exponential polynomial
solution of the form \eqref{f.eq} of \eqref{wen-equation4}, then $f\in \Gamma^d_1$.
\end{theorem}

\begin{problem}
Suppose that $f$ is an entire solution of \eqref{wen-equation4}, where
$q(z)$ is a polynomial. If $f\in\Gamma^d_1\backslash \Gamma^d_0$, then is it true that $\rho(f)=1$?
\end{problem}

Example~\ref{liu.ex} shows that
finite order meromorphic solutions of \eqref{wen-equation4} exist. This solution
is in fact a quotient of two exponential polynomials. As for further
examples of this nature, we observe that $f_1(z)=\frac{1}{e^z-1}$
and $f_2(z)=\frac{1-e^z}{2e^z-1}$ solve
	$$
	f(z)^2+2f(z)+e^zf'(z+2\pi i)=-1.
	$$
Such solutions have not been classified yet.

\begin{problem}\label{13}
Classify quotients of two exponential polynomials as solutions of \eqref{wen-equation4}.
\end{problem}

Regarding Problem \ref{13}, Gao, Liu and Liu \cite{GLL} have considered reciprocals of exponential polynomials as solutions of the non-linear differential equation
    \begin{equation}\label{simple2.eq}
    f(z)^2+a_1f(z)+q(z)e^{Q(z)}f'(z)=0,
    \end{equation}
where $Q(z)=b_qz^q+\cdots+b_0$, $b_q\neq 0$, $a_1$ is a constant and $q(z)$ is a non-zero polynomial. Note that \eqref{simple2.eq} resembles \eqref{wen-equation4} but is not
of the form \eqref{wen-equation4} because the  $c$-shift is missing.

\begin{theorem}\label{Liu.theorem}
\textnormal{(Gao-Liu-Liu \cite{GLL})}
Let $g$ be an exponential polynomial. Suppose that $f=1/g$ is a solution of \eqref{simple2.eq}. Then $a_1=0$ and
$g=d+Ae^{wz}$, where $d,A,w$ are constants.
\end{theorem}

\noindent
\textbf{About the appendices.}
Regarding the key tools in this paper, a supplementary explanation of the essence of Nevanlinna theory and the theory of convex sets in relation to exponential polynomials is given in Appendices~\ref{Nevanlinna-appendix} and \ref{convex-appendix} below.
It is important to notice that Nevanlinna theory has two independent points of view on the behavior of entire/meromorphic functions, that is, the oscillation caused by the sequence of roots and the magnitude of growth. These quantities are measured by the counting function $N$ and proximity function $m$, respectively. An important fact is that both $N$ and $m$ for an exponential polynomial $f$ can be evaluated quite acutely by using the circumference of the convex hull formed by the conjugated frequencies of $f$.
Therefore, a detailed discussion and a proof for such a
circumference formula seem to be justified.

The roots of an entire function $f$ can be detected
from its canonical factorization, while the growth of $f$ can be measured by means of its series expansion, if either
of these two things is concretely known.
Finite Riemann zeta functions are known to admit both the sum and product expressions. These functions are among  the most familiar exponential sums, and will be discussed in Appendix~\ref{zeta-appendix} below.

For the completeness of this paper, Appendix~\ref{DE-appendix} contains additional information on the so-called complex oscillation results for linear differential equations. These results
explain the background for Section~\ref{EPO-sec}.


\appendix

\section{Nevanlinna theory}\label{Nevanlinna-appendix}

It may be too much to expect that a reader could grasp all discussions in
this paper without having any background in Nevanlinna theory. Nevertheless,
it is probably useful to remind which concepts and basic results in
Nevanlinna theory will be deployed. While doing this, we also fix the notation.
There are many books on Nevanlinna theory, each emphasizing
different things, and so reading multiple books simultaneously could turn out to be the best strategy. Books such as \cite{GO, Hayman0, Laine, Tsuji, YY, yanglo} are often given as standard references
in Nevanlinna Theory.

\begin{definition}
An infinite sequence $\{z_n\}$ in
$\C\setminus\{0\}$
with no finite limit points has a finite \emph{exponent of convergence} $\lambda> 0$ if
$\{1/|z_n|\}\in\ell^{\lambda+\veps}\setminus\ell^{\lambda-\veps}$
for any $\veps\in (0,\lambda)$, while $\lambda=0$ if $\{1/|z_n|\}\in\ell^\veps$ for any $\veps>0$. Here $\ell^p$ stands for the standard space of $p$-summable complex
sequences. The \emph{genus} of $\{z_n\}$ is the unique integer $p\geq 0$ satisfying
$\{1/|z_n|\}\in\ell^{p+1}\setminus\ell^{p}$. If the value zero appears in
the sequence $\{z_n\}$ finitely many times, we define $\lambda$ and $p$ for
$\{z_n\}$ as above by ignoring the value zero in $\{z_n\}$.
\end{definition}

If $\lambda\not\in\N\cup\{0\}$, then $p$ is the integer part of $\lambda$,
while if $\lambda\in\N\cup\{0\}$, then either $p=\lambda$ or $p=\lambda-1$. In all
cases, we have $p\leq\lambda$.

The number of points $z_n$ in the disc $\{|z|\leq r\}$ counting multiplicities is
denoted by $n(r)$. The number of points $z_n$ at the origin is denoted
by $n(0)$. The corresponding \emph{integrated counting function} is defined by
	$$
	N(r)=\int_0^r \frac{n(t)-n(0)}{t}\, dt+n(0)\log r,\quad 0<r<\infty.
	$$
If $z_n\neq 0$, then $n(0)=0$, and consequently $N(r)=\int_0^r \frac{n(t)}{t}\, dt$. The exponent of convergence $\lambda$ of $\{z_n\}$ can be expressed alternatively by means of the counting functions as
	$$
	\lambda=\limsup_{r\to\infty}\frac{\log n(r)}{\log r}=	
	\limsup_{r\to\infty}\frac{\log N(r)}{\log r}.
	$$
The number of points $z_n$ in $\{|z|\leq r\}$ ignoring multiplicities is
denoted by $\overline{n}(r)$. Then $\overline{N}(r)$ is defined analogously
as $N(r)$, but using $\overline{n}(r)$. The analogous exponent of convergence
$\overline{\lambda}$ is defined using either $\overline{n}(r)$ or
$\overline{N}(r)$, similarly as above.

\begin{definition}	
For a given meromorphic function $f$, the functions $n(r,1/f)$, $n(r,f)$ and $n(r,1/(f-a))$ express the number of zeros, poles and $a$-points of $f$ in $\{|z|\leq r\}$, respectively, counting multiplicities. The corresponding integrated
counting functions are denoted by $N(r,1/f)$, $N(r,f)$ and $N(r,1/(f-a))$.
The exponents of convergence for the zeros, poles and $a$-points of $f$
are denoted by $\lambda(f)$, $\lambda(1/f)$ and $\lambda(a,f)$, respectively.
If the multiplicities are ignored, then the counting functions
$\overline{n}(r,\cdot)$ or $\overline{N}(r,\cdot)$
are used to define $\overline{\lambda}(f)$, $\overline{\lambda}(1/f)$ and $\overline{\lambda}(a,f)$ in an obvious way.
\end{definition}

Note that $\lambda(f)$ is indeed used for the zeros of $f$ and $\lambda(1/f)$
for the poles. Considering the counting functions $n(r,1/f)$ and $n(r,f)$,
one might think that the situation would be the other way around. However, this
particular notation is commonly adapted in the literature, and hence
it is used in this paper also.

\begin{definition}
The \emph{Nevanlinna characteristic} of a meromorphic $f$ is defined by
	$$
	T(r,f)=m(r,f)+N(r,f),
	$$
where
	$$
	m(r,f)=\frac{1}{2\pi}\int_0^{2\pi}\log^+|f(re^{i\theta})|\, d\theta
	$$
is called the \emph{proximity function} of $f$.  The \emph{order of growth} and the
\emph{hyper-order of growth} of $f$ are given respectively by
	$$
	\rho(f)=\limsup_{r\to\infty}\frac{\log T(r,f)}{\log r}
	\quad\textnormal{and}\quad
	\rho_2(f)=\limsup_{r\to\infty}\frac{\log\log T(r,f)}{\log r},
	$$
and if $\rho(f)\in (0,\infty)$, then the \emph{type} of $f$ is the quantity
	$$
	\tau(f)=\limsup_{r\to\infty}\frac{\log T(r,f)}{r^{\rho(f)}}.
	$$
\end{definition}

Note that
$T(r,f)=m(r,f)$ for an entire function $f$.
Exponential polynomials are typical examples of functions of finite order
and of finite type. However, if $0<\rho<\infty$ is any number, then there exists an entire function $f$ satisfying $\rho(f)=\rho$. If $\rho$ is not
an integer, then $\lambda(f)=\rho(f)$ and $f$ has infinitely many zeros. Moreover, if $\lambda(a,f)<\rho(f)$, then the value $a\in\widehat\C$ is
called a \emph{Borel exceptional value} for $f$.

The \emph{first main theorem of Nevanlinna} states that
	$$
	T\left(r,\frac{1}{f-a}\right)=T(r,f)+O(1).
	$$
Since
	$$
	N\left(r,\frac{1}{f-a}\right)\leq T\left(r,\frac{1}{f-a}\right)=T(r,f)+O(1),
	$$
it follows at once that $\lambda(a)\leq \rho(f)$ for any sequence of $a$-points of $f$.
The \emph{second main theorem of Nevanlinna} states that
	$$
	(q-2)T(r,f)=N(r,f)+\sum_{j=1}^qN\left(r,\frac{1}{f-a_j}\right)+S(r,f),
	$$
where $a_1,\ldots,a_q\in\C$ are $q\geq 2$ distinct points, and where $S(r,f)$
denotes a quantity satisfying $S(r,f)=o(T(r,f))$ as $r\to\infty$ outside an exceptional
set of finite linear measure. It follows that every
meromorphic function has at most two Borel exceptional values.

The \emph{Nevanlinna deficiency} of $a$-points of $f$ is defined by
	$$
	\delta(a,f)=\liminf_{r\to\infty}\frac{m\left(r,\frac{1}{f-a}\right)}{T(r,f)}
	=1-\limsup_{r\to\infty}\frac{N\left(r,\frac{1}{f-a}\right)}{T(r,f)}.
	$$
The deficiency $\delta(\infty,f)$ of poles of $f$ is defined analogously.
Clearly $0\leq \delta(a,f)\leq 1$ for $a\in\widehat{\C}$. If $\delta(a,f)>0$, then
$a$ is called a \emph{deficient value} for $f$. The well-known \emph{deficiency relation}
is a rather immediate consequence of the second main theorem, and is stated as
follows: The number of deficient values for $f$ is at most countable, and
	$$
	\sum_{a\in\widehat{\C}}\delta(a,f)\leq 2.
	$$
This inequality is sharp, for if $f(z)=e^{z}$, then $\delta(0,f)=\delta(\infty,f)=1$.

Let $f$ be an entire function of order $\rho\in (0,\infty)$. The \emph{Phragm\'en-Lindel\"of indicator function} of $f$ defined by
	$$
	h_f(\theta)=\limsup_{r\to\infty}\frac{\log |f(re^{i\theta})|}{r^\rho}
	$$
describes the radial growth of $f$. For example, if $f(z)=e^z$, then
$h_f(\theta)=\cos\theta$, while if $f$ is either
of the trigonometric functions $\sin z=(e^{iz}-e^{-iz})/(2i)$ or $\cos z=(e^{iz}+e^{-iz})/2$, then $h_f(\theta)=|\sin\theta|$.
In general, the function $h_f$ is clearly $2\pi$-periodic,
and it is known that if $f$ is of finite type, then $h_f$ is a continuous
function of $\theta$, and
	$$
	\tau(f)=\max_{0\leq \theta\leq 2\pi} h_f(\theta).
	$$

\section{Convex sets}\label{convex-appendix}

The value distribution of a given exponential polynomial $f$ relies
on the convex hull of the conjugates of the frequencies of $f$, as
described in Section~\ref{EP-sec}. Moreover, the critical rays in
Section~\ref{EP-sec} can be found in terms of a supporting function
of the convex hull. This motivates us to take a closer look at convex
hulls and supporting functions, starting from the definitions. Our
presentation partially follows the  presentation in \cite[pp.~74-75]{Levin1}, but the proofs are included for the convenience of the reader.

\begin{definition}
A set $A\subset\C$ is called \emph{convex} if for any two points $z_1,z_2\in A$
the line segment $[z_1,z_2]$ joining $z_1$ and $z_2$ also belongs to $A$. The intersection of all closed convex sets containing $A$ is called the \emph{convex hull} of $A$ and is denoted by $\co(A)$.
\end{definition}

Any non-empty intersection
$\cap_n B_n$ of convex sets $B_n\subset\C$ is convex. Indeed, let $z_1,z_2\in \cap_n B_n$.
Then $z_1,z_2\in B_n$ for every $n$. Since $B_n$ is convex, we have
$[z_1,z_2]\subset B_n$ for every $n$, and hence $[z_1,z_2]\subset \cap_n B_n$. Consequently, $\co(A)$ is closed and convex.

\begin{lemma}\label{intersection-lemma}
$\co(A)$ is the intersection of all closed half-planes containing the set $A$.
\end{lemma}

\begin{proof}
Let $B$ be the intersection of all closed half-planes containing $A$.
We need to prove that $\co(A)=B$.
Noting that closed half-planes are convex, we have $\co(A)\subset B$
because the intersection $\co(A)$ contains all the sets that appear
in the intersection $B$ and more.

Suppose on the contrary to the
assertion that $B\setminus \co(A)\neq\emptyset$, and let $z\in B\setminus \co(A)$.
Then $z$ belongs to the complement of $\co(A)$, which is an open set.
Thus $\operatorname{dist}(z,\co(A))>0$. Since $\co(A)$ is closed,
there exists a point $w\in\co(A)$ such that
	$$
	|z-w|=\operatorname{dist}(z,\co(A)).
	$$
Note that there cannot exist any points of $\co(A)$ on the line segment
$[z,w]$, excluding the end-point $w$. Indeed, if $\zeta$ would be such a point, then $\zeta =\alpha z+(1-\alpha)w$ for some $\alpha\in (0,1)$,
and then
	$$
	|z-\zeta|=(1-\alpha)|z-w|<\operatorname{dist}(z,\co(A)),
	$$
which is a contradiction. Thus the line segment $[z,w]$ lies in the
complement of $\co(A)$, excluding the end-point $w$. Let $D(z,R)$ be
the open disc of radius $R:=\operatorname{dist}(z,\co(A))$ centred at~$z$.
Then $w\in\partial D(z,R)$. Let $L$ be the tangent to the circle
$\partial D(z,R)$ at~$w$. Then $z\not\in L$ and $L$ is orthogonal to
$[z,w]$.

Suppose that there exists a point $\xi\in\co(A)$ on the same
side of $L$ as $z$ such that $\xi\not\in L$. By convexity, the line
segment $[w,\xi]$ belongs to $\co(A)$. Since $\xi\not\in L$, the line
segment $[w,\xi]$ cannot be a part of a tangent to $\partial D(z,R)$ at~$w$,
and hence a line through $w$ and $\xi$ intersects $\partial D(z,R)$
at two distinct points. The second point of intersection other than $w$
must be on the same side of $L$ as $z$ and $\xi$ are. Consequently, some points of $[w,\xi]$ are interior points of $D(z,R)$ and belong to $\co(A)$, which is a contradiction. Thus a point $\xi$ with
the above properties cannot be found. This implies that the closed
half-plane $H$ bounded by $L$ and not containing $z$ contains the set
$\co(A)$. But now $H$ is a closed half-plane containing $A$, so
$H$ should be one of the half-planes constituting the set $B$. As such,
$H$ should include the point $z$, which is a contradiction. Hence the assumption $B\setminus \co(A)\neq\emptyset$ is false.
\end{proof}

\begin{definition}
Let $G\subset\C$ be a bounded, closed and convex set. The \emph{supporting
function} of $G$ is the function
	$$
	k(\theta)=\sup_{z\in G}\re(ze^{-i\theta})
	=\sup_{x+iy\in G}(x\cos\theta+y\sin\theta).
	$$
For any $\theta\in\R$ the line $\ell_\theta$ defined by
	$$
	\ell_\theta=\left\{(x,y)\in\R^2 :
	x\cos\theta+y\sin\theta-k(\theta)=0\right\}
	$$
is called a \emph{supporting line} of $G$.
\end{definition}

Since $G$ is closed and bounded, every line $\ell_\theta$ has a point
in common with $G$. In addition, all points of $G$ lie on one side of
each $\ell_\theta$ because for $x+iy\in G$, we have
$x\cos\theta+y\sin\theta-k(\theta)\leq 0$ by the definition of the
supporting function. The points that $\ell_\theta$ has in common with $G$
are called \emph{supporting points} of $\ell_\theta$.

\begin{lemma}\label{supporting-points-lemma}
Each supporting line $\ell_\theta$ of $G$ has either one single point of support or a whole line segment of supporting points.
\end{lemma}

\begin{proof}
Suppose that $\ell_\theta$ has two distinct supporting points,
say $\zeta$ and $\xi$. Then
	$$
	\re(\zeta e^{-i\theta})-k(\theta)=0\quad\textnormal{and}\quad
	\re(\xi e^{-i\theta})-k(\theta)=0.
	$$
By convexity, the line segment
$[\zeta,\xi]$ belongs to $G$. Let $z\in [\zeta,\xi]$. Then there
exists a constant $\alpha\in [0,1]$ such that
	$
	z=\alpha\zeta+(1-\alpha)\xi.
	$
But now
	$$
	\re(z e^{-i\theta})-k(\theta)=\alpha\re(\zeta e^{-i\theta})
	+(1-\alpha)\re(\xi e^{-i\theta})-\alpha k(\theta)-(1-\alpha)k(\theta)=0,
	$$
so that $z$ is a supporting point of $\ell_\theta$, and consequently
each point on the line segment $[\zeta,\xi]$  is a supporting point of $\ell_\theta$.
\end{proof}

\begin{example}
Let $G=\{z : -1\leq \re(z)\leq 1,\ -1\leq\im(z)\leq 1\}$ be a square with center at the origin and sides (=supporting lines)
parallel to the coordinate axes at distance one from the origin. It is
easy to see that the supporting function of $G$ is
    $$
    k(\theta)=\left\{\begin{array}{rl}
    \cos\theta+\sin\theta,\ & 0\leq\theta\leq\pi/2\\
    -\cos\theta+\sin\theta,\ & \pi/2\leq\theta\leq\pi\\
    -\cos\theta-\sin\theta,\ & \pi\leq\theta\leq 3\pi/2\\
    \cos\theta-\sin\theta,\ & 3\pi/2\leq\theta\leq2\pi
    \end{array}\right.=|\cos\theta|+|\sin\theta|.
    $$
In particular, $k(\theta)=1$ for $\theta=0,\pi/2,\pi,3\pi/2$. The graph of $k(\theta)$ is an upward half wave with cusps at the points $n\pi/2$, $n\in\Z$.
\end{example}

\begin{lemma}\label{signed-distance}
We have $|k(\theta)|=\operatorname{dist}(0,\ell_\theta)$, that is,
$k(\theta)$ equals to the signed
distance from the origin to the supporting line $\ell_\theta$ with
respect to the angle $\theta$.
\end{lemma}

\begin{proof}
In general, it is known that the distance of a point $(x_0,y_0)$
from a line $ax+by+c=0$ is equal to the expression
	$$
	\frac{|ax_0+by_0+c|}{\sqrt{a^2+b^2}}.
	$$
Substituting $(x_0,y_0)=(0,0)$ and $c=-k(\theta)$, and using
$\cos^2\theta+\sin^2\theta=1$, we are done.
\end{proof}

Let $w_1,\ldots,w_n\in\C$ be $n\geq 2$ distinct points, and let $W=\{\overline{w}_1,\ldots,\overline{w}_n\}$ be a set of points, not necessarily associated with exponential polynomials. Since $W$ has only finitely many points, its convex hull $\co(W)$ is either a line segment or a convex polygon. The line through two consecutive vertex points
$\overline{w}_j,\overline{w}_{k}$ of $\co(W)$ is a supporting line of
$\co(W)$.

\begin{lemma}\label{k-max}
If $k(\theta)$ denotes the supporting function of $\co(W)$, then
    \begin{equation}\label{k}
    k(\theta)=\max_{1\leq j\leq n} \re \left(w_je^{i\theta}\right).
    \end{equation}
\end{lemma}

\begin{proof}
Since $W=\{\overline{w}_1,\ldots,\overline{w}_n\}\subset \co (W)$, it
is clear that
	$$
	\max_{1\leq j\leq n} \re \left(w_je^{i\theta}\right)
	=\max_{1\leq j\leq n} \re \left(\overline{w}_je^{-i\theta}\right)
	\leq \sup_{z\in \co(W)}\re(ze^{-i\theta})
	\overset{\textnormal{def}}{=}k(\theta).
	$$
It remains to prove the reverse inequality. To this end, we may suppose that
the points $\overline{w}_1,\ldots,\overline{w}_s$ are the vertices
(extreme points) of $\co(W)$ for some $s\leq n$. Then, if $z\in\co(W)$,
we have
	$$
	\re \left(z\right)
	\leq \max_{1\leq j\leq s} \re \left(\overline{w}_j\right),
	$$
and consequently
	$$
	\re \left(ze^{-i\theta}\right)
	\leq \max_{1\leq j\leq s} \re \left(\overline{w}_je^{-i\theta}\right).
	$$
Since $\co(W)$ is closed and bounded, it follows that
	$$
	k(\theta)\overset{\textnormal{def}}{=}
	\sup_{z\in \co(W)}\re(ze^{-i\theta})
	\leq \max_{1\leq j\leq s} \re \left(\overline{w}_je^{-i\theta}\right)
	\leq \max_{1\leq j\leq n} \re \left(w_je^{i\theta}\right).
	$$
This completes the proof.
\end{proof}

Let $f$ be an exponential polynomial of the form
	$$
	f(z)=F_1(z)e^{w_1z^q}+\cdots+F_n(z)e^{w_nz^q},
	$$
and let $W=\{\overline{w}_1,\ldots,\overline{w}_n\}$ be associated with $f$.
Moreover, let $\arg(z)=\theta^\bot$ denote the orthogonal ray for $\co(W)$ related to the consecutive
vertex points $\overline{w}_j,\overline{w}_k$, and let $k(\theta)$ be
the supporting function of $\co(W)$.
For $\theta=\theta^\bot$, the proof of \cite[Lemma~3.2]{division} yields
	\begin{equation}\label{theta-bot}
	k\left(\theta^\bot\right)=\re \left(w_je^{i\theta^\bot}\right)=\re \left(w_ke^{i\theta^\bot}\right)
	\geq \max_{l\neq j,k}\re \left(w_le^{i\theta^\bot}\right).
	\end{equation}
Thus, keeping \eqref{critical-directions} in mind, the critical rays are associated with
angles $\theta^*$ satisfying
    $$
    k(q\theta^*)=\re \left(w_le^{iq\theta^*}\right)=\re \left(w_me^{iq\theta^*}\right)
    $$
for a pair of indices $l,m$ such that $l\neq m$. This approach is chosen as the definition of critical rays in \cite{Stein}.

For the above $f$, we have $\rho(f)=q$, and it follows that the Phragm\'en-Lindel\"of indicator function of $f$ satisfies
    \begin{equation}\label{h}
    h_f(\theta)=\max_{1\leq m\leq n} \re \left(w_me^{iq\theta}\right),
    \end{equation}
see the proof of \cite[Lemma~3.3]{GOP}. Combining \eqref{k} and \eqref{h}, we obtain the identity
	\begin{equation}\label{kqh}	
	k(q\theta)=h_f(\theta).
	\end{equation}
This identity shows that the critical
rays can also be found in terms of the indicator function, namely at the
cusps $\theta^*$ of the graph of $h_f(\theta)$. This approach is chosen as the definition of critical rays in \cite{division} because it requires no knowledge
on convex sets.

The proof of Theorem~\ref{Stein-thm1} in \cite{WH} as well as the
proof of Theorem~\ref{Stein-ratio-thm} above are both based on the circumference formula, which will be stated and proved next.

\begin{theorem}\label{circumference-thm}
Let $W=\{\overline{w}_1,\ldots,\overline{w}_n\}$, and let $k(\theta)$ be the supporting function of $\co(W)$. Then the circumference
$\Ce$ of $\co(W)$ satisfies
	\begin{equation}\label{perimeter-formula}
	\Ce=\int_0^{2\pi} k(\theta)\, d\theta.
	\end{equation}
\end{theorem}

\begin{proof}
(a) Suppose that $\co(W)$ is a polygon having the origin as its interior
point. As in Section~\ref{zeros-expsum-sec}, we may suppose that the
vertex points of $\co(W)$ are $\overline{w}_1,\ldots,\overline{w}_s$,
$s\leq n$. By re-naming the vertices, if necessary, we may suppose that
they are organized in the counterclockwise direction from $1$ to
$s$ so that the corresponding orthogonal rays $\arg (z)=\theta^\perp_j$ satisfy \eqref{orthogonal-rays-ordered}, see Figure~\ref{Orthogonal-rays-fig}. It is now obvious from the proof of Lemma~\ref{k-max} that
	\begin{equation}\label{k-vertices}
	k(\theta)=\max_{1\leq j\leq s} \re \left(w_je^{i\theta}\right),
	\end{equation}
that is, the vertex points of $\co(W)$ determine $k(\theta)$.

For each $j\in\{1,\ldots,s\}$, denote $w_j=a_j+ib_j=r_je^{i\varphi_j}$,
where $\varphi_j\in[0,2\pi]$. For convenience, set $w_{s+1}=w_1$.
If $\theta=\theta^\perp$, then from
\eqref{theta-bot} and the fact that $\overline{w}_1,\ldots,\overline{w}_s$
are the vertex points of $\co(W)$, we find that
$k(\theta^\perp)$ is determined by precisely two terms in \eqref{k-vertices}, while if $\theta\neq\theta^\perp$, then $k(\theta)$ is determined
by precisely one term in \eqref{k-vertices}. Thus, for each
$j\in\{1,\ldots,s\}$, the assumption \eqref{orthogonal-rays-ordered} yields
	$$
    	k(\theta)=a_j\cos\theta+b_j\sin\theta
    	>\max_{k\neq j}\{a_k\cos\theta+b_k\sin\theta\},
    	\quad \theta\in(\theta_{j}^\bot,\theta_{j+1}^\perp),
	$$
where $\theta_{s+1}^\bot=\theta_1^\bot+2\pi$. Moreover,
    \begin{equation}\label{botbotbot}
    \begin{split}
    a_j\cos\theta_{j+1}^\bot+b_j\sin\theta_{j+1}^\bot
    &=a_{j+1}\cos\theta_{j+1}^\bot+b_{j+1}\sin\theta_{j+1}^\bot\\
    &> \max_{k\neq j,j+1}\{a_k\cos\theta_{j+1}^\bot
    +b_k\sin\theta_{j+1}^\bot\},
    \end{split}
    \end{equation}
for $j=1,\ldots,s$. Using $a_j=r_j\cos\varphi_j$, $b_j=r_j\sin\varphi_j$
and $w_{s+1}=w_{1}$, the discussion above gives raise to
    \begin{equation*}
    \begin{split}
    \int_0^{2\pi}k(\theta)\,d\theta
    &=\int_{\theta_{1}^\bot}^{\theta_{s+1}^\bot}k(\theta)\,d\theta=\sum_{j=1}^{s}\int_{\theta_{j}^\bot}^{\theta_{j+1}^\bot}k(\theta)\,d\theta\\
    &=\sum_{j=1}^{s}\int_{\theta_{j}^\bot}^{\theta_{j+1}^\bot}(a_j\cos\theta+b_j\sin\theta)\,d\theta\\
    &=\sum_{j=1}^{s}\left(r_j\sin(\theta_{j+1}^\bot-\varphi_j)-
    r_j\sin(\theta_{j}^\bot-\varphi_j) \right)\\
    &=\sum_{j=1}^{s}(r_j\sin(\theta_{j+1}^\bot-\varphi_j)-r_{j+1}\sin(\theta_{j+1}^\bot-\varphi_{j+1})).
    \end{split}
    \end{equation*}
We assume that the origin is an interior point of $\co(W)$, and hence the supporting function $k(\theta)$ is positive for all $\theta\in[0,2\pi]$. Recall that the area of the triangle $\Delta Ow_jw_{j+1}$ is
    $$
    A_{\Delta Ow_jw_{j+1}}=\frac{1}{2}
     r_jr_{j+1}\sin(\varphi_{j+1}-\varphi_j).
    $$
The height of the $w_jw_{j+1}$-side of the triangle $\Delta Ow_jw_{j+1}$ is $a_j\cos\theta_{j+1}^\bot+b_j\sin\theta_{j+1}^\bot=r_j\cos(\varphi_j-\theta_{j+1}^\bot)$ by Lemma~\ref{signed-distance}. This
height equals to $a_{j+1}\cos\theta_{j+1}^\bot+b_{j+1}\sin\theta_{j+1}^\bot=r_{j+1}\cos(\varphi_{j+1}-\theta_{j+1}^\bot)$
by \eqref{botbotbot}. As the area of any triangle is
$\frac12$ $\!*\!$ \textsl{base} $\!*\!$ \textsl{height}, the length of the base $|w_jw_{j+1}|$
of $\Delta Ow_jw_{j+1}$ is
    \begin{equation*}
    \begin{split}
    |w_jw_{j+1}|&=\frac{r_jr_{j+1}\sin(\varphi_{j+1}-\varphi_j)}
    {a_j\cos\theta_{j+1}^\bot+b_j\sin\theta_{j+1}^\bot}
    =\frac{r_jr_{j+1}\sin((\varphi_{j+1}-\theta_{j+1}^\bot)+(\theta_{j+1}^\bot-\varphi_j))}
    {a_j\cos\theta_{j+1}^\bot+b_j\sin\theta_{j+1}^\bot}\\
    &=\frac{r_jr_{j+1}\sin(\varphi_{j+1}-\theta_{j+1}^\bot)\cos(\theta_{j+1}^\bot-\varphi_j)}
    {r_j\cos(\varphi_j-\theta_{j+1}^\bot)}+
    \frac{r_jr_{j+1}\cos(\varphi_{j+1}-\theta_{j+1}^\bot)\sin(\theta_{j+1}^\bot-\varphi_j)}
    {r_{j+1}\cos(\varphi_{j+1}-\theta_{j+1}^\bot)}\\
    &=r_{j+1}\sin(\varphi_{j+1}-\theta_{j+1}^\bot)+r_j\sin(\theta_{j+1}^\bot-\varphi_j).
    \end{split}
    \end{equation*}
Therefore, the circumference of the polygon $\co(W)$ is
    \begin{equation*}
    \begin{split}
    \Ce&=\sum_{j=1}^{s}|w_jw_{j+1}|
    =\sum_{j=1}^{s}\left(r_{j+1}\sin(\varphi_{j+1}-\theta_{j+1}^\bot)
    +r_j\sin(\theta_{j+1}^\bot-\varphi_j)\right)
    =\int_0^{2\pi}k(\theta)\,d\theta,
    \end{split}
    \end{equation*}
proving the assertion \eqref{perimeter-formula} in the particular case
when $0\in\operatorname{Int}(\co(W))$.

(b) Suppose next that $\co(W)$ is a polygon but the origin is not its interior
point. Let $c\neq 0$ be any constant such that the origin is an
interior point of $\co(W_c)$, where $W_c=\{\overline{w}_1+c,\ldots,\overline{w}_n+c\}$. That is, the polygon $\co(W)$ is translated by a vector $c$ to another polygon $\co(W_c)$ having the origin as its interior point.

Let $k_c(\theta)$ denote the supporting function of $\co(W_c)$.
Without loss of generality, we may suppose that $\overline{w}_1,\ldots,\overline{w}_s$ are the vertex points of $\co(W)$. Then $\overline{w}_1+c,\ldots,\overline{w}_s+c$ are the vertex points of $\co(W_c)$. Denoting $c=a+ib$,
we find that
	$$
	k_c(\theta)=k(\theta)+a\cos\theta+b\sin\theta.
	$$
Therefore
	$$
	\int_0^{2\pi}k_c(\theta)\,d\theta
	=\int_0^{2\pi}k(\theta)\,d\theta
	+\int_0^{2\pi}(a\cos\theta+b\sin\theta)\,d\theta
    	=\int_0^{2\pi}k(\theta)\,d\theta.
	$$
The assertion \eqref{perimeter-formula} then follows from Part (a) because
$\co(W)$ and $\co(W_c)$ have the same circumference. 	

(c) Finally, suppose that $\co(W)$ is a line segment, in which case it
has two vertex points and two orthogonal rays. If the origin is on
the line segment $\co(W)$, we calculate $\int_0^{2\pi} k(\theta)\,d\theta$
over two intervals as in Part (a). If the origin is not on $\co(W)$,
we use a translation as in Part (b), and then proceed as above.
\end{proof}

\section{Zeros of some special partial sums of $\zeta(z)$}\label{zeta-appendix}

As in Example~\ref{finitesum}, we consider exponential sums
having zeros on a single line. For example, if $f(z)=1+e^{z(-\log 2)}$,
then the orthogonal rays for $\co(W)=[-\log 2,0]$ are at angles $\theta_1^\bot=\pi/2$ and $\theta_2^\bot=3\pi/2$, and the zeros of $f$ are precisely at the points $z_n=\frac{\pi}{\log 2} (2n+1)i$, $n\in\Z$, which are located on the imaginary axis, that is, on the orthogonal rays.
It is easy to see that $n(r,1/f)\sim\frac{\log 2}{\pi}r$, which matches with the result in~\eqref{zeros-ES-improved}.

In general, the $M$th partial sum
	$$
	f_M(z):=\sum_{n=1}^{M}e^{(-\log n)z},\quad M\geq 2,
	$$
of the Riemann zeta-function $\zeta(z)$ has orthogonal rays at angles $\theta_1^\bot=\pi/2$ and $\theta_2^\bot=3\pi/2$ for $\co(W)=[-\log M, 0]$,
and thus
    $$
    N\left(r,\frac{1}{f_M}\right)= \frac{\log M}{\pi}r+O(\log r).
    $$
This is in contrast to the zeros of $\zeta(z)$, for which
    $$
    N\left(r,\frac{1}{\zeta}\right)= \frac{\log r}{\pi}r+O(r),
    $$
see \cite[Theorem~1(1)]{Ye}. In addition, the location of the zeros of $f_M(z)$ is far from linear even for $M=3$, see \cite{Borwein}.
However, a special sub-sum $\sum_{\ell=0}^{L}e^{(-\log 2)\ell z}$ of $f_M(z)$ with $2^{L}\leq M$ or $L\leq \log_2 M$ has its zeros precisely on the orthogonal rays, that is, on the imaginary axis. This is easily observed by writing
	$$
	\sum_{\ell=0}^{L}e^{(-\log 2)\ell z}
	=\frac{1- e^{(-\log 2)(L+1) z}}{1-e^{(-\log 2) z}}.
	$$
Of course, this is also the case with the partial sub-sums $\sum_{\ell=0}^{L}e^{(-\log p)\ell z}=\sum_{\ell=0}^{L}p^{-\ell z}$ for any integer~$p\geq 2$.

Recall Euler's product formula
	$$
	\zeta(z)=\prod_p\left(1-\frac{1}{p^{z}}\right)^{-1}
	=\prod_p (1+p^{-z}+p^{-2z}+\cdots),
	$$
which is absolutely convergent for $\re(z)>1$, see \cite{Titchmarsh2}.
This leads us to consider partial sums of $\zeta(z)$ of the form
	$$
	\prod_{j=1}^{K}(1+p_j^{-z}+p_j^{-2z}+\cdots+p_j^{-N_jz})
	=\prod_{j=1}^{K}(1- e^{(-\log p_j)(N_j+1) z})/(1-e^{(-\log p_j) z}),
	$$
where $p_1,\ldots,p_K$ are the first $K$ prime numbers and $N_1,\ldots,N_K$  are positive integers.
For these exponential sums, we have the same results as  above with
	$$
	\co(W)=[-\max_{1\leq j\leq K}{N_j\log p_j}, 0].
	$$
For any fixed $K$ and for any choice of $\mathbf{N}_K:=(N_1, \ldots, N_K) \in \mathbb{N}^K$, the above partial sub-sum can be written as
	\begin{equation}\label{thin-sum}
	\sum_{\mathbf{j}\in\mathbf{N}_K} \frac{1}{{n_{\mathbf{j}}}^z},
	\end{equation}
where $n_{\mathbf{j}}=\mathbf{p}^{\mathbf{j}}=\prod_{i=1}^K p_i^{j_i}$ with  multi-indices $\mathbf{j}=(j_1, \ldots, j_K)$ and $\mathbf{p}=(p_1, \ldots, p_K)$. This function is entire and it has all of its zeros precisely on the imaginary axis.

\begin{remark}
Let $f(n):\mathbb{N}\to\mathbb{C}$ be an arithmetic function and consider the Dirichlet series $\sum_{n=1}^{\infty} f(n)n^{-s}$ with coefficients $f(n)$. When $f(n)$ is completely multiplicative, that is, $f(mn)=f(m)f(n)$ holds for any $n, m\in\mathbb{N}$, the Euler product formula is obtained as in \cite[Theorem~11.7]{Apo}: {\it Assuming $\sum f(n)n^{-s}$ converges absolutely for $\re(s)>\sigma_a$, then we have}
	$$
	\sum_{n=1}^{\infty} \frac{f(n)}{n^s}
	=\prod_p \frac{1}{1-f(p)p^{-s}},
	\quad \re(s)>\sigma_a.
	$$
In the product, $p$ runs over all the prime numbers but $f(p)=0$.
Thus let us again consider special partial sums, this time for Dirichlet series, which are given by the corresponding finite products
	$$
	F_{M,m}(s):=\prod_{p\leq M} \frac{1-f(p)^m p^{-ms}}{1-f(p)p^{-s}}
	$$
for $m, M\in \mathbb{Z}_{\geq 2}$.
Of course, as $m$, $n\to\infty$, $F_{M,n}(s)$ converges absolutely to $\sum_{n=1}^{\infty} \frac{f(n)}{n^s}$ for $\re(s)>\sigma_a$.
A zero $z$ of $F_{M,m}(s)$ occurs if and only if $p^{mz}=f(p)^m$ holds for some prime $p$ and positive integer $m$, that is,
	$$
	z=\frac{\log |f(p)|}{\log p}+i\left\{\frac{2k\pi}{m\log p}
	+\frac{\arg (f(p))}{\log p}\right\}.
	$$
We look for a function $f(n)$ keeping $\frac{\log |f(p)|}{\log p}$ invariant, say $\tau\in\mathbb{R}$, for any prime number~$p$ with $f(p)\neq 0$.
Then we have $|f(p)|=p^{\tau}$, unless $f(p)=0$, and then the functions $F_{M,m}(s)$ have all the zeros on the line $\re(s)=\tau$.
Of course, when $f(n)=n^{\tau}$ for any $n\in\mathbb{N}$, our Dirichlet series reduces simply to
	$$
	\sum_{n=1}^{\infty} \frac{1}{n^{s-\tau}}=\zeta(s-\tau)
	$$
for $\re(s)>\re(\tau)+1$.
In particular, taking $f(n)=\chi(n)$ as the Dirichlet characters $\mod k\geq 1$ as in~\cite{Apo}, one has the partial products of the Dirichlet $L$-function $L(s,\chi)$ for $\re(s)>1$, whose zeros are all on the imaginary axis, that is, $\tau=0$, since it is noted in~\cite[Note, p.~136]{Apo} that $|\xi(n)|=1$ holds for any character $\xi$ whenever the value of $\xi(n)$ is different from zero.
\end{remark}

It should be emphasized that the sum in \eqref{thin-sum} does not coincide
with the ordinary partial sum
	$$
	\sum_{1\leq n\leq M}\frac{1}{n^z}
	\quad \textnormal{with} \quad
	M:=n_{\mathbf{N}_K}=\mathbf{p}^{\mathbf{N}_K}=\prod_{j=1}^K p_j^{N_j}.
	$$
Indeed, this is the $M$-truncation of the infinite sum $\sum_{1\leq n\leq \infty}\frac{1}{n^z}$, while the sum in \eqref{thin-sum} is a ``thin out'' of this partial sum.

\begin{example}
Consider the function
	\begin{eqnarray*}
	\pi_{24}(z) &:=&
	 (1+2^{-z}+2^{-2z}+2^{-3z})(1+3^{-z})\\ &=&
	1+\frac{1}{2^z} + \frac{1}{3^z} + \frac{1}{4^z} + 	
	\frac{1}{6^z} + \frac{1}{8^z} + \frac{1}{9^z} + \frac{1}{12^z} +
	\frac{1}{16^z} + \frac{1}{18^z} + \frac{1}{24^z},
	\end{eqnarray*}
whose terms are all included in the partial sub-sum \eqref{thin-sum}
with $K=2$ and with $N_1$ and $N_2$ both being at least~$3$.	
It is known that the zeros of $\pi_{24}(z)$ reduced modulo $\frac{2\pi i}{\log(24)}$ lie on a non-linear curve in the plane known as Alexa's
penguin  \cite[Example~4.8]{Borwein}. The partial sub-sum \eqref{thin-sum}
converges absolutely to $\zeta(z)$ in $\re(z)>1$, but the function above
shows that it is not trivial that such a ``thin out'' forces the zeros
onto a straight line.
\end{example}

Recall \emph{Riemann's xi-function} $\Xi(z)=\zeta\left(\frac12+iz\right)$ satisfying
the functional equation $\Xi(-z)=\Xi(z)$. Thus $\Xi$ is an even function,
and the critical line $\re(z)=\frac12$ of the zeta-function is translated onto
the real axis.

Define $g_j(z):=1+p_j^{-z}+p_j^{-2z}+\cdots+p_j^{-N_jz}$, for which
	$$
	g_j(-z)=1+p_j^{z}+p_j^{2z}+\cdots+p_j^{N_jz}
	=p_j^{N_jz}g_j(z), \quad 1\leq j\leq K.
	$$
Therefore our partial sub-sum $G(z):=\prod_{j=1}^{K}g_j(z)$ satisfies
	$$
	G(-z)=\prod_{j=1}^{K}p_j^{N_jz}g_j(z)
	=G(z)\prod_{j=1}^{K}p_j^{N_jz},
	$$
that is, it has the symmetric relation about the imaginary axis:
	$$
	G(-z)\prod_{j=1}^{K}p_j^{-N_jz/2}
	=G(z)\prod_{j=1}^{K}p_j^{N_jz/2}.
	$$
This is, in fact, trivial, since
	$$
	\Xi_0(z):=\prod_{j=1}^{K}\sum_{k=0}^{N_j}p_j^{(N_j/2-k)z}
	=\prod_{j=1}^{K}\sum_{k=0}^{N_j}p_j^{(k-N_j/2)z}=\Xi_0(-z).
	$$
Why this form of a functional equation is interesting?
Because of a functional equation in Riemann's $\zeta(s)$ (see \cite[Theorem~2.1]{Titchmarsh2}), or its symmetric one (see \cite[(2.6.4), p.~22]{Titchmarsh2}),
	$$
	\pi^{-s/2}\Gamma\left(\frac{s}{2}\right)\zeta(s)
	=\pi^{-(1-s)/2}\Gamma\left(\frac{1-s}{2}\right)\zeta(1-s)
	$$
can be written as $\xi(z)=\xi(-z)$ by taking
	$$
	\xi(z):=\frac{1}{\pi^{z/2}}\Gamma\left(\frac{z}{2}
	+\frac{1}{4}\right) \zeta
	\left(z+\frac{1}{2}\right), \quad z=s-\frac{1}{2}.
	$$
It is known that $\zeta(z)$ can be analytically continued to a meromorphic function in the whole complex plane.
In fact, it is better to exchange the $\Gamma-$factor in   $\xi(z)$ and $\xi(-z)$ respectively to take
	$$
	\Xi(z)=\dfrac{\zeta \left(z+\frac{1}{2}\right)}{\pi^{z/2}
	\Gamma\left(-\frac{z}{2}+\frac{1}{4}\right)}
	$$
in order to cancel the unique (simple) pole of $\zeta(\pm z+1/2)$ at $z=\pm 1/2$ by the (simple) pole of $\Gamma(\mp z/2+1/4)$, respectively.
Then we are to expect the non-trivial zeros of this ``even'' and ``real-on-real'' entire function $\Xi(z)$ locate right on the imaginary axis, as we see exactly for  $\Xi_0(z)$.
Here, by the trivial zeros of $\Xi(z)$, we mean the sequence $z_m=2m+1/2$, $m=1,2,\ldots$, which is the set of all non-zero poles of $\Gamma\left(-\frac{z}{2}+\frac{1}{4}\right)$.

Recall that the Nevanlinna characteristic functions, see Appendix~A above, of $\zeta(z)$ and $\Gamma(z)$ are both comparable to $(1/\pi)r\log r$.
For this fact, we cite Theorem~1 again as well as Theorem~2 (given in Appendix) of  the paper~\cite{Ye} by Z. Ye.
Note that $\zeta(z)$ has only one pole at $=1$ and has no finite deficient values, while $\Gamma(z)$ omits zero and has the unique deficient value $\infty$, respectively.

Recall further that the even-ness of a complex function means the symmetry of a rotation of $\pi$ about the origin, not a reflection
with respect to the imaginary axis.
This could make us have a feeling of enigma for the standing of those zeros on the imaginary axis, while we have already seen this phenomenon through $\Xi_0(z)$.
The equation $\Xi(z)=\Xi(-z)$ is also found in \cite[(2.1.15), p.~16]{Titchmarsh2}.


\section{Complex linear differential equations}\label{DE-appendix}

It is well-known that the solutions of
    \begin{equation}\label{lde-appendix}
      f^{(n)}+A_{n-1}(z)f^{(n-1)}+\cdots +A_1(z)f'+A_0(z)f=0
     \end{equation}
are entire functions assuming that the coefficients $A_0(z),\ldots,A_{n-1}(z)$ are entire \cite{Laine1}. In this appendix we will shed some background
on results regarding \eqref{lde-appendix} which could be useful for the
readers of this paper.

We begin by reviewing some results from \cite{GSW} regarding the growth of solutions of \eqref{lde-appendix} when the coefficients
are polynomials. To this end, we need some definitions.
Set $d_j=\deg (A_j)$ if $A_j\not\equiv0$
and $d_j=-\infty$ if $A_j\equiv0$, $0\leq j\leq n-1$.
We define a strictly decreasing finite sequence of non-negative integers
$s_1>s_2>\cdots> s_p\geq 0$ in the following manner.
We choose $s_1$ to be the unique integer satisfying
	$$
	\frac{d_{s_1}}{n-s_1}=\max_{0\leq k\leq n-1} \frac{d_k}{n-k}
	\quad\textnormal{and}\quad
	\frac{d_{s_1}}{n-s_1}>\frac{d_k}{n-k}\ \textnormal{for all}\ 0\leq k\leq s_1.
	$$
Then given $s_j$, $j\geq1$, we define $s_{j+1}$ to be the unique integer satisfying
	$$
	\frac{d_{s_{j+1}}-d_{s_j}}{s_{j+1}-s_{j}}
	=\max_{0\leq k\leq s_j} \frac{d_k-d_{s_{j}}}{s_j-k}>-1
	$$
and
	$$
	\frac{d_{s_{j+1}}-d_{s_j}}{s_{j+1}-s_{j}}
	>\frac{d_k-d_{s_{j}}}{s_j-k}>-1\ \textnormal{for all}\ 0\leq k<s_{j+1}.
	$$
For a certain $p<n$, the integer $s_p$ will exist, but the integer $s_{p+1}$ will not exist, and then the sequence $s_1, s_2, \ldots, s_p$ terminates with $s_p$.
Further, define
	$$
	\alpha_j=1+\frac{d_{s_{j+1}}-d_{s_j}}{s_{j+1}-s_{j}},\quad j=1, 2, \ldots, p,
	$$
where we set $s_0=n$ and $d_{s_0}=d_n=0$.

\begin{theorem}
\textnormal{(Gundersen-Steinbart-Wang \cite{GSW})}
\begin{itemize}
\item[\textnormal{(a)}]
If $f$ is a transcendental solution of \eqref{lde-appendix}, then $\rho(f)=\alpha_j$ for some $j$, $1\leq j \leq p$.
\item[\textnormal{(b)}]
For any $j=1, 2, \dots, p$, there can exist at most $s_j$ linearly independent solutions $f$ of \eqref{lde-appendix} satisfying $\rho(f)<\alpha_j$.
\item[\textnormal{(c)}]
There can exist at most $s_1+1$ distinct orders of transcendental solutions of
\eqref{lde-appendix}.
\item[\textnormal{(d)}]
There can be at most $s_p$ linearly independent polynomial solutions
of \eqref{lde-appendix}.
\item[\textnormal{(e)}]
Every transcendental solution $f$ of \eqref{lde-appendix} satisfies $\rho(f)\geq\frac{1}{n-1}.$
\item[\textnormal{(f)}]
If $\{f_1,\ldots,f_n\}$ is any fundamental set of solutions of \eqref{lde-appendix}, then $\sum_{j=1}^n\rho(f_j)\geq n+d_0.$
\end{itemize}
\end{theorem}

Prior to \cite{GSW},
Wittich~\cite{Wittich1952, Wittich1968} had considered \eqref{lde-appendix} with polynomial coefficients using the method of Frobenius and the Newton-Puiseux diagram. The constants
$\alpha_j$ may be regarded as the slopes of
the corresponding Newton-Puiseux diagram.

Next we focus on a special case of \eqref{lde-appendix}, namely on linear differential equations of second order
	\begin{equation}\label{AC.01}
	f''+A(z)f=0,
	\end{equation}
where $A(z)$ is either an entire or a meromorphic function.
If $A(z)$ is entire, then \eqref{AC.01} possesses two linearly independent entire solutions. If $A(z)$ has poles, the poles need to be of multiplicity two and
the Laurent expansion of $A(z)$ at its poles need to be of a specific form in order for
meromorphic solutions to exist \cite[Theorem~6.7]{Laine}.

In the 1980's, complex oscillation theory for \eqref{AC.01} was actively studied.
In particular, the questions (Q1'), (Q2') and (Q3')~from Section~\ref{EPO-sec} were
considered. In these discussions, the Bank--Laine identity
	$$
	4A(z)=\left(\frac{E'}{E}\right)^2-\left(\frac{c}{E}\right)^2-2\frac{E''}{E}
	$$
plays an important role~\cite{Bank-Laine}. Here $E=f_1f_2$ is a product of two linearly independent solutions $f_1,f_2$ and
the Wronskian determinant $c=W(f_1,f_2)\ne0$ is a constant. The Bank-Laine identity has  led to the definitions of Bank-Laine functions and Bank-Laine sequences~\cite{Alotaibi2009, Elzaidi1999, Langley2020, LM2009},
and is also used in finding prescribed zeros for solutions \cite{Shen2}.

Bank and Laine \cite{Bank-Laine} have proved that if $\rho(A)$ is not an integer, then $\lambda(E)\geq \rho(A)$, and if $\rho(A)<1/2$,
then $\lambda(E)=\infty$.
Rossi~\cite{Rossi} and Shen \cite{Shen} proved independently that $\rho(A)=1/2$ implies $\lambda(E)=\infty$. Based on these results, Bank and Laine conjectured that whenever $\rho(A)$ is not an integer, then $\lambda(E)=\infty$~\cite{Laine, LT}. Recently, this conjecture was solved in the negative by Bergweiler and Eremenko, who constructed counterexamples for the case $\rho(A)>1$ in \cite{BE0} and for the case $1/2<\rho(A)<1$ in \cite{BE}.

\begin{theorem}
\textnormal{(Bank-Laine-Langley \cite{BLL1})}
Let $A(z)$ be a transcendental entire function of finite order satisfying the
following property: There exists a subset $H\subset \mathbb{R}$ of measure zero such that for each $\theta\in \mathbb{R}\setminus H$, one of the following conditions holds:
\begin{itemize}
\item[{\rm(a)}] $r^{-N}|A(re^{i\theta})|\to\infty$ as $r\to\infty$, for each $N>0$,
\item[{\rm(b)}] $\int_0^\infty r|A(re^{i\theta})|dr<\infty$,
\item[{\rm(c)}] there exist $K>0$, $b>0$ and  $n\geq 0$, all possibly depending on $\theta$, such that $(n+2)/2<\rho(A)$ and $|A(re^{i\theta})|\leq Kr^n$ for all $r\geq b$.
\end{itemize}
Let $E$ be the product of two linearly independent solutions of \eqref{AC.01}. Then $\lambda(E)=\infty$.
\end{theorem}

The questions (Q1), (Q2) and (Q3)~from Section~\ref{EPO-sec} have also been of
interest. In~\cite{Bank-Laine}, it was showed that if $\lambda(A)<\rho(A)$, then  any solution $f\not\equiv 0$ of \eqref{AC.01} satisfies
$\lambda(f)\geq\rho(A)$. If $\lambda(A)<\rho(A)=\infty$, then the previous conclusion
can be strengthened to $\lambda(f)=\infty$, while the general case $\lambda(A)<\rho(A)<\infty$ is still in doubt \cite{Gundersen2019}. In the special case $A(z)=e^{P(z)}$, where $P(z)$ is a non-constant polynomial, the conclusion is $\lambda(f)=\infty$ \cite{BLL1}.
For further recommended reading in this direction, see \cite{Edmund, CL, Ishizaki-Tohge}.

Results on the case when $A(z)$ is a rational function in $e^z$ can be found in
\cite{CY, Shimomura2002}. These results are motivated by an earlier research
on the periodicity of the solutions.

\begin{theorem}
\textnormal{(Bank-Laine \cite{Bank-LaineJRAM})}
Suppose that $f$ is a non-trivial solution of \eqref{AC.01} such that $\lambda(f)<\infty$, and that $A(z)$ is a non-constant function rational in $e^z$.
Then the following assertions hold:
\begin{itemize}
\item[{\rm(a)}]  If $f(z)$ and $f(z+2\pi i)$ are linearly dependent, then $f(z)$ has the representation
$$
f(z)=\psi(e^z)\exp\left(\sum_{j=\mu}^\nu d_je^{jz}+dz\right),
$$
where $\psi$ is a polynomial, $\mu$ and $\nu$ are integers with $\mu\leq \nu$, and $d, d_\mu, \ldots, d_\nu$ are constants with $d_j\ne0$ for some $j$.
\item[{\rm(b)}] If $f(z)$ and $f(z+2\pi i)$ are linearly independent, then the product $E(z)=f(z)f(z+2\pi i)$ satisfies $E(z)^2=\phi(e^z)$, where $\phi(\rho)$ is rational and analytic in $0<|\rho|<\infty$.
\end{itemize}
\end{theorem}

In closing, we shortly recall some results when $A(z)$ admits poles. In this case \eqref{AC.01} does not always possess meromorphic solutions.

\begin{theorem}
\textnormal{(Bank-Laine \cite{Bank-Laine})}
Assume that \eqref{AC.01}
possesses two linearly independent meromorphic solutions $f_1,f_2$, each of finite order of growth. Then $g=f_1/f_2$ is a non-constant meromorphic function of finite order such that
\begin{itemize}
\item[{\rm(a)}] all poles of $g$ are of odd multiplicity,
\item[{\rm(b)}] all zeros of $g'$ are of even multiplicity,
\item[{\rm(c)}] $2A(z)=S_g(z)$, where $S_g$ is Schwarzian derivative of $g$.
\end{itemize}
Conversely, suppose that $g$ is a non-constant meromorphic function of finite order, satisfying \textnormal{(a)} and \textnormal{(b)} above. If $A(z)$ is defined as in \textnormal{(c)}, then \eqref{AC.01} possesses two linearly independent meromorphic solutions $f_1, f_2$, such that $g=f_1/f_2$.
\end{theorem}

A sufficient condition for the existence of meromorphic solutions of \eqref{AC.01} is described in~\cite{Laine, LS}. Applications of this result can be found in~\cite{ILST1,ILST2,Shimomura2005}.


\medskip
\noindent
\emph{J.~Heittokangas}\\
\textsc{University of Eastern Finland, Department of Physics and Mathematics,
P.O.~Box 111, 80101 Joensuu, Finland}\\
\texttt{email:janne.heittokangas@uef.fi}

\medskip
\noindent
\emph{K.~Ishizaki}\\
\textsc{The Open University of Japan, Faculty of Liberal Arts, Mihama-ku, Chiba, Japan}\\
\texttt{email:ishizaki@ouj.ac.jp}

\medskip
\noindent
\emph{K.~Tohge}\\
\textsc{Kanazawa University, College of Science and Engineering, Kakuma-machi, Kanazawa 920-1192, Japan}\\
\texttt{email:tohge@se.kanazawa-u.ac.jp}

\medskip
\noindent
\emph{Z.-T.~Wen}\\
\textsc{Shantou University, Department of Mathematics, Daxue Road No.~243, Shantou 515063, China}\\
\texttt{e-mail:zhtwen@stu.edu.cn}

\end{document}